\documentclass[12pt]{amsart}
\usepackage{color}
\usepackage{graphicx}
\usepackage{amssymb}
\usepackage{amsmath}
\usepackage{amsthm}
\usepackage{amscd}
\usepackage{epsfig}
\usepackage{tikz-cd}
\usepackage{cite}
\usepackage{quiver}
\usepackage{slashed}
\usepackage{appendix}
\usepackage[hidelinks]{hyperref}
\usepackage[capitalise]{cleveref}
\usepackage{verbatim}
\usepackage{fullpage}
\usepackage{mathrsfs}

\hypersetup{colorlinks=true,linkcolor=black,citecolor=black,urlcolor=black}

\newtheorem{theorem}{Theorem}[section]

\newtheorem{corollary}[theorem]{Corollary}
\newtheorem{lemma}[theorem]{Lemma}

\newtheorem{proposition}[theorem]{Proposition}

\newtheorem{remark}[theorem]{Remark}
\theoremstyle{definition}
\newtheorem{definition}[theorem]{Definition}
\numberwithin{equation}{section}

\newcommand{\norm}[1]{\left\Vert#1\right\Vert}
\newcommand{\abs}[1]{\left\vert#1\right\vert}

\newcommand{\pf}{\mathrm{pf}}
\newcommand{\PF}{\mathrm{PF}}
\newcommand{\dt}{\mathrm{det}}
\newcommand{\LX}{LX}
\newcommand{\LLX}{L^2X}
\newcommand{\R}{\mathbb R}

\newcommand{\Z}{\mathbb Z}

\newcommand{\N}{\mathbb N}
\newcommand{\CC}{\mathbb C}

\newcommand{\cir}{\mathbb{S}^1}
\newcommand{\tr}{\mathrm{Tr}}
\newcommand{\Tr}{\mathrm{Tr}}

\newcommand{\g}{\mathfrak{g}}
\newcommand{\disc}{\CC_{< 1}^*}

\newcommand{\dd}{\mathfrak{d}}
\newcommand{\cc}{\mathfrak{c}}
\newcommand{\CCS}{\mathrm{CS}}

\newcommand{\be}{\begin{equation}}
\newcommand{\ee}{\end{equation}}

\newcommand{\per}{\mathscr{PER}}

\newcommand{\ie}{\textit{i}.\textit{e}., }

\newcommand{\hol}{hol}

\usepackage[hidelinks]{hyperref}
\makeatletter

\newcommand{\Rmnum}[1]{\expandafter\@slowromancap\romannumeral #1@}
\makeatother
\title{Elliptic Chern Characters and Elliptic Atiyah--Witten Formula}
\author{\textnormal{Geyang Dai}}
\address{Department of Mathematics, National University of Singapore, Singapore 119076}
\email{e0983446@u.nus.edu}

\author{\textnormal{Fei Han}}
\address{Department of Mathematics, National University of Singapore, Singapore 119076}
\email{mathanf@nus.edu.sg}

\begin{document}

\begin{abstract}

Let $G$ be a compact, connected, and simply connected Lie group. 
A principal $G$-bundle over a manifold $X$, equipped with a connection, together with a positive-energy representation of the loop group $LG$, 
gives rise to a circle-equivariant gerbe module on the free loop space $LX$.
From this data we construct the \emph{elliptic Chern character} on $LX$, and a refinement, the \emph{elliptic Bismut--Chern character}, on the double loop space $L^2X$. 

Generalizing the classical Atiyah--Witten formula from the free loop space $LX$ to the double loop space $\LLX$, 
we establish an \emph{elliptic Atiyah--Witten formula}. 
The \emph{elliptic holonomy} on $\LLX$ is defined by $\tau$-deformed equivariant twisted parallel transport on $\LX$.
We show that the four Pfaffian sections, corresponding to the four spin structures on an elliptic curve, are identified with the four \emph{elliptic holonomies} arising from the four virtual level-one positive-energy representations when $G=\mathrm{Spin}(2n)$.
These constructions are intimately connected to the moduli of $G_{\CC}$-bundles over elliptic curves and conformal blocks in the context of Chern--Simons gauge theory.

\end{abstract}

\maketitle
\tableofcontents
\section{Introduction}
%----------------------------------------------------------------------------------------------------------------------------------------------------------------
\subsection{Background and Motivation}

The loop-space approach to the Atiyah--Singer index theorem is based on two
main ingredients: the \textbf{Atiyah--Witten formula}
\cite{AST_1985__131__43_0} and the \textbf{Bismut--Chern character}
\cite{bismut_index_1985}. Together, they relate the classical index formula
to geometric structures on the free loop space.

Let $(X,g)$ be a $2n$-dimensional Riemannian manifold, and write $LX$ for
its free loop space. Assume that $X$ is spin. The Atiyah--Witten formula
then says that for each loop $\gamma\in LX$ one has
\begin{equation}
  \tr_s\!\bigl(\hol_{\mathcal{S}_X}(\gamma)\bigr)
  \;=\;
  \pf_{\zeta}\!\bigl(\nabla_{\dot{\gamma}}\bigr),
\end{equation}
where $\tr_s(\hol_{\mathcal{S}_X}(\gamma))$ is the supertrace of the
holonomy in the $\mathbb{Z}_2$-graded spinor bundle $\mathcal{S}_X$, and
$\pf_{\zeta}(\nabla_{\dot{\gamma}})$ is the $\zeta$-regularized Pfaffian of
the covariant derivative along the loop. This identity provides the link between the loop space expression of the
local index density associated with the Dirac operator on $X$ and the
localization procedure on $LX$.

The second ingredient is the Bismut--Chern character $BCh_V$, an
equivariantly closed differential form on $LX$ associated with a vector
bundle with connection $(V,\nabla^V)$ over $X$. It is a
loop-space refinement of the ordinary Chern character form $Ch_V$ on
$X$. When $BCh_V$ is paired with the loop-space index density and one
applies localization formally on $LX$, the result is the index of the Dirac
operator twisted by $V$. A systematic account of the
Hamiltonian--Lagrangian correspondence underlying these constructions can
be found in~\cite{Bismut2011DuistermaatHeckmanFA}.

\medskip

In \cite{Witten1987,Witten1988TheIO}, Witten explained how the
Hamiltonian--Lagrangian correspondence extends naturally from a
finite-dimensional manifold $X$ to its free loop space $LX$.
Formally, the $S^1$-equivariant index of the Dirac--Ramond operator on
$LX$ coincides with the partition function of the $(2|1)$-supersymmetric
sigma model with target $X$. Upon localization to $X$, both sides recover
the Witten genus.

A rigorous construction of the Dirac--Ramond operator requires, as a
first step, the existence of a spinor bundle over $LX$, which in turn
necessitates that $X$ admit a string structure. This program was
proposed by Stolz and Teichner \cite{Stolz2005TheSB}. In
\cite{Kristel2020SmoothFB}, the spinor bundle
$\mathcal{S}_{LX} \to LX$ is constructed rigorously as a smooth Fock
bundle, making essential use of rigged Hilbert space techniques. When
$X$ does not admit a string structure, one is led instead to a loop
spinor gerbe module over the loop-spin lifting gerbe on $LX$.

Motivated by the Atiyah--Witten formula on $LX$, it is
natural to seek an analogue on the double loop space $\LLX$, which we
shall call the \textbf{elliptic Atiyah--Witten formula}. 
One expects the existence of the \textbf{elliptic holonomy} associated to
the loop spinor gerbe module, playing the role of an elliptic analogue of
the supertrace of the holonomy in the spinor bundle. On the analytic side, one is again led to Pfaffians of Dirac
operators, now defined on the elliptic curves, coupled with vector potentials parametrized by $\LLX$.
The
elliptic holonomy, viewed as a section of the transgression line bundle
associated to the lifting gerbe, is expected to agree with this
Pfaffian, which arises canonically as a section of the Pfaffian line
bundle \cite{Freed1987OnDL,Bunke2009StringSA} over $\LLX$.
One of the aims of this paper is to establish this elliptic
Atiyah--Witten formula.

In close analogy with the finite-dimensional situation, where a vector
bundle is used to twist the spin complex on a manifold $X$, a principal
$G$-bundle $P$ together with a positive-energy representation $\pi$ determines a natural twisting of the spin complex on
 $\LX$. Two gerbe modules on $LX$ now enter the picture: the loop spinor gerbe module
$\mathcal{S}_{LX}$, and the gerbe module $\mathcal{E}$ determined by the principal
$G$-bundle together with the chosen positive-energy representation $\pi$ of the loop group.
When the corresponding twisting gerbes cancel, equivalently, when there
is anomaly cancellation, the coupling of these two gerbe modules would give
well-defined twisted spin complex on $LX$. Formally, upon localization
to $X$, the $S^1$-equivariant index of the resulting twisted
Dirac--Ramond operator computes the twisted Witten genus 
\cite{Brylinski1990RepresentationsOL,Liu1995OnMI}.

From the Lagrangian perspective, and guided by the role of the
Bismut--Chern character on the free loop space $LX$, it is natural to
expect, for the gerbe module  $\mathcal{E}$ on $LX$, a corresponding Bismut--Chern
character at the level of the double loop space $L^2X$.
We refer to this refinement as the \textbf{elliptic Bismut--Chern
character}.  This notion was also proposed in
\cite{BerwickEvans2019SupersymmetricLM}, from the perspective of the
Stolz--Teichner program \cite{StolzTeichner2004Elliptic}, which explores the connection between
supersymmetric field theories and elliptic cohomology and tmf
\cite{Hopkins2002ATM}.  We will investigate the
elliptic Bismut--Chern character along these lines in future.

When restricted back to $X$, the elliptic Bismut--Chern character
recovers precisely the twisting term appearing in the twisted Witten
genus. A further aim of this paper is to construct and
investigate the various Chern characters associated to the data
$(P,\pi)$, on $X$, on the loop space $LX$, and on the double loop space
$L^2X$.

\medskip

We now outline our plan in greater detail. 

We begin with the geometry of lifting gerbes on the loop space $LX$, together with gerbe modules arising from positive-energy representations. For simplicity, we assume $G$ is a compact, connected, simple, and simply connected Lie group.
Let $P\to X$ be a principal $G$-bundle equipped with a connection
$A$, and consider the associated loop bundle $LP\to LX$. The lifting
gerbe of $LP$ to the basic central extension of the loop group is denoted
by 
\be
(\mathcal{G}_P,\nabla^B)\;\longrightarrow\; LX.
\ee
This gerbe carries a natural $\cir$-equivariant structure with vanishing
moment map (see Subsection \ref{subsec:geometry of lifting gerbe and cir-equivariance}).

Let $\per^k(LG)$ denote the representation ring of level-$k$
positive-energy representations of the centrally extended loop group
$\widetilde{LG}$. For each $\mathcal{H}\in \per^k(LG)$, there exists an
$\cir$-equivariant gerbe module with connection
$(\mathcal{E},\nabla^{\mathcal{E}})$ over the $k$-fold tensor power
$(\mathcal{G}_P,\nabla^B)^{\otimes k}$. 

When the principal $G$-bundle $P$
admits a string-$G$ structure, this is equivalent to saying that the principal bundle $LP\to LX$ can be lifted to a principal
$\widetilde{LG}$-bundle. In this case,
$\mathcal{E}$ becomes an infinite-dimensional vector bundle on $LX$, as in
\cite{Brylinski1990RepresentationsOL}.

In the first part of this paper, we focus on the free loop space $LX$ and
construct Chern characters associated to the gerbe module
$(\mathcal{E},\nabla^{\mathcal{E}})$. Precisely, we introduce two
characters on $LX$: the \textbf{$q$-graded Bismut--Chern character} and
the \textbf{elliptic Chern character}. The latter is defined as a
$\disc$-extension of the $\cir$-equivariant twisted Chern character
associated to the gerbe module $(\mathcal{E},\nabla^{\mathcal{E}})$, where $\disc$ is the punctured open unit disk.  

In the second part, we turn to the geometry of the double loop space
$\LLX$. We introduce the \textbf{elliptic holonomy} associated to the
gerbe module $(\mathcal{E},\nabla^{\mathcal{E}})$ and establish an
\textbf{elliptic Atiyah--Witten formula} in the case
$G=\mathrm{Spin}(2n)$ at level one. We further construct the
\textbf{elliptic Bismut--Chern character} on $\LLX$ associated to
$(\mathcal{E},\nabla^{\mathcal{E}})$. One finds that the two Chern
characters constructed on $LX$ arise naturally as the restrictions of
this elliptic Bismut--Chern character on $\LLX$ along the two loop
directions.

In developing the elliptic Atiyah--Witten theory, we observe that the
elliptic holonomy admits an alternative description via
\emph{$2$-transgression}, directly from $X$ to $\LLX$, which extends
naturally to mapping spaces of higher genus. This $2$-transgression
approach is based on the quantization commutes with
reduction in Chern--Simons gauge theory (see
\cite{Preparation}). It would be interesting to
compare this double loop space perspective with loop space formulations
arising from the conformal field theory of the Wess--Zumino--Witten model
\cite{Distler2007HeteroticCW}. From a physical perspective, the
$\LLX$/$LX$ descriptions fit naturally into the framework of the
Chern--Simons/Wess--Zumino--Witten correspondence.

%----------------------------------------------------------------------------------------------------------------------------------------------------------------
\subsection{Main Results}

We now describe the main results of this paper.

\subsubsection{Equivariant twisted Bismut--Chern character}\label{intro:eq twisted bch}

In order to prepare a geometric framework for the elliptic Chern
character on $\LX$ and the elliptic Bismut--Chern character on $\LLX$,
we first develop, in Section~\ref{sec:twisted bismut}, a general theory
of \emph{equivariant twisted Chern characters} and
\emph{equivariant twisted Bismut--Chern characters}.

Building on \cite{Han2014ExoticTE}, we construct the
$\cir$-equivariant twisted Bismut--Chern character as a loop-space
refinement of the $\cir$-equivariant twisted Chern character introduced
in \cite{Mathai2002ChernCI}. More precisely, let $M$ be a
$\cir$-manifold, and let $(\mathcal{G},\nabla^B)$ be a
$\cir$-equivariant gerbe with connection and vanishing moment map (for
general equivariant gerbes, see
\cite{Meinrenken2002TheBG}). We construct the
$\cir$-equivariant twisted Bismut--Chern character
$BCh_{\cir,\mathcal{G}}$ on $LM$ as a loop-space refinement of the
$\cir$-equivariant twisted Chern character
$Ch_{\cir,\mathcal{G}}$.

Moreover, if the gerbe $\mathcal{G}$ is trivial when restricted to the
$\cir$-fixed point set $Y=M^{\cir}$, we obtain the following commutative
diagram:
\begin{equation}\label{intro: diagram of M}
\begin{tikzcd}
    & h_{S^1\times \cir}^{2*}(LM,(\mathcal{L}^B,\nabla^{\mathcal{L}^B},\overline{H}))
    \arrow[ld, "Lj^*"'] \arrow[dr,"i^*"] & \\
h_{S^1}^{2*}(LY)(u)
    \arrow[dr,"i_{\cir}",swap]
& K^0_{\cir}(M,\mathcal{G})
    \arrow[dashed,l, "BCh|_{Y}",swap]
    \arrow[dashed,u, "BCh_{\cir, \mathcal{G}}"]
    \arrow[dashed,d, "Ch|_{Y}"']
    \arrow[dashed,r, "Ch_{\cir,\mathcal{G}}"]
& h_{\cir}^{2*}(M,H)(v)
    \arrow[ld, "j^*"'swap] \\
    & h^{2*}(Y)(u,v) &
\end{tikzcd}
\end{equation}
We briefly explain notations. $H\in \Omega^3(M)^{\cir}$ is
the curving $3$-form of the gerbe $(\mathcal{G},\nabla^B)$.
$K^0_{\cir}(M,\mathcal{G})$ is 
$\cir$-equivariant twisted $K$-theory.
$h_{S^1}^{*}(LY)$ denotes the completed $S^1$-equivariant
cohomology of the loop space $LY$.
 $h_{\cir}^{*}(M,H)$ denotes the $\cir$-equivariant cohomology
of $M$ twisted by the $3$-form $H$.
Finally,
$h_{S^1\times \cir}^{*}(LM,(\mathcal{L}^B,\nabla^{\mathcal{L}^B},\overline{H}))$
denotes the exotic $(S^1\times \cir)$-equivariant cohomology of $LM$, constructed from $(S^1\times \cir)$-invariant differential forms over $LM$ twisted by
the transgression line bundle $(\mathcal{L}^B,\nabla^{\mathcal{L}^B})$
associated to the gerbe (see Definition \ref{exotic T^2 twisted}).

It is worth emphasizing that the {\em vanishing of the moment map is
not essential} for the construction of the equivariant twisted
Bismut--Chern character. We impose this condition here only because, in
later applications, we will specialize to the case $M=\LX$ and take
$(\mathcal{G},\nabla^B)$ to be the lifting gerbe, which indeed has the vanishing moment
map. Furthermore, the entire construction extends naturally to $G$-equivariant twisted Bismut--Chern character for any compact Lie group $G$, when formulated in the Cartan model of
equivariant cohomology.

\subsubsection{Lifting gerbe, gerbe modules from positive energy representations and Chern characters on \LX} 

$\, $

Given a principal $G$-bundle with connection together with a
positive-energy representation,  in
Section~\ref{sec: Elliptic Loop Chern Character on LX}, we construct two Chern
characters on $\LX$: the \textbf{$q$-graded Bismut--Chern character} and
the \textbf{elliptic Chern character}.

Let $(P,A)\to X$ be a principal $G$-bundle equipped with a connection,
and let $\per^k(LG)$ denote the representation ring of level-$k$
positive-energy representations of the loop group $LG$. For
$\mathcal{H}\in \per^k(LG)$, there is an energy decomposition
\be
\mathcal{H}=\bigoplus_{n\geq 0}\mathcal{H}_n,
\ee
where each $\mathcal{H}_n$ is a finite-dimensional $G$-representation.
Associated to this decomposition is a $q$-graded vector bundle
\begin{align}
\Psi(P,\mathcal{H})
\;=\;
\sum_{n=0}^{\infty}\bigl(P\times_G\mathcal{H}_n\bigr)\,q^n
\;\in\; K(X)[[q]].
\end{align}
The corresponding $q$-graded Chern character is given by
\be
\begin{split}
   Ch:\;\per^k(LG)&\longrightarrow H^{2*}(X)[[q]], \\
      \mathcal{H}&\longmapsto
      \sum_{n=0}^{\infty} Ch\bigl(P\times_G\mathcal{H}_n\bigr)\,q^n .
\end{split}
\ee
This character appears naturally in the twisted Witten genus
\cite{Brylinski1990RepresentationsOL,Liu1995OnMI}, and is closely related
to Miller’s elliptic character \cite{Miller1989EllipticCW}. It also
arises in \cite{Berwick-Evans:2021jlr} in the construction of cocycle
representatives for universal Euler and Thom classes in complex analytic
equivariant elliptic cohomology. Modular invariance emerges only after multiplication by the modular anomaly
$q^{m_{\Lambda}}$, where $\Lambda$ denotes the highest weight of
$\mathcal{H}$. Then the modified character
$q^{m_{\Lambda}} Ch(\mathcal{H})$ can be expressed as a linear
combination of Jacobi theta functions. 

Motivated by Bismut-Chern character, it is natural to define
the \emph{$q$-graded Bismut--Chern character} as a loop-space refinement
of the $q$-graded Chern character:
\be
\begin{split}
 BCh:\; \per^k(LG)&\longrightarrow h_{S^1}^{2*}\!\bigl(\LX\bigr)[[q]],\\
        \mathcal{H}&\longmapsto
        \sum_{n=0}^{\infty}
        BCh\bigl(P\times_G\mathcal{H}_n\bigr)\,q^n .
\end{split}
\ee

\medskip
In order to obtain a genuine elliptic Chern--Weil theory of loop groups on
$\LX$, we work with the lifting gerbe
$(\mathcal{G}_P,\nabla^B)$ on $\LX$ and with gerbe modules arising from
positive-energy representations.

The lifting gerbe measures the
obstruction to lifting the principal $LG$-bundle $LP\to \LX$ to a
principal $\widetilde{LG}$-bundle (see
\cite{Gomi2001ConnectionsAC,Waldorf2010ALS}).
Let $R$ denote the curvature of the connection $A$, and set
$
\Phi=\langle R,R\rangle 
$
as the degree $4$ curvature form defined by the minimal bilinear form. It serves as the curving of the Chern--Simons $2$-gerbe on
$X$ \cite{Waldorf2009StringCA}. The associated lifting gerbe on $\LX$ is
$\cir$-equivariant with vanishing moment map, and its curving, denoted by
$H$, is obtained by transgressing $\Phi$ from $X$ to $\LX$.

For $\mathcal{H}\in \per^k(LG)$, there exists an associated
infinite-dimensional gerbe module
$(\mathcal{E},\nabla^{\mathcal{E}})$ on the $k$-fold tensor power
$(\mathcal{G}_P,\nabla^B)^{\otimes k}$. A naive approach to defining an
elliptic Chern character would be to apply the
$\cir$-equivariant twisted Chern character to
$(\mathcal{E},\nabla^{\mathcal{E}})$ on $\LX$. However, this approach fails
because the trace-class condition required in
\cite{Mathai2002ChernCI} is not satisfied in this infinite-dimensional
setting.

\medskip
\noindent\emph{To remedy this, we instead work over the interior of a punctured disk rather than restricting to its boundary circle.}

\medskip
Let $\disc=\{q\in \CC^*, |q|<1\}$ be the punctured disk, which can be identified with the upper-plane $\mathbb{H}$ via the  map $q=e^{2\pi i\tau}$.
For $\tau\in \mathbb{H}$, we introduce the cohomology group
$h_{\disc}^{*}(\LX,kH)\big|_{\tau}$, defined as the cohomology of the
complex
\begin{align}
  \bigl(\Omega^*(\LX)^{\cir},\;
  \mathcal{D}_{\tau}= d+\tau\,\iota_K - kH \bigr).
\end{align}
We refer to this as the \textbf{completed-periodic $\disc$-extended
$S^1$-equivariant twisted cohomology of $\LX$ at $\tau$}.

\begin{remark}
We emphasize that the symbol $\disc$ in the notation does not denote an
action of $\disc$ on $\LX$. 
Throughout, we work with
differential forms over $\CC$.
\end{remark}

Our \textbf{elliptic Chern character} is then defined as a map
\be
 ECh_A:\; \per^k(LG) \longrightarrow 
 \bigl\{\text{cocycles in }\; h_{\disc}^{2*}\!\bigl(\LX,kH\bigr)\bigr\}
 \big|_{\tau}.
\ee

We now proceed to explain the construction in detail.
At the level of representation theory, the Kac--Weyl character diverges when $q\in S^1$.
Nevertheless, upon extending $\cir$ to the punctured disc
$\disc=\{q\in \CC^*,\, |q|<1\}$, the convergence issue can be overcome.
The parameter $\tau\in\mathbb{H}$ on the upper half-plane now appears as the generator of the $\disc$--action.
We therefore consider the \textbf{$\tau$--deformed curvature}, \ie the equivariant curvature with respect to the complex vector field $\tau K$,
where $K$ is the generator of rotation of loops on $\LX$.

Let $\widetilde{L\g_{\CC}}'=L\g_{\CC}\oplus \CC \dd\oplus \CC \cc$ be the complex double extended loop algebra (see Subsection \ref{Double extended loop alg}),
where $\cc$ is the central part and $\dd$ is the derivation part.
Locally, the $\tau$-deformed curvature takes the form on $\LX$ with values in 
 $\widetilde{L\g_{\CC}}'$ as follows:
\begin{align}
 \tau \dd + \bigl(\tau \iota_K\hat{A} + \hat{R} + \square \cc\bigr),
\end{align}
where $\hat{A}$ is the evaluation of the connection $1$-form on $LX$,
and $\hat{R}$ is the evaluation of the curvature $2$-form on $LX$.
The central part $\square\cc$ arises from the local data of the lifting gerbe.
The elliptic Chern character on $\LX$ is locally given by
\begin{align}
 \Tr_{\mathcal{H}}\!\left[\exp\bigl(\tau \dd + (\tau \iota_K\hat{A} + \hat{R} + \square \cc)\bigr)\right].   
\end{align}
We prove that these local expressions assemble into a globally defined even form on $\LX$
and it is $D_{\tau}$-closed, based on the framework of equivariant twisted Chern character discussed in Subsubsection \ref{intro:eq twisted bch}.

The main difficulty is to establish convergence of the expression above.
To overcome this, we will combine energy estimates in \cite{Goodman1984StructureAU} with the Floquet theory \cite{FrenkelOrbital}.

We now summarize our two formulations of Chern characters defined on $\LX$ in the following diagram.
\be 
\begin{tikzcd}[column sep=large, row sep=large]
h_{S^1}^{2*}(LX)[[q]]
    \arrow[dr, "i^*"'] &
\per^k(LG)
    \arrow[l, "BCh"']
    \arrow[d, "Ch"]
    \arrow[r, "ECh"]
& h_{\disc}^{2*}(\LX,kH)\big|_{\tau}
    \arrow[dl, "i^*"] \\
& h^{2*}(X)[[q]] &
\end{tikzcd}
\ee

The relation between above two now is not clear.
Afterwards, we will show
{\bf they are restrictions along two directions of loops of the elliptic Bismut-Chern character}.

As an application, we define the elliptic Chern character of the loop
spinor bundle $\mathcal{S}_{\LX}$ on $\LX$, and thereby construct a
differential form on $\LX$ which may be viewed as a
$\tau$-deformation of the $\hat{A}$-form on loop space
(see \cref{hatA on LX}). We further establish its relation to the
Witten genus, upon localization to $X$
(see \cref{localized to be Witten genus}).

\medskip

We now turn to the double loop space $\LLX$ and describe our results in
this setting.

\subsubsection{Elliptic Atiyah--Witten formula on $\LLX$}

As indicated above, one of the central aims of this paper is to extend
the classical Atiyah--Witten formula from the loop space to the
double loop space. We refer to this extension as the
\textbf{elliptic Atiyah--Witten formula}, and establish it in
Section~\ref{Sec: Elliptic AW and Modular}.

Recall that the classical Atiyah--Witten formula asserts that the
supertrace of the spinor holonomy along a loop coincides with the
$\zeta$-regularized Pfaffian of the Dirac operator on the circle. 
In the
elliptic setting, the role of the supertrace is played by the super
\emph{elliptic holonomy} of the loop spinor gerbe module, while the
analytic counterpart is given by the Pfaffian of family of 
Dirac operators on an
elliptic curve, coupled to vector potentials parametrized by $\LLX$ (see \eqref{intro: family index}).

We define the \textbf{elliptic holonomy} on $\LLX$ as a $\tau$-deformed
equivariant twisted holonomy of a gerbe module
$(\mathcal{E},\nabla^{\mathcal{E}})\to \LX$ arising from a positive-energy
representation $\mathcal{H}$. The resulting elliptic holonomy
$Ehol_{\mathcal{H}}$ is a $T^2$-equivariant section of the transgression
line bundle $\mathcal{L}(\mathcal{G}_P)\to \LLX$, where
$\mathcal{L}(\mathcal{G}_P)$ denotes the transgression line bundle
associated to the lifting gerbe.

We locally consider a first-order differential operator in the $y$-direction 
with values in $\widetilde{L\g_{\CC}}'$:
\begin{align}
\mathcal{D}_y=\frac{d}{dy}-(\tau \dd+\iota_{\tau \partial_x - \partial_y}\widetilde{A}),
\end{align}
where the evaluation on the $y$-direction $(\iota_{\tau \partial_x - \partial_y}\widetilde{A})_y\in \widetilde{L\g_{\CC}}$ includes a central part that arises from the local data of the lifting gerbe.
The elliptic holonomy is then defined as the trace of the parallel transport of $\mathcal{D}_y$ 
under the positive-energy representation $\mathcal{H}$ (see Subsection~\ref{subsec: Elliptic Holonomy} for details).

A key analytical difficulty is that convergence of the elliptic
holonomy cannot be established for arbitrary smooth double loops. To
overcome this issue, we restrict attention to
\emph{gauged polystable double loops}, for which convergence can be
proved rigorously. Polystability is defined in terms of the long-time
convergence of the Yang--Mills flow on Riemann surfaces, and is
characterized by complex gauge-group orbits of flat connections (cf.
\cite{Preparation}).

We now explain the elliptic Atiyah--Witten formula. Let $(X,g)$ be a
$2n$-dimensional spin Riemannian manifold with Levi--Civita connection
$\nabla^g$ on $TX$. Denote by $P_{\mathrm{Spin}}\to X$ the spin frame
bundle, and by $A$ the induced spin connection.

\medskip
\noindent\textbf{The analytic side.}
The Pfaffian line bundle arises naturally from a $2$-transgression
construction. Fix $\tau\in\mathbb{H}$ and endow the elliptic curve
\[
\Sigma_\tau=\CC/(\Z\oplus \tau\Z)
\]
with its standard flat K\"ahler metric of unit volume. For each choice
of spin structure on $\Sigma_\tau$, specified by a $2$-torsion point
$\frac{i+\tau j}{2},(i,j)\in \Z_2\times \Z_2$, there is an associated Pfaffian line bundle
\[
\mathrm{PF}_{i,j}\;\longrightarrow\;\LLX,
\]
equipped with the Bismut--Freed connection
$\nabla^{\mathrm{PF}}_{i,j}$ and the Quillen metric
$g^{\mathrm{PF}}_{i,j}$. This line bundle is canonically the square root
of the determinant line bundle arising from the family index theorem
applied to the family
\be
\begin{tikzcd}
    & ev^*(TX,\nabla^g) \arrow[d] &\\
 \Sigma_\tau \arrow[r]   & \LLX \times \Sigma_\tau \arrow[d]  \\
&    \LLX &
\end{tikzcd}\label{intro: family index}
\qquad\Longrightarrow\qquad
\begin{tikzcd}
(\mathrm{PF}_{i,j},\nabla^{\mathrm{PF}}_{i,j},g^{\mathrm{PF}}_{i,j})
\arrow[d] \\
\LLX .
\end{tikzcd}
\ee
In the genus-one case, the Atiyah--Singer index theorem implies that the
index of this family vanishes identically. Consequently, each Pfaffian
line bundle $\mathrm{PF}_{i,j}$ admits a canonical Pfaffian section,
denoted $\pf_{i,j}[\tau]$.

\medskip
\noindent\textbf{The topological side.}
From the viewpoint of $(1+1)$-transgression, we consider the four virtual
positive-energy representations of $L\mathrm{Spin}(2n)$ at level one,
which form a basis of $\per^1(L\mathrm{Spin}(2n))$ (see
\cite{Brylinski1990RepresentationsOL}):
\be
\per^1(L\mathrm{Spin}(2n))
=
\{\,S^+-S^-,\; S^++S^-,\; S_+-S_-,\; S_++S_-\,\}.
\ee
These four virtual representations correspond to the four spin
structures on the elliptic curve, according to the identifications
\be
S_{11}=S^+-S^-,\qquad
S_{10}=S^++S^-,\qquad
S_{01}=S_+-S_-,\qquad
S_{00}=S_++S_-.
\ee
To each $S_{ij}$ there is an associated elliptic holonomy
\[
Ehol_{S_{ij}}[\tau]\in
\Gamma\bigl(\LLX,\mathcal{L}(\mathcal{G}_P)\bigr)^{T^2},
\]
where $\mathcal{L}(\mathcal{G}_P)$ denotes the transgression line bundle
of the lifting gerbe.

\medskip
In particular, the odd spin structure corresponds to the level-one Fock
representation $S_{11}=S^+-S^-$. This representation may be regarded as
the loop-space analogue of the finite-dimensional spin representation.
The associated gerbe module is precisely the loop spinor gerbe module
$\mathcal{S}_{\LX}\to\LX$. If $X$ admits a string structure, the
transgressed loop-spin structure lifts this gerbe module to an honest
Fock bundle on $\LX$, yielding the spinor bundle
$\mathcal{S}_{\LX}$ of the loop space. For rigorous constructions via
smooth Fock bundles, we refer to \cite{Kristel2020SmoothFB}.

\begin{theorem}[Elliptic Atiyah--Witten formula]\label{elliptic AW formula}
There is an isometry between the Pfaffian line bundle and the
transgression line bundle associated to the lifting gerbe,
\begin{align}
\Phi_{ij}:\;
(\mathrm{PF}_{ij},\nabla^{\mathrm{PF}}_{ij},g^{\mathrm{PF}}_{i,j})
\xrightarrow{\;\cong\;}
(\mathcal{L}(\mathcal{G}_P,\nabla^B),h),
\end{align}
where $h$ denotes the canonical metric on the transgression line bundle.
Under the isometry, for gauged poly-stable double loops one has
\be
\Phi_{1-i,\,1-j}\!\bigl(\pf_{1-i,\,1-j}[\tau]\bigr)
\;=\;
q^{m_{ij}}\cdot Ehol_{S_{ij}}[\tau],
\ee
where $m_{ij}$ denotes the modular anomaly.
\end{theorem}

\begin{remark} The identification of the odd–spin Pfaffian with the elliptic holonomy of
the loop spinor gerbe module associated to
$S_{11}=S^+-S^-$ is particularly striking: it is the most natural
analogue of the classical Atiyah–Witten formula. 
In contrast,
the remaining three even spin structures admit no comparable
finite-dimensional counterpart, reflecting their intrinsically
elliptic nature.

When restricted to constant double loops, identified with $X$, the
elliptic holonomy corresponding to $S^+-S^-$ vanishes. On the analytic
side, the Pfaffian section associated to the odd spin structure also
vanishes, since the Dirac operator on the torus with odd spin structure
(the $\bar{\partial}$-operator) has a nontrivial kernel of constant
dimension. This vanishing is not accidental: it precisely mirrors the
familiar phenomenon that occurs when restricting the classical
Atiyah--Witten formula from the loop space $\LX$ back to $X$.
\end{remark}

Let $i_y,i_x:\LX\to \LLX$ denote the inclusions corresponding to
double loops that are constant in the $y$- and $x$-directions,
respectively. By construction, the elliptic holonomy restricts along the
$y$-direction to the degree-zero component of the elliptic Chern
character. Its restriction along the $x$-direction coincides with the
degree-zero component of the $q$-graded Bismut--Chern character on $\LX$,
namely the $q$-graded holonomy. A natural question is how the elliptic
Atiyah--Witten formula degenerates to the classical Atiyah--Witten
formula in the limit $\tau\to i\infty$.

On the holonomy side, the restriction of the modified elliptic holonomy
along the $x$-direction satisfies
\be \label{res to LX as q}
i_x^*\bigl(q^{m_{\Lambda}}\cdot Ehol_{S^+-S^-}[\tau]\bigr)(\gamma)
    = q^{\frac{n}{12}}\bigl(\Tr_s[hol_{\mathcal{S}_X}(\gamma)] + O(q)\bigr),
    \quad \gamma\in \LX.
\ee
Here $O(q)$ denotes higher-order terms in $q$, arising from the holonomy
of higher-energy modes in the loop group representation. After dividing
by the Virasoro anomaly factor $q^{\tfrac{n}{12}}$, which originates from
the regularization $\sum_{m=1}^\infty m=-\tfrac{1}{12}$, the elliptic
holonomy reduces precisely to the supertrace of the spinor holonomy on
$\LX$.

On the Pfaffian side, the degeneration is more subtle, and the Virasoro
anomaly is not immediately visible. Nevertheless, the elliptic
Atiyah--Witten formula allows one to deduce the limiting behavior of the
Pfaffian. Combining the elliptic Atiyah--Witten formula,
\cref{res to LX as q}, and the classical Atiyah--Witten formula
\be
\Tr_s\!\bigl[hol_{\mathcal{S}_X}(\gamma)\bigr]
=
\pf_{\zeta}\!\bigl(\nabla_{\dot{\gamma}}\bigr),
\ee
we obtain the following corollary. Let $\pf[\tau]$ denote the Pfaffian section corresponding to the odd spin structure.
\begin{corollary}
\be \label{limq}
 q^{-\frac{n}{12}}\,
i_x^*\bigl(\pf[\tau]\bigr)(\gamma)
    \to
\pf\!\bigl(\nabla_{\dot{\gamma}}\bigr),
\quad q\to 0.
\ee
\end{corollary}
This formula clarifies how the Pfaffian on the double loop space $\LLX$
degenerates to the Pfaffian on the loop space $\LX$ as the complex
structure parameter of the elliptic curve approaches the cusp,
$\tau\to i\infty$. Specifically, this degeneration is a loop
space analogue of the following product expansion:
\be
q^{-\frac{1}{12}}\cdot \frac{\theta_{11}(z,\tau)}{\eta(\tau)}
    =(e^{\pi i z}-e^{-\pi i z})
      \prod_{m=1}^{\infty}\bigl(1-q^m e^{2\pi i z}\bigr)
      \bigl(1-q^m e^{-2\pi i z}\bigr).
\ee

\subsubsection{Geometric quantization of Chern--Simons gauge theory and double-loop space geometry}

$\,$

Our proof of the elliptic Atiyah–Witten formula is based on the
geometric quantization of Chern–Simons gauge theory
\cite{Axelrod1991GeometricQO}. We begin by describing the isometry of line bundles in \cref{elliptic AW formula}.
The identification between the transgression line bundle of the lifting gerbe
and the Pfaffian line bundle is not accidental: both are isomorphic to the Chern–Simons prequantum line bundle.

Let
\(\mathcal{A}=\Omega^1(\Sigma,\g)\) denote the space of \(G\)-connections on the torus \(\Sigma\),
and let \(\Sigma G=C^{\infty}(\Sigma,G)\) be the gauge group.
In \cref{subsec: EQ of Line Bundles} we review the construction of the universal Chern–Simons line bundle
\((\mathcal{L},h,\nabla)\) over \(\mathcal{A}\), which is a \(\Sigma G\)-equivariant prequantum line bundle \cite{FREED1995237}.
Evaluating a connection on double loops $\gamma\to \gamma^*A$ 
defines a \(\Sigma G\)-equivariant map
\begin{align}
    \Psi_A:\; L^2P \longrightarrow \mathcal{A}.
\end{align}
This induces the stack morphism \(\LLX\to [\mathcal{A}/\Sigma G]\).
Pulling back \((\mathcal{L},h,\nabla)\) along \(\Psi_A\) produces the Chern–Simons line bundle
\((\mathcal{L}_{\CCS},\nabla_{\CCS}^A,h)\) on \(\LLX\), which we will show coincides with the transgression line bundle \(\mathcal{L}(\mathcal{G}_P,\nabla^B,h)\). On the other hand,  the isometry between Pfaffian line bundle and Chern–Simons line
bundle on $\LLX$ has already been established in \cite{Bunke2009StringSA}. 
In \cite{Preparation}, the isometry is shown universally
over $\mathcal{A}$, viewing both line
bundles as $\Sigma G$–equivariant Hermitian prequantum line bundles.

The main idea of the proof of Theorem~\ref{elliptic AW formula} is to
establish the identity first on the coarse moduli $M_G[\tau]$, \ie the moduli space of
semistable holomorphic principal $G_{\CC}$-bundles over the elliptic
curve. When $G=\mathrm{Spin}(2n)$, the elliptic Atiyah--Witten formula on
$\LLX$ then arises as a consequence of the corresponding formula on
$M_G[\tau]$, which we refer to as the \emph{elliptic Atiyah--Witten
formula on the moduli space}, asserting that the push-down
of the four Pfaffian sections to $M_G[\tau]$ conincide the
Kac--Weyl characters of the four level-one virtual representations
$\{S_{ij}\}$. See \cref{subsec:Elliptic Atiyah-Witten on MG} for details.

Throughout this analysis, Chern--Simons gauge theory appears naturally
through the mechanism of $2$-transgression. In particular, the elliptic
holonomy admits a natural interpretation in terms of the principle that
\emph{quantization commutes with reduction} for Chern--Simons gauge
theory.

Unlike the $(1+1)$-transgression, where the parameter $\tau\in\mathbb{H}$
serves as the generator of the $\disc$-action, the role of $\tau$ here is
of a different nature: it specifies the complex structure of the torus
and thereby determines the choice of polarization in geometric
quantization. Once $\tau$ is fixed, the space of connections
$\mathcal{A}$ acquires a K\"ahler structure. After choosing this
polarization, the space
\be
\mathcal{H}_k(G)|_{\tau}:=H^0_{\Sigma G}(\mathcal{A},\mathcal{L}^k)
\ee
is the space of quantization before reduction, while its
push-down, the space of quantization after reduction, is precisely the
genus-one conformal block $V_k(G)|_{\tau}$. Precise definitions of
$M_G[\tau]$ and $V_k(G)|_{\tau}$ are taken from
\cite{Axelrod1991GeometricQO,Looijenga1976RootSA}, and the necessary
background will be reviewed in
\cref{genus one CS and CB}.

Within this framework, there is a complex-analytic form of the principle
that quantization commutes with reduction, to be compared with the
algebraic-stack counterpart in \cite{Teleman1998TheQC}.

\begin{theorem}[{\cite{Preparation}}]\label{[Q,R]=0}
The push-down map
\be
r:\mathcal{H}_k(G)|_{\tau}\longrightarrow V_k(G)|_{\tau}
\ee
is bijective.
\end{theorem}
As a consequence, we obtain the composition
\be
\begin{tikzcd}
V_k(G)|_{\tau} \arrow[r,"r^{-1}"] &
\mathcal{H}_k(G)|_{\tau} \arrow[r,"\Psi_A^*"] &
\Omega^0(\LLX,\mathcal{L}_{\CCS}^k)^{T^2},
\end{tikzcd}
\ee
which assigns to each element of the conformal block a $T^2$-invariant
section of the Chern--Simons line bundle over the entire double loop
space.

The following theorem shows that the conformal blocks are generated by
Kac--Weyl characters.

\begin{theorem}[{\cite{Looijenga1976RootSA}}]\label{looijenga character and genus one cb}
The genus-one conformal block $V_k(G)|_{\tau}$ is generated by the
modified Kac--Weyl characters of level-$k$ positive-energy
representations. Consequently, after tensoring with $\CC$, the modified
Kac--Weyl character map
\be
\mathrm{ch}:\per^k(LG) \longrightarrow V_k(G)|_{\tau}
\ee
is an isomorphism.
\end{theorem}
On gauged polystable double loops, the composition
$
\Psi_A^*\circ r^{-1}\circ \mathrm{ch}
$
coincides with the elliptic holonomy $Ehol_A$ obtained via
$(1+1)$-transgression. Accordingly, we continue to refer to the map
$\Psi_A^*\circ r^{-1}$ as the elliptic holonomy.

\begin{remark}
We now further explain why it is natural to refer to $\Psi_A^*\circ r^{-1}$ as
the \emph{elliptic holonomy}. Classically, the trace of the
universal holonomy on $X=BG$ yields the character map. After tensoring
with $\CC$, the ordinary trace of the holonomy on $BG$ gives
\begin{align}\label{analogue BG}
hol:\mathrm{Rep}(G)\cong K_G(pt)=K(BG)\longrightarrow
\CC[LBG]=\CC[G]^G,
\end{align}
sending a representation to its character as a class function on $G$. Here we use the
Morita equivalence of groupoids
$LBG=[\mathcal{A}_{S^1}/LG]=[G/G]$ (cf.\cite{Behrend2003EquivariantGO}).

Passing from $K$-theory to elliptic cohomology, the representation ring
$\mathrm{Rep}(G)$ is replaced by the ring of positive-energy
representations $\per^k(LG)$. According to Grojnowski
\cite{Grojnowski2007EllipticCD}, Ando \cite{Ando2000TheWG}, and Lurie
\cite{Lurie2009ASO}, this ring appears as the $k$-twisted
$G$-equivariant elliptic cohomology of a point.
Correspondingly, there
is an elliptic analogue of \eqref{analogue BG}:
\begin{align}
\per^k(LG)=Ell^k_G(pt)\longrightarrow
H^0(\Sigma BG,\mathcal{L}^k).
\end{align}
Here $\Sigma BG$ should be interpreted algebraically as 
$[\Sigma_{\tau},BG_{\CC}]=Bund_G(\Sigma_{\tau})$, the moduli stack of
holomorphic $G_{\CC}$-bundles over the elliptic curve $\Sigma_{\tau}$.
The universal elliptic holonomy thus takes the form
\begin{align}
Ehol:\per^k(LG)\longrightarrow
H^0(Bund_G(\Sigma_{\tau}),\mathcal{L}^k).
\end{align}
Recall that $Bund_G(\Sigma_{\tau})$ can be presented as the complex-analytic Atiyah--Bott stack $[\mathcal{A}/\Sigma G_{\CC}]$, and that $\mathcal{L}$ denotes the universal Chern--Simons line bundle.
 Hence the right-hand side is precisely $\mathcal{H}_k(G)$, 
the space of quantization before reduction.
Taken together with Theorem \ref{[Q,R]=0}, this justifies the interpretation of $Ehol$ as the elliptic analogue of holonomy.

In \cite{Freed2005LoopGA}, $\per^k(LG)$ is identified with the $k$-twisted $K$-theory on $[G/G]$ as a ring.
$\mathrm{Bund}_G(\Sigma_{\tau})$ admits a $q$-conjugacy class presentation $[LG^h/_qLG^h]$ (see, e.g., \cite{etingof_central_1994,Baranovsky1996ConjugacyCI}), where $LG^h$ denotes the holomorphic loop group.
The elliptic holonomy can also be viewed as a map from twisted $K$-theory on $[G/G]$ to holomorphic sections of the transgression line bundle over  $[LG^h/_qLG^h]$.
\end{remark}

Many of the constructions arising from $2$-transgression admit natural
extensions beyond the genus-one case. Let $J$ be a complex structure on a
closed surface $\Sigma$ of genus $g>1$, and denote by $V_k(G)|_{J}$ the
corresponding genus-$g$ conformal block at level $k$. In complete
analogy with the genus-one situation, one may define the
\textbf{higher-genus holonomy} as the composition
\begin{align}
Ehol_A:\;
V_k(G)|_{J} \longrightarrow \mathcal{H}_k(G)|_{J}
\longrightarrow \Gamma(\Sigma X,\mathcal{L}_{\CCS}^k),
\end{align}
where $\Sigma X$ denotes the mapping space from $\Sigma$ into $X$. This
construction provides a higher-genus generalization of the elliptic
holonomy, see Subsection
\ref{subsec: from QR=0 of CS and Conformal Blocks} for details.

We will not explore these deep topics any further in this paper. We only present some computations of the relative Pfaffian line bundle in the genus-one case and investigate the modular behavior of the relative odd Pfaffian. For discussions of the relative Pfaffian
line bundle, we refer to \cite{Freed1987OnDL,Bunke2009StringSA}. We identify its natural
counterpart on $[\mathcal{A}/\Sigma G]\times\mathbb{H}$ and study its
push-down to the subspace of constant Cartan connections. The relative odd Pfaffian, up to
a phase factor, is given by
$
\frac{\theta_{11}(z,\tau)}{\eta(\tau)^3}.
$
Here $\frac{\theta_{11}(z,\tau)}{\eta(\tau)}$ gives the odd Pfaffian, and  $\eta(\tau)^{-2}$ from relative structure cancels the
$(\mathbb{Z}/24\mathbb{Z})$-anomaly. However, the resulting relative odd
Pfaffian is still not modular under the $S$-transformation $(\tau,z)\to (-1/\tau,z/\tau)$. This fact reflects that the Hitchin connection is only projectively flat.

\medskip

\subsubsection{Elliptic Bismut--Chern characters on $\LLX$}

In the final part of the paper, we construct the \textbf{elliptic
Bismut--Chern character} on $\LLX$, whose degree-zero component recovers
the elliptic holonomy. 

This construction arises naturally from
$(1+1)$-transgression. To this end, we introduce the
\emph{completed-periodic exotic twisted}
$(S^1\times \disc)$-equivariant cohomology of $\LLX$, twisted by the
$k$-th power of the Chern--Simons line bundle
$(\mathcal{L}_{\CCS},\nabla_{\CCS}^A)$, and denote it by
\begin{align}
 h_{S^1\times \disc}^{*}\!\bigl(
 \LLX,(\mathcal{L}_{\CCS},\nabla_{\CCS}^A,\overline{H})^k
 \bigr).
\end{align}
As before, we emphasize that the double loop space $\LLX$ itself does not
carry a genuine $\disc$-action; the appearance of $\disc$ merely reflects
the replacement of the vector field $K$ by $\tau K$ in the corresponding
equivariant differential.

Within this framework, the elliptic Bismut--Chern character is defined
as the map
\begin{align}
EBCh_A:\; \per^k(LG) \longrightarrow 
\bigl\{\text{cocycles in }\;
h_{S^1\times \disc}^{2*}\!
\bigl(\LLX,(\mathcal{L}_{\CCS},\nabla_{\CCS}^A,\overline{H})^k\bigr)|_{\tau}
\bigr\}.
\end{align}
From the perspective of the double loop space, this construction
provides a unified conceptual framework that clarifies the relationship
between the $q$-graded Bismut--Chern character and the elliptic Chern
character on $\LX$: both arise simply as restrictions of the elliptic
Bismut--Chern character on $\LLX$ along the two independent loop
directions.

However, we must emphasize that since we can only establish convergence of the elliptic
Bismut--Chern character on gauged polystable double loops, the above picture 
remains a subject for future work.

All of the Chern characters introduced in this paper are summarized in
the following commutative diagram:
\be 
\begin{tikzcd}
& 
h_{S^1 \times \disc}^{2*}
\!\bigl(
\LLX,\,(\mathcal{L}_{\CCS}, \nabla_{\CCS}^{A}, \overline{H})^k
\bigr)\big|_{\tau}
\arrow[ld, "i_{10}^*"']
\arrow[dr, "i_{01}^*"]
& \\[0.5em]
h_{S^1}^{2*}(\LX)[[q]]
    \arrow[dr, "i^*"']
& 
\per^k(LG)
    \arrow[l, "BCh"']
    \arrow[d, "Ch"]
    \arrow[r, "ECh"]
    \arrow[u, "EBCh"]
&
h_{\disc}^{2*}(\LX, kH)\big|_{\tau}
    \arrow[dl, "i^*"]
\\
& h^{2*}(X)[[q]] &
\end{tikzcd}
\ee
This diagram captures transgression procedures and
highlights the role of the elliptic Bismut--Chern character as the
central organizing object from which the various loop-space Chern
characters naturally descend.

\medskip

\noindent\textbf{Organization of the paper.}
In \cref{sec: Loop groups and positive energy representations}, we review
loop groups, loop algebras, and positive-energy representations. In
\cref{sec:twisted bismut}, we introduce the equivariant twisted
Bismut--Chern character. The elliptic Chern character on $\LX$ is
constructed in \cref{sec: Elliptic Loop Chern Character on LX}. In
\cref{Sec: Elliptic AW and Modular}, we establish the elliptic
Atiyah--Witten formula on $\LLX$ and investigate its relation to
Chern--Simons gauge theory. Finally, in
\cref{sec: Elliptic Bismut Chern Character}, we introduce the elliptic
Bismut--Chern character on $\LLX$.

\medskip

%%%%%%%%%%%%%%%%%%%%%%%%%%%%%%%%%%%%%%%%%%%%%%%%%%%%%%%
\section{Loop Groups, Loop Algebras, and Positive Energy Representations}\label{sec: Loop groups and positive energy representations}
We present an overview of loop groups, loop algebras, and positive energy representations. 
The Hamiltonian geometry of loop groups plays a crucial role in understanding the geometry of the lifting gerbe on the loop space.
Positive energy representations are the central theme throughout this paper.

{\bf This section is organized as follows.}
In Subsection \ref{subsec: Central extension of LG}, we review the Hamiltonian geometry on loop groups and the central extension of loop groups.
In Subsection \ref{subsec: Extended Structures and Central Extensions}, we review the extended structures on loop algebras.
In Subsection \ref{subsec: Positive Energy Representations and Characters}, we review positive energy representations and Kac-Weyl characters.
\subsubsection*{\textbf{Conventions and Notations}}
Without further specification, $G$ is a group of Cartan type, \ie compact, connected, simple and simply connected.
$G_{\CC}$ is the complexification of $G$, $\g_{\CC}=\g\otimes_{\R}\CC$ is the complexification of the Lie algebra $\g$.
Let $T$ be a maximal torus of $G$, $\mathfrak{t}$ is the Lie algebra of $T$, which is the Cartan subalgebra of $\g$.
$\mathfrak{t}_{\CC}$ is the complexification of $\mathfrak{t}$.

The Killing form $\mathfrak{K}$ is given by the trace of adjoint representation.
Since $\g$ is assumed to be simple, Killing form $\mathfrak{K}$ is non-degenerate and can be normalized to the minimal one as the generator.
Let
$\langle \cdot,\cdot \rangle$ be the minimal bilinear form on Lie algebra $\g$ such that the longest roots have square length $2$.  Then
\begin{align}
\mathfrak{K}(X,Y)=\Tr(ad_Xad_Y)=2h^{\vee}\langle X,Y \rangle.
\end{align}
 $h^{\vee}$ is the dual Coxeter number. For the general representation $\rho:\g \to End(V)$, 
 the Dynkin index $d_V$ is defined by
\begin{align}
\Tr_V(\rho(X)\rho(Y))=d_V\langle X,Y \rangle.
\end{align}
From the definition, Dynkin index of the adjoint representation is $d_{ad}=2h^{\vee}$.

\subsection{Hamiltonian geometry on \texorpdfstring{$LG$}{LG} and Central Extensions}\label{subsec: Central extension of LG}
 $LG:=\CC^{\infty}(S^1,G)$ is the smooth loop group and $L\g$ is its Lie algebra. 
Let $R:S^1\to Aut(LG)$ be the rotation of loops, \ie for $g\in LG$,
$R_tg=g_{-t+\cdot}$. The induced vector field on $LG$ generating rotations is the opposite of the derivative 
$K:g\to g'=\frac{dg}{dt}$.

Let $\mu\in \Omega^1(G,\g)$ be the left Maurer-Cartan form of $G$. We frequently use the map $LG\to L\g$, $g\to i_{\partial_t}(g^*\mu)=g^{-1}g'$.

We define the {\bf  energy functional} $E$ on $LG$ as
\be 
E[g]=\frac{1}{4\pi^2}\int_{S^1}\langle g^{-1}g',g^{-1}g'\rangle.
\ee
It is clear that the energy functional is invariant under the adjoint action of $LG$ and the rotation by $S^1$.

 $LG$ can be equipped with the symplectic form $\omega_{LG}$, such that $S^1$-action is Hamiltonian and the moment map is exactly the energy $E$.
\begin{definition}
The symplectic form of $\omega_{LG}$ is given by $\omega$. For $g\in LG$,
$x,y\in T_{g}(LG)$,
\begin{align}
    \omega_{LG}(x,y)_{g}:=\omega(g^{-1}x,g^{-1}y).
\end{align}
It's direct to see $\omega_{LG}$ is indeed symplectic and is $LG\rtimes S^1$-invariant.
\end{definition}
The Chern-Simons $3$-form $\Theta$ as the generator of $H^3(G)$ is
\begin{align}
\Theta(X,Y,Z)=\frac{1}{8\pi^2}\langle X,[Y,Z]\rangle.
\end{align}
On the cohomology level, $[\omega_{LG}/2\pi]\in H^2(LG)$ is the transgression of the Chern-Simons $3$-form $[\Theta]\in H^3(G)$.
\begin{lemma}[{\cite[(4.4.4)]{pressley_loop_1988}},\cite{Coquereaux1989StringSO}]\label{up to exact of LG}
\begin{align}
\frac{1}{2\pi}\omega_{LG}=\int_{S^1}\Theta+d\beta,
\end{align}
where $\beta \in \Omega^1(LG)$,
\begin{align}
\beta_{\gamma}(y)=-\frac{1}{8\pi^2}\int_{S^1}\langle \gamma^{-1}\gamma', (\gamma^{-1}y)\rangle.
\end{align}
\end{lemma}
By computing $\iota_K\beta$, directly we have the following corollary.
\begin{corollary}\label{moment map is energy}
$S^1$-action on $(LG,\omega_{LG})$ is Hamiltonian, the moment map
is the energy $E$. $(LG,\omega_{LG},E)$ is a Hamiltonian $S^1$-space.
\end{corollary}
Since the cohomology class $[\omega_{LG}/2\pi]$ is integral, $(LG,\omega_{LG})$ admits a prequantization. 
The corresponding prequantum circle bundle can be realized as the central extension of $LG$, described by the following short exact sequence of Lie groups:
\be
    \begin{tikzcd}
    1 \arrow[r]& S^1 \arrow[r]& \widetilde{LG}\arrow[r,"\rho"] & LG\arrow[r] & 1.
    \end{tikzcd}
\ee
The model of the central extension follows \cite{Murray2001HiggsFB}.
There is no essential difference with \cite{Mickelsson1987KacMoodyGT}. 

Consider smooth paths $P(LG)=\{\gamma:[0,1]\to LG\mid\gamma(0)=e\}$. 
For $\gamma\in P(LG)$, let $\gamma^{\wedge}:[0,1]\times S^1 \to G$ denote its adjoint map, defined by
$\gamma^{\wedge}(s,t)=\gamma(s)(t)$, with $\gamma^{\wedge}(0,t)=e$ for all $t\in S^1$. 
Let $ev:P(LG)\to LG$ denote the evaluation map at the endpoint.
The central extension $\widetilde{LG}$ is defined as the quotient $(P(LG)\times S^1)/\sim$, where the equivalence relation is
\begin{align}
(\gamma_1,z_1)\sim (\gamma_2,z_2) \quad \text{if}\quad  ev(\gamma_1)=ev(\gamma_2)\quad\text{and}\quad \exp\!\big(i \int_{D}\sigma^*\omega_{LG}\big)\cdot z_1=z_2.
\end{align}
Here $\sigma:D\to LG$ is any extension of $(\gamma_1\#\gamma_2^{-1}):S^1\to LG$ to a smooth map on the disc $D$.
Since the cohomology class $[\omega_{LG}/2\pi]$ is integral, this definition is independent of the choice of extension $\sigma$.
The projection $\rho:\widetilde{LG}\to LG$ is defined by the evaluation map $ev:P(LG)\to LG$. The group multiplication on $\widetilde{LG}$ is given by
\begin{align}
(\gamma_1,z_1)\cdot (\gamma_2,z_2) \mapsto (\gamma_1\gamma_2,z_1z_2\cdot \exp{(-2\pi i \int_{[0,1]\times S^1}(\gamma_1,\gamma_2)^*\alpha)}),
\end{align}
where $\alpha\in \Omega^1(LG\times LG)$ is defined as follows. Let 
$\pi_1,\pi_2:LG\times LG\to LG$ be the projections to the first and second factors,
$
    \alpha:=\int_{S^1}\langle \pi_1^*\hat{\mu},\pi_2^*i_{\partial_t}\hat{\mu}\rangle dt,
$
where $\hat{\mu}\in \Omega^1(LG,L\g)$ is the evaluation of the Maurer-Cartan form $\mu$. 
The pair $(\alpha,\omega_{LG})$ defines a Deligne cocycle, which guarantees that the group structure and associativity are well-defined.

The connection $v$ on $\widetilde{LG}$ is defined canonically. The angular form $\theta\in \Omega^1(S^1)$ and the transgression $\int_{[0,1]}\omega_{LG}\in \Omega^1(P(LG))$ combine to form
\begin{align}
\theta+\int_{[0,1]}\omega_{LG}\in \Omega^1(P(LG)\times S^1).
\end{align}
This combination can be descended to the connection $v$ on $\widetilde{LG}$, whose curvature is given by $dv=\rho^*\omega_{LG}$.
\begin{proposition}
Let $\widetilde{K}$ be the vector field generating the $S^1$-action on $\widetilde{LG}$, such that
\begin{align}
i_{\widetilde{K}}v=-\rho^*E.
\end{align}
\begin{proof}
For $(\gamma,t)\in P(LG)\times S^1$,
$
ev^*(i_{\widetilde{K}}v)(\gamma,t)=\int_{[0,1]}(\iota_K\omega_{LG})_{\gamma}=\int_{[0,1]}-dE=-E(ev(\gamma)).
$
\end{proof}
\end{proposition}
This demonstrates that the associated line bundle of $(\widetilde{LG},v)$ is an $S^1$-equivariant prequantum line bundle on the Hamiltonian $S^1$-space $(LG,\omega_{LG},E)$.
\subsubsection*{\textbf{Maurer-Cartan Form of $\widetilde{LG}$}}
The evaluation $\hat{\mu}\in \Omega^1(LG,L\g)$ is the Maurer-Cartan form of $LG$. The Maurer-Cartan form $\widetilde{\mu}$ of $\widetilde{LG}$ is given by
\begin{align}\label{MC form of central extension}
    \widetilde{\mu}=\rho^*\hat{\mu}+v\in \Omega^1(\widetilde{LG},\widetilde{L\g}).
\end{align}
It satisfies the Maurer-Cartan equation due to the prequantum structure:
\begin{align}
d\widetilde{\mu}+\frac{1}{2}[\widetilde{\mu},\widetilde{\mu}]_{\widetilde{LG}}=dv+\rho^* \omega_{LG}=0.
\end{align}

\subsection{Double Extended Loop Algebras}\label{subsec: Extended Structures and Central Extensions}  
In this subsection, we review the semi-product and central extensions of loop groups and loop algebras.

We call the semi-direct product 
$\widehat{LG}:=LG\rtimes S^1$ the {\bf extended loop group}, whose Lie algebra $\widehat{L\g}=L\g \oplus \R \dd$ can be identified with $\g$-valued first-order differential operators
$
(\eta,A)\to \eta\frac{d}{dt}+A.
$
The adjoint action $\widehat{Ad}:\widehat{LG}\to Aut(\widehat{L\g})$ is given by combining rotation and conjugation:
\be
\widehat{Ad}_{g,\theta}(\eta\frac{d}{dt}+A)=g(\eta\frac{d}{dt}+R_{\theta}A) g^{-1}.
\ee
Let $\omega:L\g\times L\g\to \R$ be the cocycle
\begin{align}
\omega(X,Y):=\int_{S^1}\langle X,dY\rangle=\frac{1}{2\pi}\int_{0}^{2\pi}\langle X(t),Y'(t)\rangle dt,
\end{align}
which defines the central extension of $L\g$.

\begin{definition}\label{Double extended loop alg}
We denote by  
\begin{align}
\widetilde{L\g}':= L\g \oplus \R \dd\oplus \R \cc = \widehat{L\g}\oplus \R \cc,
\end{align}
where $\cc$ is the central element. We call $\widetilde{L\g}'$ the \textbf{double extended loop algebra}. The Lie bracket is given by: for $a,b\in L\g$ and $t\in S^1$,
\begin{align}
[a,b]_t= [a(t),b(t)] - \omega(a,b)\cc, \quad
[\dd,a]= a'.
\end{align}
\end{definition}

Denote by $\langle\cdot,\cdot\rangle$ the invariant bilinear form on $L\g$ induced from the minimal form of $\g$, \ie for $X,Y\in L\g$,
\begin{align}
\langle X,Y \rangle = \int_{S^1}\langle X(t),Y(t)\rangle \, dt.
\end{align}
 $\widetilde{L\g}'$ admits an extension, also denoted by $\langle\cdot, \cdot \rangle$, defined as follows:
\begin{align}
\langle (X + a\dd, c), (Y + b\dd, d) \rangle = \langle X, Y \rangle + ad + bc.
\end{align}   

\begin{proposition}[{\cite[(3.1.5)]{FrenkelOrbital}}]\label{Adjoint formula}
The adjoint action $\widetilde{Ad}$ of $LG$ on $\widetilde{L\g}'$ is given by
\begin{align}
\widetilde{Ad}_g(a\cc + b\dd + x) = \widehat{Ad}_g(b\dd + x) + \left(a + \langle g^{-1}g', x \rangle - \frac{b}{2}\langle g'g^{-1}, g'g^{-1}\rangle\right)\cc.
\end{align}
\end{proposition}
Denote the group cocycle $\mathcal{Z}: LG \times L\g \to \R$ by
\begin{align}\label{Z cocycle}
\mathcal{Z}(g, x) = \langle g^{-1}g', x \rangle.
\end{align}
It satisfies the cocycle condition
\begin{align}
\mathcal{Z}(g_1g_2, x) = \mathcal{Z}(g_1, Ad_{g_2}x) + \mathcal{Z}(g_2, x),
\end{align}
which corresponds to the central extension $\widetilde{LG}$.

%--------------------------------------------------------------------------------------------------------
\subsection{Positive Energy Representations and Kac-Weyl Character}\label{subsec: Positive Energy Representations and Characters} 
We adopt the convention that a positive energy representation of $LG$ refers to a smooth unitary positive energy representation of $\widetilde{LG}'$.
 Up to essential equivalence, every positive energy representation is completely reducible, unitary, extends to the complexification, and admits a projective intertwining action of $\mathrm{Diff}^+(S^1)$ (see \cite{pressley_loop_1988}).
\subsubsection{Integrable Highest Weight Modules and Kac-Weyl Character}
The irreducible positive energy representation is completely described by the corresponding {\bf integrable highest weight modules} at the Lie algebra level.

Consider the subalgebra of polynomial loops in $\widetilde{L\g_{\CC}}'$, referred to as the {\bf affine Kac-Moody algebra} and denoted by $\widehat \g$, which has the structure
$\widehat \g=\CC \cc\oplus \CC \dd \oplus  \g\otimes \CC[z,z^{-1}]$. The algebra $\widehat \g$ has a triangular decomposition
\be
    \widehat \g=\hat{\mathfrak{n}}_-\oplus\hat{\mathfrak{h}}\oplus \hat{\mathfrak{n}}.
\ee
Here, $\hat{\mathfrak{n}}_{\pm}$ are nilpotent subalgebras, and 
\begin{align}
    \hat{\mathfrak{h}}=\CC \cc\oplus \CC \dd\oplus \mathfrak{h}_{\CC}
\end{align}
is the Cartan subalgebra.

Let $\Delta$ be the root space where $\{\alpha_i\}$ is the root basis, and $Q=\Z(\alpha_1,\cdots, \alpha_l)$ is the root lattice.
\begin{definition}{\label{highest weight module}}
A $\widehat \g$-module $L(\Lambda)$ is called a \textbf{highest weight module} with highest weight $\Lambda\in \hat{\mathfrak{h}}^*$ 
 if there exists a nonzero vector $v_{\Lambda}\in L(\Lambda)$ such that
 \begin{align*}
    \hat{\mathfrak{n}}_+(v_{\Lambda})=0, \quad h(v_{\Lambda})=\Lambda(h)v_{\Lambda},\quad
U(\widehat \g)v_{\Lambda}=L(\Lambda),
 \end{align*}
 where $U(\widehat \g)$ is the universal enveloping algebra of $\widehat \g$.
 The level of $L(\Lambda)$ is the scalar $k=\Lambda(c)$. It is called \textbf{integrable} if $\Lambda\in P^+$,
  where $P^+=\{\lambda\in \hat{\mathfrak{h}}^*|(\lambda,\alpha_i)\in \N \}$ is the set of dominant integral weights. 
\end{definition}
By the Serre relations, the integrable property ensures that any element $a \in \widehat{\g}$ acts nilpotently on every vector $v \in L(\Lambda)$.
This integrability guarantees that the highest weight module $L(\Lambda)$ integrates uniquely to a positive energy representation of $LG$.

\begin{definition}
Let $\per^k(LG)$ denote the abelian group generated by equivalence classes of positive energy representations of $LG$ at level $k$.
\end{definition}
Every integrable highest weight module $\mathcal{H} = L(\Lambda)$ decomposes as a direct sum of finite-dimensional irreducible $\g_{\CC}$-modules:
\begin{align}
    L(\Lambda) = \bigoplus_{\mu \in \hat{\mathfrak{h}}^*} L(\Lambda)_{\mu}.
\end{align}

The \textbf{modified Kac-Weyl character}, defined for $\tau \in \mathbb{H}$, is given by
\begin{align}
\chi_{\Lambda}(\tau, \cdot) = q^{m_{\Lambda}} \sum_{\mu \in \hat{\mathfrak{h}}^*} \dim L(\Lambda)_{\mu} e^{\mu},
\end{align}
where the modular anomaly is defined as
\begin{align}
m_{\Lambda} = h_{\Lambda} - \frac{c}{24}.
\end{align}
Here, $h_{\Lambda} = \frac{\langle \Lambda + 2\rho, \Lambda \rangle}{2(k + h^{\vee})}$ denotes the lowest conformal weight, and
$c = \frac{k \dim \g}{k + h^{\vee}}$ is the central charge.

\subsubsection{Theta Functions and the Dedekind Eta Function}\label{theta functions}
\newcommand{\thchar}[3]{\vartheta\!\begin{bmatrix}#1\\#2\end{bmatrix}\!(#3)}
\newcommand{\q}{e^{2\pi i \tau}}
Dedekind eta function $\eta(\tau)$ is 
\be
\eta(\tau)=q^{1/24}\prod_{n=1}^{\infty}(1-q^n).
\ee
Theta functions with characteristics $(a, b)\in \R^2$ are defined as
\be
\vartheta\!\begin{bmatrix}a\\ b\end{bmatrix}\!(z,\tau)
=\sum_{n\in\mathbb Z}
\exp\!\big(\pi i (n+a)^2\tau + 2\pi i (n+a)(z+b)\big).
\ee
The four standard $\theta_{ij} (i,j\in\{0,1\})$ as half-characteristics
\be
\theta_{00}(z,\tau)=\vartheta\!\begin{bmatrix}0\\0\end{bmatrix}\!(z,\tau),\qquad
\theta_{01}(z,\tau)=\vartheta\!\begin{bmatrix}0\\\tfrac12\end{bmatrix}\!(z,\tau),
\ee
\be
\theta_{10}(z,\tau)=\vartheta\!\begin{bmatrix}\tfrac12\\0\end{bmatrix}\!(z,\tau),\qquad \nonumber
\theta_{11}(z,\tau)=\vartheta\!\begin{bmatrix}\tfrac12\\\tfrac12\end{bmatrix}\!(z,\tau).
\ee
The product expansions are, for \(q=e^{2\pi i\tau},\ \xi=2\pi i z\).
\begin{align}
\theta_{11}(\xi,\tau)
&= q^{1/8}\big(e^{\xi/2}-e^{-\xi/2}\big)\prod_{n=1}^{\infty}(1-q^n)\,(1-q^n e^{\xi})\,(1-q^{n}e^{-\xi}),  \\
\theta_{10}(\xi,\tau) 
&= q^{1/8}\big(e^{\xi/2}+e^{-\xi/2}\big)\prod_{n=1}^{\infty}(1-q^n)\,(1+q^n e^{\xi})\,(1+q^{n}e^{-\xi}), \nonumber \\
\theta_{00}(\xi,\tau) 
&= \prod_{n=1}^{\infty}(1-q^n)\,(1+q^{\,n-\tfrac12}e^{\xi})\,(1+q^{\,n-\tfrac12}e^{-\xi}), \nonumber \\
\theta_{01}(\xi,\tau) 
&= \prod_{n=1}^{\infty}(1-q^n)\,(1-q^{\,n-\tfrac12}e^{\xi})\,(1-q^{\,n-\tfrac12}e^{-\xi}). \nonumber 
\end{align}

\subsubsection{Level one positive energy representations of $L\mathrm{Spin}(2l)$}
We provide explicit formulas for the level-one positive-energy representations of $L\mathrm{Spin}(2l)$, which will be used repeatedly throughout the paper.
The space $\per^1(L\mathrm{Spin}(2l))$ is spanned by four irreducible representations: $S^+,S^-,S_+,S_-$. 
For further details, see \cite{Brylinski1990RepresentationsOL,Liu1995OnMI}.
\begin{proposition}\label{characters are theta functions}
    The modified super characters of level-one representations can be expressed as products of Jacobi theta functions:
\begin{align}
 \chi_{S^+-S^-}=\prod_{j=1}^l \frac{\theta_{11}(\alpha_j,\tau)}{\eta(\tau)},\quad    \chi_{S_+-S_-}=\prod_{j=1}^l \frac{\theta_{01}(\alpha_j,\tau)}{\eta(\tau)}\\
  \chi_{S^++S^-}=\prod_{j=1}^l \frac{\theta_{10}(\alpha_j,\tau)}{\eta(\tau)},\quad  \chi_{S_++S_-}=\prod_{j=1}^l \frac{\theta_{00}(\alpha_j,\tau)}{\eta(\tau)}\nonumber.
\end{align}
Here $\{\pm \alpha_j\}$ denote the roots of $\mathrm{spin}(2l)$. These four virtual representations correspond naturally to the four spin structures on the elliptic curve $\Sigma_{\tau}$. For convenience, we use the notation
\be
S_{11}=S^+-S^-,\qquad S_{10}=S^++S^-,\qquad S_{01}=S_+-S_-,\qquad S_{00}=S_++S_-.
\ee
We denote by $\chi_{S_{ij}}$ the modified super character of $S_{ij}$ for $i,j \in \{0,1\}$.
\end{proposition}

Here we only compute the $G=\mathrm{Spin}(2l)$ and level one.
The modified Kac-Weyl character can generally be expressed as a linear combination of higher theta functions. 
For computations and the Kac-Weyl character formula, we refer the reader to \cite{Kac1990InfiniteDL, pressley_loop_1988}.

\section{Equivariant Twisted Chern and Bismut--Chern Characters}\label{sec:twisted bismut}
In \cite{Han2014ExoticTE}, the twisted Bismut--Chern character is introduced based on the Bismut--Chern character in \cite{bismut_index_1985}. 
The purpose of this section is to introduce the equivariant twisted Bismut--Chern character, which provides a finite-dimensional geometric model for the elliptic Bismut--Chern character. 
Later, we will set $M=\LX$ and let $\mathcal{G}$ denote the lifting gerbe. 

Let $M$ be a $\cir$-manifold
and
 $\mathcal{G}$ be a $\cir$-equivariant gerbe with connection $\nabla^{B}$ whose curving is $H$ and whose moment map vanishes.
 Let $(\mathcal{L}^B,\nabla^{\mathcal{L}^B})=\mathcal{L}(\mathcal{G},\nabla^B)$ denote the transgression line bundle over $LM$.
The $\cir$-equivariant twisted Bismut--Chern character $BCh_{\cir, \mathcal{G}}$ is defined such that the following diagram commutes:
\be
    \begin{tikzcd}
    K^0_{\cir}(M,\mathcal{G}) \arrow[rr,"BCh_{\cir, \mathcal{G}}"] \arrow[dr,"Ch_{\cir, \mathcal{G}}"] &  &  h_{S^1\times \cir}^{2*}(LM,(\mathcal{L}^B,\nabla^{\mathcal{L}^B},\overline{H})) \arrow[dl,"i^*"] \\
        &  H_{\cir}^{2*}(M,H) &
    \end{tikzcd}
\ee
{\bf This section is organized as follows.}
 In Subsection \ref{subsec: bismut chern}, we review the geometry of extended loop bundles and the Bismut-Chern character. 
    In Subsection \ref{subsec: completed Periodic Exotic Twisted equivariant Cohomology of LM}, we describe the geometry of $\cir$-equivariant gerbe $(\mathcal{G},\nabla^B)$, its
 transgressed line bundle  $\mathcal{L}(\mathcal{G},\nabla^B)=(\mathcal{L}^B,\nabla^{\mathcal{L}^B})\to LM$ and establish the completed periodic exotic twisted $(S^1\times \cir)$-equivariant cohomology on $LM$.
    In Subsection \ref{subsec: Eq twisted BCH}, we define the equivariant twisted Bismut-Chern character and restrictions.

%--------------------------------------------------------------------------------------------------------
\subsection{Bismut-Chern Character of Extended Loop Bundles}\label{subsec: bismut chern}
We recall the construction of the Bismut–Chern character from \cite{bismut_index_1985}, which can be viewed as the characteristic form of the extended loop bundle over $\LX$.

To be compatible with further discussions, here we pursue a principal bundle approach.
Start from the principal $G$-bundle $\pi:P\to X$ with the connection $A$,
the loop bundle $LP\to \LX$ is a principal $LG$-bundle. 
Let $K$ be the generator of rotation $\cir$ on $\LX$, $K'$ be its lifting on $LP$.
    
We first introduce some basic operations to construct differential forms on loop bundles.
\begin{enumerate}
    \item The \textbf{evaluation} of of $\omega\in \Omega^*(P,\g)$ is written as $\hat{\omega} \in \Omega^*(LP,L\g)$, given $X\in T_{\gamma}LP$,
    \begin{align*}
    \hat{\omega}(X_1\cdots X_n)_{t}:=\omega(X_1(\gamma(t)),\cdots X_n(\gamma(t)) )\in \g.
    \end{align*}
   Set $\omega_t\in \Omega^*(LP,\g)$ as the value of $\hat{\omega}$ at time $t$.
    \item 
    The \textbf{average} of $\omega\in \Omega^*(P,\g)$ is written as $\overline{\omega} \in \Omega^{*}(LP,\g)$, that
    \begin{align*}
    \overline{\omega}(X_1\cdots X_n)_{\gamma}:=\int_{S^1}\omega(X_1(\gamma(t))\cdots X_n(\gamma(t)))dt=\int_{S^1}\omega_t dt \in \g.
    \end{align*}
    \item
    The \textbf{transgression} of $\omega\in \Omega^*(P,\g)$ is written as $\int_{S^1}\omega\in \Omega^{*-1}(LP,\g)$, that
    \begin{align*}
    (\int_{S^1}\omega)(X_1\cdots X_{n-1})_{\gamma}:=\int_{S^1}\omega(\dot{{\gamma}},X_1(\gamma(t))\cdots X_{n-1}(\gamma(t)))dt\in \g=\iota_{K'}\overline{\omega}.
    \end{align*}
\end{enumerate}
The evaluation $\hat{A}\in \Omega^1(LP,L\g)$ is a connection on the loop bundle $LP\to \LX$, whose curvature is the evaluation $\hat{R}$.

One key observation is that the principal $LG$-bundle $LP\to \LX$ is not $\cir$-equivariant. To address this, we consider the semi-product, the extended loop group $\widehat{LG}= LG\rtimes S^1$, and the corresponding extended loop bundle $\widehat{LP}=(LP\times S^1)\to \LX$, which is a principal $\widehat{LG}$-bundle over $\LX$. 

We extend the connection $\hat{A}$ to $\widehat{A}$ on $\widehat{LP}\to \LX$ as follows:
for $(Y,b)\in T_{(r,s)}\widehat{LP}$, where $r\in LP$ and $s\in S^1$,
    \begin{align}\label{Connection on semiproduct bundle}
        \widehat{A}(Y,b)=(b,-k_{-s}\hat{A}(Y))\in  \widehat{L\g}.
    \end{align}
    Here, $k_s$ denotes the rotation of $L\g$ by $s$. 
The vector field $K''=(K',1)$ acts on the extended loop bundle $\widehat{LP}$.
\begin{proposition}
The equivariant curvature of $\widehat{A}$ with respect to $K''$, restricted to $LP\subset \widehat{LP}$ 
where $b=0$, can be written as
\begin{align}
\widehat{R}_{\cir}=\iota_{K''}\widehat{A}+\hat{R}=(1,\iota_{K'}\hat{A}+\hat{R})\in \Omega^*(LP,\widehat{L\g}).
\end{align}
\end{proposition}
For convenience, we always identify $\widehat{L\g}$ with $\g$-valued first-order differential operators, and the equivariant curvature is then
\begin{align}
\widehat{R}_{\cir}=\frac{d}{dt}-(\iota_{K}\hat{A}+\hat{R}) \in \Omega^*(LP,\widehat{L\g}).
\end{align}
Let $P_t=\mathrm{ev}_t^*(P)$ denote the evaluation of $LP$ at time $t$. The connection $A$ defines the horizontal vector field on $P$; denote by $g_t:P_0\to P_t$ the $G$-bundle morphism given by parallel transport of $A$ from time $0$ to $t$. Let $R_{t}\in \Omega^2(P_t,\g)$ denote the evaluation of the curvature at time $t$.
The Bismut--Chern character with respect to connection $A$ and representation $V$ is
\begin{align}\label{Bismut-Chern formula}
BCh_A(V)=\Tr_V\left[\sum_{n=0}^{\infty}\int_{\Delta_n}\prod_{i=1}^ng_{t_i}^*(R_{t_i})dt_i\right],
\end{align}
where $\Delta_n=\{0\leq t_1\leq \cdots \leq t_n\leq 1\}$ is the standard $n$-simplex.

The above definition actually arises from the transport equation, which is a first-order differential equation with values in differential forms on $LP$:
\begin{align}
\left(\frac{d}{dt}-(\iota_{K}\hat{A}_t+R_t)\right)\psi(t)=0, \quad \psi(0)=I.
\end{align}
Then $BCh_V[A]=\Tr_V[\psi(1)]$. 

From this, we can also consider formal differences of two representations and similarly define the super Bismut--Chern character.
This construction generalizes naturally to $\mathbb{Z}_2$-graded vector bundles equipped with superconnections.
\subsubsection*{\textbf{$\cir$-equivariant Euler class of $\LX$}}
Let $X$ be a spin manifold equipped with a Riemannian metric $g$.
Denote by $P_{\mathrm{Spin}(X)}$ the spin frame bundle, and let $A$ be the spin connection induced from the Levi-Civita connection.
 Let $R$ denote its curvature.
Let $\mathcal{S}_X=\Delta^{+}\oplus \Delta^{-}$ be the spinor bundle over $X$.
$Ch_s(\mathcal{S}_X)$ is the super Chern character of spinor bundles.

The super Bismut-Chern character $BCh_s(\mathcal{S}_X)$ given by the super-trace is a $(d+\iota_{K})$-closed even form on $\LX$ that 
\begin{enumerate}
    \item The degree zero is super-trace of the holonomy of spinors $\Tr_s[hol_{\mathcal{S}_X}]$.
    \item Restricted to the constant loop space $X$, it is $Ch_s(\mathcal{S}_X)$.
\end{enumerate}
By localization,
\begin{align}
  \int_{\LX}BCh_s(\mathcal{S}_X)\to \int_X \hat{A}(X)\wedge Ch_s(\mathcal{S}_X)=\int_X e(X).
\end{align}
Then it could be viewed as the 
$\cir$-equivariant Euler class of $\LX$.

%--------------------------------------------------------------------------------------------------------

%--------------------------------------------------------------------------------------------------------

\subsection{Exotic Equivariant Twisted Cohomology}\label{subsec: completed Periodic Exotic Twisted equivariant Cohomology of LM}
The free loop space $LM$ has an intrinsic circular symmetry arising from loop rotation, denoted by $S^1$. 
Let $\partial_y$ denote the vector field generating these rotations.
The space $M$ itself admits a $\cir$-action, which induces an action on $LM$ by acting on each point of the loop. Let $\partial_x$ denote the vector field on $M$ generating the $\cir$-action. We also use the notation $\partial_x$ for the induced vector field on $LM$.
\subsubsection{Equivariant gerbe with vanishing moment map}
\begin{definition}
   Suppose that $\{U_{\alpha}\}$ is a maximal open cover of $M$ with the property that 
$H^i(U_{\alpha_I})=0$ for $i=2,3$ where $U_{\alpha_I}=\bigcup_{i \in I}U_{\alpha_i},\abs{I}<\infty$.
We call such an open cover a Brylinski open cover of $M$.
For convenience, we assume the existence of $\cir$-invariant Brylinski cover $\{U_{\alpha}\}$ and pursue everything locally.
\end{definition}

General concepts of equivariant gerbes can see \cite{Meinrenken2002TheBG}.
 Here we only deal with the $\cir$-equivariant gerbe with the vanishing moment map.
 This condition is not necessary but simplifies the construction.
Gerbe with connection
 $(\mathcal{G},\nabla^B)$ is given by following geometric data.
\begin{itemize}
    \item Global $H\in \Omega^3(M)$ that $\frac{1}{2\pi i}H $ has integral periods.
    \item On each $U_{\alpha}$, there's a $2$-form $B_{\alpha}\in \Omega^2(U_{\alpha},i\R)$ such that
    $
    dB_{\alpha}=H|_{U_{\alpha}}.
    $
    \item On each double intersection $U_{\alpha\beta}$, there's a line bundle $L_{\alpha\beta}\to U_{\alpha\beta}$ with connection $\nabla^{L}_{\alpha\beta}$
    whose equivariant curvature 
    $
    F^L_{\alpha\beta}=B_{\beta}-B_{\alpha}.
    $
    \item On each triple intersection $U_{\alpha\beta\gamma}$, there's an isomorphism of line bundles
    $
    \mu_{\alpha\beta\gamma}:L_{\alpha\beta}\otimes L_{\beta\gamma}\to L_{\alpha\gamma}
    $
    satisfying the cocycle condition on quadruple intersections.
\end{itemize}
The gerbe $(\mathcal{G},\nabla^B)$ is $\cir$-equivariant
if all of them are $\cir$-equivariant.
Its moment map vanishes if both $\iota_{\partial_x}B_{\alpha}$ and the moment map of $\nabla^L_{\alpha\beta}$ vanishes.

When the line bundle $L_{\alpha\beta}$ is trivialized,
the connection $\nabla_{\alpha\beta}^L=d+A_{\alpha\beta}$ where $A_{\alpha\beta}\in \Omega^1(U_{\alpha\beta},i\R)$ that 
\begin{align}
    dA_{\alpha\beta}=B_{\beta}-B_{\alpha}.
\end{align}
Then the isomorphism is given by $z_{\alpha\beta \gamma}:U_{\alpha\beta\gamma}\to S^1$ that
$
(z_{\alpha\beta\gamma},A_{\alpha\beta},B_{\alpha}) 
$
is a Deligne cocycle (see \cite{Brylinski1993}).
\subsubsection{Transgression line bundle}
The transgression line bundle $(\mathcal{L}^B,\nabla^{\mathcal{L}^B})$ is described as follows.
The transition function is given by
\begin{align}
    h_{\alpha\beta}=e^{\int_{S^1}A_{\alpha\beta}}:LU_{\alpha}\cap LU_{\beta}\to S^1,
\end{align}
and the connection on $LU_{\alpha}$ is
\begin{align}
     \nabla^{\mathcal{L}^B}|_{LU_{\alpha}}=d+\iota_{\partial_y}\overline{B_{\alpha}}=d+\int_{S^1}B_{\alpha}.
\end{align}
The curvature is given by
$
R^{\mathcal{L}^B}=\int_{S^1}H.
$
The moment map $\mu^{\mathcal{L}^B}$ vanishes for both the $S^1$ and $\cir$ actions.
\subsubsection{Setting up the cohomology}
We denote 
\be H^*(BS^1)=\CC[v],\quad H^*(B\cir)=\CC[u],\quad H^*(B(S^1\times \cir))=\CC[u,v],\ee 
where $u$ and $v$ are generators of degree $2$.
Let $\CC(u,v)$ denote the ring of Laurent polynomials in two variables.
\begin{definition}\label{Eq twisted coho}
The completed-periodic $\cir$-equivariant twisted cohomology on $M$ is given by
\begin{align}
h^*_{\cir}(M,H)=H^*(\Omega^*(M)^{\cir}[u,u^{-1}],d-u\iota_{\partial_x}-H).
\end{align}
\end{definition}
For the exotic version, we define 
\begin{align}\label{D_H}
\mathcal{D}_{\overline{H}}=\nabla^{\mathcal{L}^B}+v\iota_{\partial_y}-u\iota_{\partial_x}-\overline{H}
\end{align}
as an operator on $\Omega^*(LM,\mathcal{L}^B)(u,v)$.
Locally, it can be expressed as the conjugation
\begin{align}
  \mathcal{D}_{\overline{H}}|_{LU_{\alpha}}=e^{\overline{B_{\alpha}}}(d+v\iota_{\partial_y}-u\iota_{\partial_x})e^{\overline{-B_{\alpha}}}.
\end{align}
Then $(\mathcal{D}^{\overline{H}})^2=L_{v\partial_y-u\partial_x}$. It vanishes on $T^2$-invariant part $\Omega^*(LM,\mathcal{L}^B)^{T^2}(u,v)$.
\begin{definition}\label{exotic T^2 twisted}
  The exotic twisted $T^2$-equivariant cohomology is 
\begin{align}
    h_{S^1\times \cir}^*(LM,(\mathcal{L}^B,\nabla^{\mathcal{L}^B},\overline{H})):=H(\Omega^*(LM,\mathcal{L}^B)^{T^2}(u,v),D_{\overline{H}}).
\end{align}  
\end{definition}
Let \(i:M\to LM\) be the inclusion. It induces the restriction of the cohomology to the constant loop space
\begin{align}
i^*:h_{S^1\times \cir}^*(LM,(\mathcal{L}^B,\nabla^{\mathcal{L}^B},\overline{H})) \to h^*_{\cir}(M,H).
\end{align}
\begin{remark}\label{cohomology with nonvanishing moment map}
When the moment map of \((\mathcal{G},\nabla^B)\) does not vanish, we can still define the cohomology in the same way. We replace \(H\) with its equivariant extension \(H_{\cir}=H+u\mu\), which is a \((d-u\iota_{\partial_x})\)-closed \(3\)-form in the Cartan model representing the \(\cir\)-equivariant Dixmier-Douady class.
\end{remark}

%--------------------------------------------------------------------------------------------------------
\subsection{Equivariant Twisted Bismut-Chern Character}\label{subsec: Eq twisted BCH}
We define the $\cir$-equivariant twisted Bismut-Chern character, which is the loop space refinement of the $\cir$-equivariant twisted Chern character.
The $\cir$-equivariant gerbe module with connection $(\mathcal{E},\nabla^{\mathcal{E}})$ is defined as follows:
$E_{\alpha}\to U_{\alpha}$ is a $\cir$-equivariant $n$-dimensional Hermitian vector bundle equipped with the Hermitian connection $\nabla_{\alpha}$, along with a $\cir$-invariant isomorphism that preserves both metrics and connections:
\begin{align}
\Phi_{\alpha\beta}:(L_{\alpha\beta},\nabla_{\alpha\beta})\otimes(E_{\beta},\nabla_{\beta})\to (E_{\alpha},\nabla_{\alpha}),
\end{align}
which is compatible on triple overlaps.

When trivialized, the connection 
$
\nabla_{\alpha}=d+A_{\alpha}
$
where $A_{\alpha}\in \Omega^1(U_{\alpha},u(n))$ is a matrix $1$-form.
Let $R_{\cir,\alpha}=R_{\alpha}+u\mu_{\alpha}$ be its equivariant curvature,
where $R_{\alpha}=dA_{\alpha}+A_{\alpha}\wedge A_{\alpha}$ is the curvature and $\mu_{\alpha}=-\iota_{\partial_x}A_{\alpha}$ is the moment map.
\subsubsection{Equivariant Twisted Chern Character}\label{eq-twisted chern}
Follow by $F_{\alpha\beta}^L=B_{\beta}-B_{\alpha}$, 
\begin{align}\label{phiab compa}
\Phi_{\alpha\beta}^{-1}(B_{\alpha}I+R_{\cir,\alpha})\Phi_{\alpha\beta}=B_{\beta}I+R_{\cir,\beta}.
\end{align}
Then 
$e^{B_{\alpha}}tr[\exp{(R_{\cir,\alpha})}]$ 
can be patched together to be
the $\cir$-equivariant twisted Chern character $Ch_{\cir,\mathcal{G}}(\mathcal{E},\nabla^{\mathcal{E}})$.
It defines 
    \begin{align}
    Ch_{\cir,\mathcal{G}}:K_{\cir}^0(\mathcal{G},H) \to  H_{\cir}^{2*}(M,H).
    \end{align}
\subsubsection{Equivariant Twisted Bismut-Chern Character}\label{eq-twisted bismut D operator}
When $E_{\alpha}\to U_{\alpha}$ is trivialized, $\hat{A}_{\alpha}\in \Omega^1(LU_{\alpha},Lu(n))$ represents the evaluation of the connection form $A_{\alpha}$ on $LU_{\alpha}$. 
The equivariant curvature is given by $\hat{R}_{\cir,\alpha}=\hat{R}_{\alpha}+u\hat{\mu}_{\alpha}$, where $\hat{R}_{\alpha}\in \Omega^2(LU_{\alpha},Lu(n))$ is the evaluation of the curvature $R_{\alpha}$ on $LU_{\alpha}$, 
and $\hat{\mu}_{\alpha}=-\iota_{\partial_x}\hat{A}_{\alpha}$ is the evaluation of the moment map. 
Additionally, $\hat{B}_{\alpha}\in \Omega^2(LU_{\alpha})$ represents the evaluation of the $B$-field $B_{\alpha}$.

Consider the first-order matrix differential operator in the $y$-direction with values in $\Omega^*(LU_{\alpha})(u,v)$.
\begin{align}
\mathcal{D}_{\alpha}=\frac{d}{dy}-v\iota_{\partial_y}\hat{A}_{\alpha}-(\hat{R}_{\cir,\alpha}+\hat{B}_{\alpha})
=\frac{d}{dy}-[(i_{v\partial_y-u\partial_x}\hat{A}_{\alpha}+\hat{R}_{\alpha})+\hat{B}_{\alpha}].
\end{align}
Its transport equation is, for $\psi\in \CC^{\infty}([0,1],\CC^n)\otimes \Omega^{2*}(LU_{\alpha})(u,v)$,
\begin{align}
\mathcal{D}_{\alpha}\psi=0, \quad \psi(0)=1.
\end{align}
The equivariant twisted Bismut-Chern character is locally given by
$
BCh_{\cir,\mathcal{G}, \alpha}=\tr[\psi(1)].
$

Let $\Delta_n=\{(t_1,\cdots,t_n)|0\leq t_1\leq \cdots \leq t_n\leq 1\}$ be the standard $n$-simplex. We solve the transport equation by iterated integrals.
Set $Y_{\alpha}=i_{v\partial_y-u\partial_x}\hat{A}_{\alpha}+\hat{R}_{\alpha}$,   
\begin{align}
BCh_{\cir,\mathcal{G}, \alpha}=\Tr[\sum_{n=0}^{\infty}\int_{\Delta_n} \prod_{i=1}^n(Y_{\alpha}+\hat{B}_{\alpha})_{t_i}dt_i]=e^{(\overline{B_{\alpha}})}tr(\sum_{n=0}^{\infty}\int_{\Delta_n} \prod_{i=1}^n (Y_{\alpha})_{t_i}dt_i).
\end{align}
The last equality follows from the fact that $(\hat{B}_{\alpha})_t$ is central commutative with everything and
\begin{align}
\sum_{n=0}^{\infty}\int_{\Delta_n} \prod_{i=1}^n(\hat{B}_{\alpha})_{t_i}dt_i=\sum_{n=0}^{\infty}\frac{1}{n!}(\int_0^1(\hat{B}_{\alpha})_t dt)^n=e^{\overline{B_{\alpha}}}.
\end{align}
From \eqref{phiab compa}, 
\be
\Phi_{\alpha\beta}^{-1}(Y_{\alpha}+\hat{B}_{\alpha})_t\Phi_{\alpha\beta}=(Y_{\beta}+\hat{B}_{\beta})_t.
\ee
Then we have
$BCh_{\cir,\mathcal{G}, \alpha}=BCh_{\cir,\mathcal{G}, \beta}$ on the overlap $LU_{\alpha}\cap LU_{\beta}$.
It globally defines the $\cir$-equivariant twisted Bismut-Chern character $BCh_{\cir, \mathcal{G}}(\mathcal{E},\nabla^{\mathcal{E}})$.

Also, it is a cocycle representative of  $h_{S^1\times \cir}^*(LM,(\mathcal{L}^B,\nabla^{\mathcal{L}^B},\overline{H}))$.
\begin{proposition}\label{D_H closed}
$BCh_{\cir, \mathcal{G}}(\mathcal{E},\nabla^{\mathcal{E}})$ is $\mathcal{D}_{\overline{H}}$-closed.
\begin{proof}
Recall that 
$
  \mathcal{D}_{\overline{H}}|_{LU_{\alpha}}=e^{-\overline{B_{\alpha}}}(d+v\iota_{\partial_y}-u\iota_{\partial_x})e^{\overline{B_{\alpha}}}.
$.
It's suffice to prove the local non-twisted version.
\begin{align}
(d+v\iota_{\partial_y}-u\iota_{\partial_x})tr(\sum_{n=0}^{\infty}\int_{\Delta_n} \prod_{i=1}^n (Y_{\alpha})_{t_i}dt_i)=0.
\end{align}
This follows from the Bianchi identity. For details, one can refer back to the transport equation
and simply replace the curvature with the equivariant curvature in the proof of \cite[Theorem 3.9]{bismut_index_1985}.
\end{proof}
\end{proposition}
\begin{remark}
As before, when the moment map of the gerbe does not vanish, we can still define the equivariant twisted Bismut–Chern character by simply replacing $H$ with $H_{\cir}$ everywhere.
\end{remark}
\subsubsection{Restriction to constant loop space $M$}
When restricted to the constant loop space, consider the operator
\begin{align}
\mathcal{D}_{\alpha}=\frac{d}{dy}-(\hat{R}_{\cir,\alpha}+\hat{B}_{\alpha})
\end{align}
where $(\hat{R}_{\cir,\alpha}+\hat{B}_{\alpha})$ is constant.
The solution of the transport equation is given by the exponential
\begin{align}
BCh_{\cir, \mathcal{G}}(\mathcal{E},\nabla^{\mathcal{E}})|_M=\exp(B_{\alpha})\mathrm{tr}[\exp(R_{\cir,\alpha})]=Ch_{\cir,\mathcal{G}}(\mathcal{E},\nabla^{\mathcal{E}}).
\end{align}
Thus, we obtain the following commutative diagram:
\begin{equation}
    \begin{tikzcd}
    K^0_{\cir}(M,\mathcal{G}) \arrow[rr,"BCh_{\cir, \mathcal{G}}"] \arrow[dr,"Ch_{\cir,\mathcal{G}}"] &  &  h_{S^1\times \cir}^{2*}(LM,(\mathcal{L}^B,\nabla^{\mathcal{L}^B},\overline{H})) \arrow[dl,"i^*"] \\
        &  h_{\cir}^{2*}(M,H)[v,v^{-1}] &
    \end{tikzcd}
\end{equation}
\subsubsection{Restriction to fixed-point $M^{\cir}$ and its loop space}
The torus $T^2= (S^1 \times \cir)$ acts naturally on $LM$. The fixed point set of this torus action is the $\cir$-fixed point set of $M$.
We assume $M^{\cir}$ is a submanifold of $M$.
Let $j: M^{\cir} \to M$ denote the inclusion map, which induces $Lj: LM^{\cir} \to LM$.
The inclusion of constant loops of $M^{\cir}$ is denoted by $i_{\cir}: M^{\cir} \to LM^{\cir}$.
\begin{equation}
    \begin{tikzcd}
                                    & M \arrow[rd,"i"] &\\
    M^{\cir}\arrow[ru,"j"] \arrow[rd,"i_{\cir}"] \arrow[rr]  &      & LM\\
       & LM^{\cir}\arrow[ru,"Lj"] & 
\end{tikzcd}
\end{equation}
We assume that the gerbe $(\mathcal{G},\nabla^{B})$ restricts trivially over $M^{\cir}$. 
More precisely, both $B_{\alpha}$ on $U_{\alpha}^{\cir}$ and $A_{\alpha\beta}$ on $U_{\alpha\beta}^{\cir}$ vanish.
Consequently, the curving satisfies $H = 0$.
Over $M^{\cir}$, the twisting and equivariance both disappear, reducing to ordinary de Rham cohomology:
\begin{align}
j^*: h_{\cir}^*(M,H) \to h^*(M^{\cir})[u,u^{-1}].
\end{align}
On the other hand, the cohomology of $LM^{\cir}$ is the standard completed-periodic $S^1$-equivariant cohomology (untwisted):
\begin{align}
h_{S^1}^*(LM^{\cir}) := H^*(\Omega^*(LM^{\cir})^{S^1}[v,v^{-1}], d + v\iota_{\partial_y}).
\end{align}
When restricted to $M^{\cir}$, gerbe modules become ordinary vector bundles.
\begin{equation}
    \begin{tikzcd}
        &  & h_{S^1}^{2*}(LM^{\cir})\arrow[dd,"i_{\cir}^*"] \\
    K^0_{\cir}(M,\mathcal{G})\arrow[r,"j^*"] & K^0(M^{\cir}) \arrow[ur,"BCh"]\arrow[dr,"Ch"]\\
        &   & H^{2*}(M^{\cir})[v,v^{-1}]
    \end{tikzcd}
\end{equation}
We have $BCh|_{M^{\cir}}=BCh\circ j^*$ and $Ch|_{M^{\cir}}=Ch\circ j^*$.
The loop space $LM$ admits a natural $(S^1\times \cir)$-action.
There are two restriction directions: one to $M$, and another to $LM^{\cir}$, and finally to $M^{\cir}$.

The following commutative diagram summarizes all constructions in this section and the relationships between them.
\begin{equation}\label{M and LM diagram}
\begin{tikzcd}
     & h_{S^1\times \cir}^{2*}(LM,(\mathcal{L}^B,\nabla^{\mathcal{L}^B},\overline{H}))  \arrow[ld, "Lj^*"'] \arrow[dr,"i^*"] & \\
h_{\cir}^{2*}(LM^{\cir})(u) \arrow[dr,"i_{\cir}",swap]& K^0_{\cir}(M,\mathcal{G}) \arrow[dashed,l, "BCh|_{M^{\cir}}",swap] \arrow[dashed,u, "BCh_{\cir, \mathcal{G}}"] \arrow[dashed,d, "Ch|_{M^{\cir}}"'] \arrow[dashed,r, "Ch_{\cir,\mathcal{G}}"] &  h_{S^1}^{2*}(M,H)(v) \arrow[ld, "j^*"'swap] \\
     & h^{2*}(M^{\cir})(u,v)   &
\end{tikzcd}
\end{equation}
\begin{remark}
    In \cite{Han2014ExoticTE}, it is shown that the localization map $j^*$ is an isomorphism when $M$ is a strongly regular $\cir$-manifold.
On the other hand, $LM$ is a strongly regular $S^1$-manifold, since the fixed-point set $M$ admits an $S^1$-invariant tubular neighborhood in $LM$.
Therefore, $i_{\cir}^*$ is also an isomorphism.
Similar arguments should apply in the equivariant twisted case, showing that $i^*$ and $Lj^*$ are isomorphisms.
\end{remark}

%%%%%%%%%%%%%%%%%%%%%%%%%%%%%%%%%%%%%%%%%%%%%%%%%%%%%%%

\section{Chern Characters on \texorpdfstring{$\LX$}{LX}}\label{sec: Elliptic Loop Chern Character on LX}
In this section, we introduce two Chern characters on $\LX$ associated with positive-energy representations of loop groups: the $q$-graded Bismut--Chern character and the elliptic Chern character.
The elliptic Chern character is defined as the $\disc$-deformed equivariant twisted Chern character of the gerbe module associated with a positive-energy representation.
It can be viewed as an elliptic Chern--Weil theory for loop groups.

On $X$, the elliptic Chern character is defined by the $q$-graded Chern character, which assembles the ordinary Chern characters of the energy subspaces. Under the energy decomposition, we have the map
\begin{align}
\per^k(LG)\otimes \CC \to \mathrm{Rep}(G)\otimes \CC[[q]].
\end{align}
The $q$-graded Chern character appears naturally in twisted genera \cite{Brylinski1990RepresentationsOL,Liu1996MODULARFA,Miller1989EllipticCW}, 
and plays a key role in understanding modular invariance.
This construction extends naturally to the loop space, giving rise to the corresponding $q$-graded Bismut-Chern character.

Given a principal $G$-bundle $P \to X$, the looped principal $LG$-bundle $LP \to \LX$ is naturally induced.
The presence of a string $G$-structure, \ie the absence of string $G$-anomaly, 
allows a positive energy representation of $LG$ to define an admissible Virasoro-equivariant vector bundle on $\LX$ in \cite{Brylinski1990RepresentationsOL}. 
However, when the string $G$ anomaly exists, the lifting gerbe $\mathcal{G}_P$ on $\LX$ obstructs such an existence.
Instead, for $\mathcal{H} \in \per^k(LG)$, one can construct the associated infinite-dimensional gerbe module $(\mathcal{E}, \nabla^{\mathcal{E}})$ for the $k$-fold tensor power $(\mathcal{G}_P, \nabla^B)^{\otimes k}$, yielding a map
\be
\per^k(LG) \to K_{\cir}(\LX, k\mathcal{G}_P).
\ee
A naive approach to defining its Chern character would be to take $M=\LX$ and define the $\cir$-equivariant twisted Chern character as in \cref{eq-twisted chern}. 
However, the trace-class condition required in \cite{Mathai2002ChernCI} fails in this infinite-dimensional setting.
 This divergence is already apparent in representation theory: the formal characters of positive energy representations diverge when $q \in S^1$.

As shown in \cite{pressley_loop_1988}, the positive energy representation extends to an action of the punctured disk $\CC_{<1}^*$, where the character converges for $q \in \disc$.
To define the equivariant twisted Chern character properly, we replace $S^1$ with $\disc$ and define the character as differential forms over complex coefficients.

A key challenge is ensuring convergence. The convergence of the degree-zero part can be established using Floquet theory in \cite{FrenkelOrbital}.
For higher degrees, we address this using energy estimates. Roughly, under the energy decomposition, the rotation acts as $q^n$ on the $n$-energy subspace, while other geometric terms grow only polynomially.

It is worth noting that, since the gerbe module associated with positive energy representations is infinite-dimensional, there are analytical subtleties concerning \textbf{smooth} structures. For more details on Fock representations and the smooth spinor bundle on $\LX$, see \cite{Kristel2020SmoothFB}.
The subtlety arises from the fact that positive energy representations of loop groups are only smooth in a dense subspace.
For rigorous treatments of infinite-dimensional representations, we refer to \cite{Goodman1984StructureAU}.
Here, we discuss Chern characters at the level of differential forms, without delving into the smoothness of transition functions.

As an application, although the Dirac-Ramond operator on $\LX$ has not yet been rigorously constructed, 
we can properly define the $\disc$-equivariant $\hat{A}$-genus of $\LX$, thereby formally establishing the Kirillov formula of the $\disc$-equivariant index theorem on $\LX$. When localized to the constant loop space $X$, it recovers the Witten genus.

{\bf This section is organized as follows.}
 In Subsection \ref{subsec: Elliptic Character and q-graded Bismut-Chern Character}, we define the $q$-graded Chern and Bismut-Chern character.
In Subsection \ref{subsec:geometry of lifting gerbe and cir-equivariance}, we describe the geometry of the lifting gerbe and the $\cir$-equivariance.
In Subsection \ref{subsec:Elliptic loop Chern Character}, 
 we define the elliptic Chern character on $\LX$ as the $\disc$-equivariant twisted Chern character.
 In Subsection \ref{Convergence of Elliptic Loop Chern Character}, we deal with the convergence of elliptic Chern character.
In Subsection \ref{subsec: Trivialized Elliptic Loop Chern Character by String G-structure}, we consider the geometric trivialization of Chern-Simons $2$-gerbe and its transgression.
The formula of trivialized elliptic Chern character is given.
In Subsection \ref{subsec: Spinor Bundle and Equivariant A-hat-genus on LX},
we propose the $\disc$-equivariant $\hat{A}$-genus of $\LX$.
\subsection{\texorpdfstring{$q$}{q}-graded Chern and Bismut-Chern Character}\label{subsec: Elliptic Character and q-graded Bismut-Chern Character}
The $q$-graded Chern character of positive energy representations arises naturally in the study of twisted genera.
Given $\mathcal{H} \in \per^k(LG)$, the energy decomposition yields
\begin{align}
\mathcal{H} = \bigoplus_{n \geq 0} \mathcal{H}_n,
\end{align}
where $\mathcal{H}_n$ is the energy-$n$ subspace. There is an associated $q$-series of graded vector bundles:
\begin{align}
\Psi(P, \mathcal{H}) = \sum_{n \geq 0} (P \times_G \mathcal{H}_n) \, q^n \in K(X)[[q]].
\end{align}
\begin{definition}
The $q$-graded Chern character is defined by
\begin{align}
Ch_A(\mathcal{H})_q = \sum_{n \geq 0} Ch_{\mathcal{H}_n}(P, A) \, q^n,
\end{align}
where $Ch_{\mathcal{H}_n}(P, A)$ is the Chern character of the associated vector bundle $(P, A) \times_G \mathcal{H}_n$.
Similarly, the $q$-graded Bismut-Chern character is 
\begin{align}
BCh_A(\mathcal{H})_q = \sum_{n \geq 0} BCh_{\mathcal{H}_n}(P, A) \, q^n.
\end{align}
\end{definition}

\subsection{Geometry of the Lifting Gerbe}\label{subsec:geometry of lifting gerbe and cir-equivariance}
We describe the geometry of the lifting gerbe $(\mathcal{G}_P,\nabla^B)$ of the loop bundle $LP \to \LX$ and establish its $\cir$-equivariant structure. 
Our construction of the lifting gerbe follows \cite{Gomi2001ConnectionsAC}.
Many of the notations and formulas introduced in this subsection will be used repeatedly throughout the subsequent sections.

Let $P \to X$ be a principal $G$-bundle equipped with a connection $A$. Locally, the bundle is described by:
\begin{itemize}
    \item Transition functions $g_{\alpha\beta}: U_\alpha \cap U_\beta \to G$,
    \item Connection forms $A_\alpha \in \Omega^1(U_\alpha, \g)$, satisfying
    \begin{align}
      A_\beta = \mathrm{Ad}_{g_{\alpha\beta}^{-1}}(A_\alpha) + g_{\alpha\beta}^*\mu = g_{\alpha\beta}^{-1}A_\alpha g_{\alpha\beta} + g_{\alpha\beta}^{-1}dg_{\alpha\beta}.
    \end{align}
\end{itemize}
The Chern-Simons $3$-form is given by:
\begin{align}
\CCS(A_\alpha) = \langle A_\alpha \wedge dA_\alpha \rangle + \frac{2}{3}\langle A_\alpha \wedge A_\alpha \wedge A_\alpha \rangle \in \Omega^3(U_\alpha).
\end{align}
The local data for the principal $LG$-bundle $(LP, \hat{A}) \to \LX$ are $\{LU_\alpha, Lg_{\alpha\beta}, \hat{A}_\alpha\}$, where $\hat{A}_\alpha \in \Omega^1(LU_\alpha, L\g)$ is the evaluation of $A_\alpha$. Let $K$ denote the generator of rotations on $\LX$. 
The geometric data describing the lifting gerbe are:
\begin{itemize}
    \item $c_\alpha = \int_{S^1} \langle \iota_K \hat{A}_\alpha \wedge \hat{A}_\alpha \rangle \, dt \in \Omega^1(LU_\alpha)$,
    \item $\Psi_\alpha = \int_{S^1} \langle A_\alpha \wedge R_\alpha \rangle - \frac{1}{2} \omega(\hat{A}_\alpha, \hat{A}_\alpha) \in \Omega^2(LU_\alpha)$,
    \item $\widetilde{A}_\alpha = \hat{A}_\alpha + c_\alpha \cc \in \Omega^1(LU_\alpha, \widetilde{L\g})$.
\end{itemize}
There is a fiberwise extension of Lemma \ref{up to exact of LG}:
\begin{lemma}[\cite{Coquereaux1989StringSO}]
\begin{align}
    \Psi_\alpha - d c_\alpha = \int_{S^1} \CCS(A_\alpha).
\end{align}
\end{lemma}
$\widetilde{A}_\alpha$ serves as a connection form for the local principal $\widetilde{LG}$-bundle on $LU_\alpha$. Its curvature is given by:
\begin{align}\label{local wideLG curvature}
    \widetilde{R}_\alpha = \hat{R}_\alpha + \frac{1}{2}[\widetilde{A}_\alpha, \widetilde{A}_\alpha] 
    = \hat{R}_\alpha + \left( \frac{1}{2}\omega(\hat{A}_\alpha, \hat{A}_\alpha) + d c_\alpha \right)\cc.
\end{align}
Applying \cite[Theorem 3.9]{Gomi2001ConnectionsAC} to the central extension of loop groups, we obtain the following theorem that describes the local data of the lifting gerbe with connection.
\begin{theorem}\label{!!!cocycle}
Let $\widetilde{Lg}_{\alpha\beta}: L(U_\alpha \cap U_\beta) \to \widetilde{LG}$ be a lifting of $Lg_{\alpha\beta}$. Define the following:
\begin{itemize}
    \item $z_{\alpha\beta\gamma} = \widetilde{Lg}_{\alpha\beta} \widetilde{Lg}_{\beta\gamma} \widetilde{Lg}_{\gamma\alpha} \in \Omega^0(L(U_\alpha \cap U_\beta \cap U_\gamma), S^1)$,
    \item $u_{\alpha\beta} = \widetilde{A}_\beta - \big[\widetilde{Ad}_{Lg_{\alpha\beta}^{-1}} \widetilde{A}_\alpha + \widetilde{Lg}_{\alpha\beta}^* \widetilde{\mu}\big] \in \Omega^1(L(U_\alpha \cap U_\beta), \R)$,
    \item $K_\alpha = \Psi_\alpha - d c_\alpha = \int_{S^1} \CCS(A_\alpha) \in \Omega^2(LU_\alpha, \R)$.
\end{itemize}
The Čech cochain $(z_{\alpha\beta\gamma}, u_{\alpha\beta}, K_\alpha)$ describes the lifting gerbe with connection $(\mathcal{G}_P, \nabla^B)$: $K_\alpha$ is the local $B$-field on $LU_\alpha$, and $u_{\alpha\beta}$ is the connection form of the line bundle $L_{\alpha\beta} \to L(U_\alpha \cap U_\beta)$.
\begin{proof}
We will verify the cocycle conditions for $(z_{\alpha\beta\gamma}, u_{\alpha\beta}, K_\alpha)$. The first condition follows from
\begin{align}
d \log z_{\alpha \beta \gamma} &= \widetilde{Ad}_{Lg_{\gamma \alpha}^{-1}} \widetilde{Ad}_{Lg_{\beta\gamma}^{-1}}\widetilde{Lg}_{\alpha\beta}^*\widetilde{\mu} 
+ \widetilde{Ad}_{Lg_{\gamma \alpha}^{-1}} \widetilde{Lg}_{\beta \gamma}^*\widetilde{\mu}
+ \widetilde{Lg}_{\gamma \alpha}^*\widetilde{\mu}.
\end{align}
We now verify the second condition,
\begin{align}\label{K diff =u}
du_{\alpha\beta} = K_\alpha - K_\beta.
\end{align}
Since $u_{\alpha\beta} = \widetilde{A}_\beta - [\widetilde{Ad}_{Lg_{\alpha\beta}^{-1}} \widetilde{A}_\alpha + \widetilde{Lg}_{\alpha\beta}^* \widetilde{\mu}]$,
\begin{align}\label{LGtilde curv rel}
\widetilde{R}_\beta = \widetilde{Ad}_{Lg_{\alpha\beta}^{-1}}(\widetilde{R}_\alpha) + d u_{\alpha\beta}.
\end{align}
Consider the central part of \cref{local wideLG curvature},
\begin{align}
du_{\alpha\beta} = \left[ \frac{1}{2}\omega(\hat{A}_\beta, \hat{A}_\beta) + d c_\beta - \frac{1}{2}\omega(\hat{A}_\alpha, \hat{A}_\alpha) - d c_\alpha \right] - \mathcal{Z}(Lg_{\alpha\beta}^{-1}, \hat{R}_\alpha).
\end{align}
Also,
\begin{align}
\int_{S^1} \langle A_\alpha \wedge R_\alpha \rangle - \langle A_\beta \wedge R_\beta \rangle = -\mathcal{Z}(Lg_{\alpha\beta}^{-1}, \hat{R}_\alpha).
\end{align}
Combined with
\begin{align}
\Psi_\alpha = \int_{S^1} \langle A_\alpha \wedge R_\alpha \rangle - \frac{1}{2}\omega(\hat{A}_\alpha, \hat{A}_\alpha),
\end{align}
we have
\begin{align}
du_{\alpha\beta} = (\Psi_\alpha - d c_\alpha) - (\Psi_\beta - d c_\beta) = K_\alpha - K_\beta.
\end{align}
\end{proof}
\end{theorem}
The curving of the lifting gerbe locally is given by $dK_{\alpha}$. This defines a global $3$-form $H$ on $\LX$, which is the transgression of $\Phi$.
\begin{remark}
The local $B$-field $K_{\alpha}$ can also be derived from the reduced splitting 
(see \cite{Murray2001HiggsFB,Waldorf2010ALS}).
Locally, the reduced splitting $L_{\alpha}:LU_{\alpha}\times \widetilde{L\g}\to \R$ is defined for $y+m\cc\in \widetilde{L\g}$ as follows:
\begin{align}
L_{\alpha}(\gamma,y+m\cc)=\int_{S^1}\Tr[A_{\alpha}\wedge y]_{\gamma}-m.
\end{align}
Thus, the local $B$-field is given by
\begin{align}
K_{\alpha}=L_{\alpha}(\widetilde{R}_{\alpha})=\int_{S^1}\langle A_{\alpha}\wedge R_{\alpha}\rangle-(\frac{1}{2}\omega(\hat{A}_{\alpha},\hat{A}_{\alpha})+dc_{\alpha})=\int_{S^1}\CCS(A_{\alpha}).
\end{align}
\end{remark}
\begin{corollary}\label{Rel of curv-B}
   From \cref{K diff =u} and \cref{LGtilde curv rel}, 
\begin{align}
(\widetilde{R}_{\beta}+K_{\beta})=\widetilde{Ad}_{Lg_{\alpha\beta}^{-1}}(\widetilde{R}_{\alpha}+K_{\alpha}).
\end{align}
\end{corollary}
\subsubsection*{\textbf{$\cir$-equivariance}}
$(z_{\alpha\beta\gamma},u_{\alpha\beta},K_{\alpha})$ describes a
$\cir$-equivariant gerbe when
 $\widetilde{Lg_{\alpha\beta}}$ is  $\cir$-equivariant.
 A more conceptual explanation for the vanishing of the moment map is that the lifting gerbe is, in fact, the 1-transgression of the Chern-Simons 2-gerbe.

\begin{proposition}\label{moment map vanishes of uab}
The moment map of $(z_{\alpha\beta\gamma},u_{\alpha\beta},K_{\alpha})$ vanishes, \ie
$\iota_{K}(u_{\alpha\beta})=0$.
\begin{proof}
Recall that 
$u_{\alpha\beta}=\widetilde{A}_{\beta}-[\widetilde{Ad}_{Lg_{\alpha\beta}^{-1}}\widetilde{A}_{\alpha}+\widetilde{Lg_{\alpha\beta}}^*\widetilde{\mu}]$,
\begin{align}
\iota_{K}(u_{\alpha\beta})=\iota_{K}(c_{\beta}-c_{\alpha})-\iota_{K}(\widetilde{Lg_{\alpha\beta}}^*v)-\mathcal{Z}(Lg_{\alpha\beta}^{-1},\iota_{K}\hat{A}_{\alpha}).
\end{align}
It directly follows from Corollary \ref{moment map is energy} that $i_{\widetilde{K}}v=-\rho^*E$. Therefore, we have $\iota_{K}(\widetilde{Lg_{\alpha\beta}}^*v)=-E\circ Lg_{\alpha\beta}$.
\end{proof}
\end{proposition}

\subsection{Elliptic Chern Character on \texorpdfstring{$\LX$}{LX}}\label{subsec:Elliptic loop Chern Character}
We define the elliptic Chern character on $\LX$ as the $\disc$-equivariant twisted Chern character of the gerbe module associated with the positive energy representation $\mathcal{H} \in \per^k(LG)$.

Given $\tau \in \mathbb{H}$,
we replace the vector field $K$ with $\tau K$, 
the resulting equivariant curvature locally can be written as follows.
\begin{definition}[Local $\tau$-deformed curvature]
    \begin{align}
\widetilde{R}_{\alpha}[\tau]=\tau \dd+[\tau \iota_{K}(\widetilde{A}_{\alpha})+\widetilde{R}_{\alpha}]\in \Omega^*(LU_{\alpha},\widetilde{L\g}_{\CC}').
\end{align}
\end{definition}

The elliptic Chern character is defined locally as
\begin{align}
ECh_{A}(\mathcal{H})|_{LU_{\alpha}} = \Tr_{\mathcal{H}}\left[\exp\left(\widetilde{R}_{\alpha}[\tau] + K_{\alpha}\right)\right].
\end{align}

We now demonstrate the global well-definedness and equivariant closedness of the elliptic Chern character. The convergence is more difficult and will be addressed in the following subsection.
\subsubsection{Global well-definedness}
\begin{lemma}\label{rel of tau curv-B}
\begin{align}
\widetilde{R}_{\alpha}[\tau]=
\widetilde{Ad}_{Lg_{\alpha\beta}^{-1}}(\widetilde{R}_{\beta}[\tau])+(du_{\alpha\beta})\cc.
\end{align}
\end{lemma}
\begin{proof}
By Corollary \ref{Rel of curv-B}, it suffices to show
\begin{align}\label{transform of connection}
(\tau \dd+\tau \iota_{K}(\widetilde{A}_{\alpha}))=
\widetilde{Ad}_{Lg_{\alpha\beta}^{-1}}((\tau \dd+\tau \iota_{K}(\widetilde{A}_{\beta}))).
\end{align}
This follows from Proposition \ref{moment map vanishes of uab} combined with Proposition \ref{Adjoint formula}, which gives
\begin{align}
    \widetilde{Ad}_{Lg_{\alpha\beta}^{-1}}(\dd+\iota_{K}(\widetilde{A}_{\beta}))=(\dd+\iota_{K}(\widetilde{A}_{\alpha})).
\end{align}
\end{proof}
As a corollary, we obtain
\begin{corollary}\label{rel of tau-eq curvature}
\begin{align}
(\widetilde{R}_{\alpha}[\tau]+K_{\alpha})=\widetilde{Ad}_{Lg_{\alpha\beta}^{-1}}(\widetilde{R}_{\beta}[\tau]+K_{\beta}).
\end{align}
\end{corollary}
Since the positive energy representation $\mathcal{H}$ is unitary, combined with the corollary above, we have
\begin{align}
\Tr_{\mathcal{H}}(\exp{(\widetilde{R}_{\alpha}[\tau]+K_{\alpha})})=\Tr_{\mathcal{H}}(\exp{(\widetilde{R}_{\beta}[\tau]+K_{\beta})})\quad
\text{on} \quad LU_{\alpha}\cap LU_{\beta}.
\end{align}
\subsubsection{Equivariant closedness}
\begin{proposition}
 $ECh_{A}(\mathcal{H})\in \Omega^{2*}(\LX)^{\cir}$ is $(d+\tau \iota_K-kH)$-closed.
\end{proposition}
\begin{proof}
It suffices to show that 
$
(d+\tau \iota_K)\Tr_{\mathcal{H}}[\exp{(\widetilde{R}_{\alpha}[\tau])}]=0.
$
The Bianchi identity still holds in this infinite-dimensional case, but we must be careful about convergence when taking the trace.
\begin{align}
    (d+ad_{\widetilde{A}_{\alpha}}+\tau \iota_{K})\cdot (\widetilde{R}_{\alpha}[\tau])=(d+ad_{\widetilde{A}_{\alpha}}+\tau \iota_{K})\cdot( \tau \dd+[\tau \iota_{K}(\widetilde{A}_{\alpha})+\widetilde{R}_{\alpha}])=0.
\end{align}
By Proposition \ref{coupled with Lg trace convergence} in the next subsection,
$\Tr_{\mathcal{H}}[\widetilde{A}_{\alpha},\exp{(\widetilde{R}_{\alpha}[\tau])}]=0.
$
Therefore,
\begin{align}
(d+\tau \iota_K)\Tr_{\mathcal{H}}[\exp{(\widetilde{R}_{\alpha}[\tau])}]=\Tr_{\mathcal{H}}[(d+ad_{\widetilde{A}_{\alpha}}+\tau \iota_K)\exp{(\widetilde{R}_{\alpha}[\tau])}]=0.
\end{align}
\end{proof}
\subsection{Convergence}\label{Convergence of Elliptic Loop Chern Character}
We now deal with the convergence of elliptic Chern character.
\subsubsection{Energy estimates}
Let $\pi:\widetilde{L\g_{\CC}}'\to \mathrm{End}(\mathcal{H})$ be a positive energy representation. 
The following lemma is a consequence of the Segal-Sugawara construction.
\begin{lemma}[{\cite[Lemma 3.2]{Goodman1984StructureAU},\cite{Wassermann1998OperatorAA}}]\label{key lemma for estimate}
For a vector $v$ of fixed energy $n$ and unit norm, there exists a constant $C>0$ such that for any $m\in \Z$ and $y_m\in \g$,
\begin{align}
\norm{\pi(y_m\otimes z^m)\cdot v}\leq C(\abs{m}+n)\abs{y_m}.
\end{align}
\end{lemma}
Under the representation, $L\g$ acts as skew-adjoint but \textbf{unbounded} operators. The following proposition demonstrates that even though it is unbounded, it can be controlled by the energy $n$ as $n\to \infty$.
\begin{proposition}\label{Lg action controlled by energy}
Let $y\in L\g$, $v$ is a unit norm vector of fixed energy $n$.
When $n$ is sufficiently large, there exists $C'(y)>0$ such that
\begin{align}
\norm{\pi(y)\cdot v}\leq C'(y)n.
\end{align}
\end{proposition}
\begin{proof}
Consider the Fourier decomposition $y=\sum_{m\in \Z} y_m \otimes z^m,\ y_m\in \g$.
Since $y$ is smooth, $\abs{y_m}$ decays faster than any polynomial in $m$. 
Therefore, by the lemma above, for sufficiently large $n$, there exists $C'(y)>0$ such that
\begin{align}
\norm{\pi(y)\cdot v}\leq \sum_{m\in \Z}C(\abs{m}+n)\abs{y_m}\leq C'(y)n.
\end{align}
\end{proof}

To be more precise about the constant $C'(y)$, we should follow \cite{Goodman1984StructureAU} to introduce the Sobolev norms on $L\g$.
But here we only care about the convergence, so we omit many analytic details.

Next, the convergence of the Kac-Weyl character is known.
Given $x\in \mathfrak{t}_{\CC}$, let $V_{n,\mu,c}$ be the weight space of weight $(n,\mu,c)$.
By the decomposition of highest weight modules, the formal character is
\begin{align}
\Tr_{\mathcal{H}}[\exp{( (\pi(\tau \dd+x)))}]=\sum_{n,\mu} q^n \mu(\exp{( x)})\dim(V_{n,\mu,c}).
\end{align}
The following lemma shows that the formal character decays exponentially as the energy $n\to \infty$.
\begin{lemma}[{\cite[5.3]{Slodowy1985ACA}}, \cite{Kac1990InfiniteDL}]\label{decay of formal character}
    \begin{align}
    \sum_{n,\mu}\abs{q^n \mu(\exp{(x)})}\dim(V_{n,\mu,c})\preceq \sum_n n^r \exp{(-\epsilon n)}
    \end{align}
for some $r\in \N,\epsilon>0$.
\end{lemma}
The absolute convergence of the formal character can also be derived from the properties of theta functions after multiplication by the anomaly factor $q^{m_{\Lambda}}$.
Moreover, we demonstrate that the trace remains convergent when coupled with several elements \( y^1, \dots, y^k \in L\g \). 
The off-diagonal components of the operator \( \pi(y^1) \cdots \pi(y^k) \) do not contribute to the trace,
 so it suffices to estimate the diagonal part.
\begin{proposition}\label{coupled with Lg trace convergence}
For \( y^1, \cdots, y^k \in L\g \) and \( x \in \mathfrak{t}_{\CC} \), 
 \begin{align}
\Tr_{\mathcal{H}}[\exp{((\pi(\tau\dd+x)))}\cdot (\pi(y^1)\cdots \pi(y^k))]<\infty.
\end{align}
\end{proposition}
\begin{proof}
To motivate our approach, we start with the case \( k=1 \). It suffices to consider the constant part \( y \in \mathfrak{t}_{\CC} \) since other non-constant components are off-diagonal.
The norm \( \norm{\pi(y)}_{\mathcal{H}_n} \) is bounded by \( n \) for sufficiently large \( n \), while \( \exp{((\pi(\tau\dd+x)))} \) decays exponentially, thus ensuring the trace is convergent.
For any \( k \), the idea is fundamentally the same. 

Under the Fourier decomposition
\begin{align}
    y^j=\sum_{i_j\in \Z}y_{i_j}^j\otimes z^{i_j}, \quad j=1\cdots k,
\end{align}
the constant part of \( \pi(y^1)\cdots \pi(y^k) \) is given by
\begin{align}
\sum_{i_1+\cdots +i_k=0, i_j\in \Z}\pi(y^1_{i_1}\otimes z^{i_1})\cdots \pi(y^k_{i_k}\otimes z^{i_k}).
\end{align}
Given a unit norm vector \( v \) of fixed energy \( n \), using Lemma \ref{key lemma for estimate} iteratively, we have
\begin{align}
&\norm {\sum_{i_1+\cdots i_k=0}\pi(y^1_{i_1}\otimes z^{i_1})\cdots \pi(y^k_{i_k}\otimes z^{i_k})v }  \\
&\ \preceq 
\sum_{i_1+\cdots i_k=0}(\abs{i_1}+\abs{n-i_1})\abs{y_{i_1}^1}\norm{\pi(y_{i_2}^2\otimes z^{i_2})\cdots \pi(y_{i_k}^k\otimes z^{i_k})v} \cdots \nonumber \\
&\ \preceq \sum_{i_1+\cdots i_k=0}(\abs{i_1}+\abs{n-i_1})(\abs{i_2}+\abs{n-i_1-i_2} )\cdots (\abs{i_k}+\abs{n})\abs{y_{i_1}^1}\cdots \abs{y_{i_k}^k}. \nonumber
\end{align}
Since \( y^j \) is smooth, \( a_{i_j}=\abs{y^j_{i_j}} \) has \( \frac{1}{\abs{i_j}^{\infty}} \) decay as \( i_j\to \infty \). On the other hand, 
the coefficient
\begin{align}
c_{i_1,\cdots i_k}=(\abs{i_1}+\abs{n-i_1})(\abs{i_2}+\abs{n-i_1-i_2} )\cdots (\abs{i_k}+\abs{n})
\end{align}
only exhibits polynomial growth in \( n \). As the energy \( n\to \infty \), 
\begin{align}
    \norm {\sum_{i_1+\cdots i_k=0}\pi(y^1_{i_1}\otimes z^{i_1})\cdots \pi(y^k_{i_k}\otimes z^{i_k})v } \preceq  n^k, n\to \infty.
\end{align}
By Lemma \ref{decay of formal character}, we have
\begin{align}
\abs{\Tr_{\mathcal{H}}[\exp{(\tau\dd+x)}\cdot (y_1\cdots y_k)]} \preceq \sum_{n}n^{k+r}\exp{(-\epsilon n)}<\infty.
\end{align}  
\end{proof}
\subsubsection{Convergence of elliptic Chern character}
Locally, the elliptic Chern character is 
\[
ECh_{A}(\mathcal{H})|_{LU_{\alpha}}
=\Tr_{\mathcal{H}}[\exp{(\tau \dd+[\tau \iota_K\hat{A}_{\alpha}+\hat{R}_{\alpha}])}]\cdot \exp{(k[(d+\tau \iota_K)c_{\alpha}+\frac{1}{2}\omega(\hat{A}_{\alpha},\hat{A}_{\alpha})])}.
\]

We only need to establish the convergence of
$\Tr_{\mathcal{H}}[\exp{(\tau\dd+\tau \iota_K\hat{A}_{\alpha}+\hat{R}_{\alpha})}]$.
More precisely, convergence means that for any $\gamma\in LU_{\alpha}$ and arbitrary tangent vectors $X_1,\cdots, X_{2k}\in T_{\gamma}(\LX)$, the resulting operator is trace-class with respect to the positive energy representation $\mathcal{H}$.

Floquet theory, as applied in \cite[(7.2.5)]{FrenkelOrbital}, establishes the convergence of the degree-zero component. Combined with the energy estimates above, we prove convergence of the entire form.

\begin{proposition}\label{convergence of Ech}
 For any $\gamma\in LU_{\alpha}$,
$
  \Tr_{\mathcal{H}}[\exp{(\tau \dd+[\tau \iota_K\hat{A}_{\alpha}+\hat{R}_{\alpha}])}]_{\gamma}\in \Omega^{2*}(LU_{\alpha})_{\gamma}
$
is convergent.
\end{proposition}
\begin{proof}
We establish convergence in the presence of tangent vectors $X_1, \ldots, X_{2k} \in T_{\gamma}(LU_{\alpha})$.
Let $Y_i = \hat{R}(X_{2i-1}, X_{2i})_{\gamma} \in L\g$ for $i = 1, \ldots, k$.
By Floquet theory, there exists a gauge transformation $g_0 \in LG$ such that
\begin{align}
\widehat{Ad}_{g_0}(\tau \dd + \tau(\iota_K\hat{A}_{\alpha})_{\gamma}) = \tau \dd + \tau K_0,
\end{align}
where $K_0 \in \mathfrak{t}$ is a constant element.
Therefore,
\begin{align}
&\Tr_{\mathcal{H}}[\exp(\tau \dd + [\tau \iota_K\hat{A}_{\alpha} + \hat{R}_{\alpha}])]_{\gamma}(X_1 \cdots X_{2k})\\
&\quad = \Tr_{\mathcal{H}}[\exp(\widetilde{Ad}_{g_0}(\tau \dd + [\tau \iota_K\hat{A}_{\alpha} + \hat{R}_{\alpha}]))]_{\gamma}(X_1 \cdots X_{2k}) \nonumber\\
&\quad = \exp(k\triangle) \cdot \Tr_{\mathcal{H}}\bigl[\exp(\pi(\tau \dd + \tau K_0)) \tfrac{1}{k!}\pi((Ad_{g_0}Y_1)) \cdots \pi((Ad_{g_0}Y_k))\bigr], \nonumber
\end{align}
where $\triangle$ denotes the additional central term arising from the adjoint action.
The convergence follows from Proposition \ref{coupled with Lg trace convergence}.
\end{proof}
\subsection{Trivialization by String-\texorpdfstring{$G$}{G} Structure}\label{subsec: Trivialized Elliptic Loop Chern Character by String G-structure}
We study the trivialization of the lifting gerbe $(\mathcal{G}_P,\nabla^B)$, particularly via the transgression of the geometric string $G$-structure.
The trivialization of the lifting gerbe is described by local data $(h_{\alpha\beta},\mathcal{B}_{\alpha})$, where $h_{\alpha\beta}:LU_{\alpha}\cap LU_{\beta}\to S^1$ and $\mathcal{B}_{\alpha}\in \Omega^1(LU_{\alpha})$, such that
the cocycle in \cref{!!!cocycle} becomes a coboundary (see \cite{Gomi2001ConnectionsAC}):
\begin{align}
(z_{\alpha\beta\gamma},u_{\alpha\beta},K_{\alpha})=\delta(h_{\alpha\beta},\mathcal{B}_{\alpha}).
\end{align}
We can correct the transition functions and connections by
\begin{itemize}
    \item $\widetilde{Lg_{\alpha\beta}}^{cor}=h_{\alpha\beta}\cdot \widetilde{Lg_{\alpha\beta}} : LU_{\alpha}\cap LU_{\beta}\to \widetilde{LG}$.
    \item $\widetilde{A}_{\alpha}^{cor}=\widetilde{A}_{\alpha}+\mathcal{B}_{\alpha}\cc: \in \Omega^1(LU_{\alpha},\widetilde{L\g})$.
\end{itemize}
After the correction, it results in the lifting principal $\widetilde{LG}$-bundle $(\widetilde{LP},\widetilde{A})\to \LX$, which is compatible with $(LP,\hat{A})\to \LX$ as described in \cite{Coquereaux1989StringSO}:
\be
    \begin{tikzcd}
        (\widetilde{LP},\widetilde{A})\arrow[dr]\arrow[rr]&  & (LP,\hat{A})\arrow[dl]\\
                                      &\LX &
    \end{tikzcd}
\ee
where $\widetilde{LP}\to LP$ is a principal circle bundle.
The geometric string $G$-structure refers to the geometric trivialization of the Chern-Simons $2$-gerbe on $X$.
 It can be transgressed to become the geometric trivialization of the lifting gerbe $(\mathcal{G}_P,\nabla^B)$ on $\LX$. 
 There exists a $B_{\alpha}\in \Omega^2(U_{\alpha})$ and
$
\mathcal{B}_{\alpha}=\int_{S^1}B_{\alpha}.
$
 Then
$\{B_{\alpha}\}$ defines the connection of trivialization of the Chern-Simons $2$-gerbe $(\mathbb{CS}_P,\nabla_A)$, which is called the string-$G$ connection, such that
\begin{align}\label{11}
C=\CCS(A_{\alpha})-dB_{\alpha}
\end{align}
is the curving. It is a globally defined $3$-form on $X$ satisfying $dC=\Phi$.
Let 
\begin{align}
\mathcal{C}=\int_{S^1}C\in \Omega^2(\LX)
\end{align}
be the transgression, which serves as the curving of the trivialization of the lifting gerbe.
To see this, we can take the transgression of both sides of \cref{11},
\begin{align}
\mathcal{C}=K_{\alpha}-d\mathcal{B}_{\alpha}.
\end{align}
The corrected curvature is locally expressed as 
\begin{align}
\widetilde{R}_{\alpha}^{cor}=\widetilde{R}_{\alpha}+d\mathcal{B}_{\alpha}\cc.
\end{align}

\begin{proposition}
The trivialized elliptic Chern Character $ECh_{A}(\mathcal{H})|_{LU_{\beta}}$ is given by
\begin{align}
\Tr_{\mathcal{H}}[(\exp{(\tau \dd+\tau \iota_K \hat{A}_{\beta}+\hat{R}_{\beta}+\tau \iota_Kc_{\beta}\cc)}]\exp{(k\int_{S^1}\langle A_{\beta}\wedge R_{\beta}\rangle)}\cdot \exp{(-k\mathcal{C})}.
\end{align}
\end{proposition}
\begin{proof}
The corrected $\tau$-deformed curvature is given by
\begin{align}
\widetilde{R}_{\beta}^{cor}[\tau]&=\tau \dd+\tau \iota_K\widetilde{A}_{\beta}^{cor}+\widetilde{R}_{\beta}^{cor} =(\tau \dd+\tau \iota_K \hat{A}_{\beta}+\hat{R}_{\beta})+\square_{\beta}\cc.
\end{align}
Recall that
\begin{align}
\widetilde{R}_{\beta}=\hat{R}_{\beta}+(\frac{1}{2}\omega(\hat{A}_{\beta},\hat{A}_{\beta})+dc_{\beta})\cc,
\end{align}
Then the central part is
\begin{align}
\square_{\beta}&=((d+\tau \iota_K)c_{\beta}+d\mathcal{B}_{\beta}+\frac{1}{2}\omega(\hat{A}_{\beta},\hat{A}_{\beta}))\\ 
&=\tau \iota_Kc_{\beta}+K_{\beta}-\mathcal{C}+dc_{\beta}+\frac{1}{2}\omega(\hat{A}_{\beta},\hat{A}_{\beta})=\tau \iota_Kc_{\beta}-\mathcal{C}+\int_{S^1}\langle A_{\beta}\wedge R_{\beta}\rangle. \nonumber
\end{align}
The last equality follows from the formulas in 
\cref{subsec:geometry of lifting gerbe and cir-equivariance}: $K_{\beta}+dc_{\beta}=\Psi_{\beta}$ and $\Psi_{\beta}+\frac{1}{2}\omega(\hat{A}_{\beta},\hat{A}_{\beta})=\int_{S^1}\langle A_{\beta}\wedge R_{\beta}\rangle$.

Since $\mathcal{H}\in \per^k(LG)$ is a level-$k$ positive-energy representation, the formula for the trivialized elliptic Chern character follows.
\end{proof}  

Let us explain the geometric interpretation of each term:
$\mathcal{E}_{\beta}=\iota_Kc_{\beta}=\int_{S^1}\langle \iota_K\hat{A}_{\beta},\iota_K\hat{A}_{\beta}\rangle$ can be viewed as the kinetic energy of the gauge field along the loop $\gamma$. 
It is defined only locally and serves as a correction term, as in \cite{FrenkelOrbital}.
The factor $\exp(-k\mathcal{C})$ arises from the global $B$-field $\mathcal{C}$, which emerges from the geometric trivialization of the lifting gerbe.
$\exp{(k\int_{S^1}A_{\beta}\wedge R_{\beta})}$ again arises from the reduced splitting.

\subsection{\texorpdfstring{$\disc$}{disc}-equivariant \texorpdfstring{$\hat{A}$}{hatA}-genus}\label{subsec: Spinor Bundle and -Equivariant A-hat-genus on LX}\label{subsec: Spinor Bundle and Equivariant A-hat-genus on LX}
As an application, we propose the $\disc$-equivariant $\hat{A}$-genus on $\LX$, which serves as the elliptic analogue of the classical $\hat{A}$-genus on $X$.
\subsubsection{String and Loop-Spin}
Let $(X,g)$ be a spin Riemannian manifold of dimension $2n$.
Let $P = P_{\mathrm{Spin}(X)}$ denote the spin frame bundle of $TX$,
which is a principal $G = \mathrm{Spin}(2n)$-bundle.
Let $A$ denote the spin connection induced by the Levi-Civita connection.
In this case, the intrinsic Chern-Simons $2$-gerbe associated to the spin frame bundle $(P,A)$ is denoted by $(\mathrm{CS}_X,\nabla^g)$.
The curving is the string anomaly
\begin{align}
    \Phi=\langle R,R\rangle=\frac{1}{2}p_1(X,g),
\end{align}
where $p_1(X,g)$ is the first Pontryagin form.
The factor $\frac{1}{2}$ arises from the fact that the Dynkin index is $2$. The minimal form on $\mathrm{spin}(2n)$ is thus half the Killing form.

A string structure on $X$ refers to the trivialization of $(\CCS_X,\nabla^g)$. It can be transgressed to give the trivialization of the lifting gerbe $(\mathcal{G}_{\mathrm{Spin}(X)},\nabla^B)$ on $\LX$,
which is equivalent to lifting the principal $L\mathrm{Spin}(X)$-bundle $L\mathrm{Spin}(X)\to \LX$ to a principal $\widetilde{L\mathrm{Spin}(2n)}$-bundle.
This lifting is called the loop-spin structure (see \cite{waldorf_spin_2012}).
\subsubsection{Spinor bundle on $\LX$}
The construction of the spinor bundle over $\LX$ has garnered significant attention as it is a prerequisite for the rigorous construction of the Dirac-Ramond operator on $\LX$, which is central to the Stolz conjecture.
We briefly summarize the construction of the (twisted) loop spinor bundle $\mathcal{S}_{\LX}$ on $\LX$. It is the gerbe module of the lifting gerbe $\mathcal{G}_{\mathrm{Spin}(X)}$ on $\LX$,
 associated with the level-one Fock representation $\mathcal{F} = S^+-S^-$, where $S^\pm \in \per^1(L\mathrm{Spin}(2n))$. 

Assume $X$ admits a string structure. A loop-spin structure on $\LX$ then provides a lift
$\widetilde{L\mathrm{Spin}(X)}\to \LX$. The loop spinor bundle is defined as the associated Fock bundle
\begin{align}
\mathcal{S}_{\LX}=\widetilde{L\mathrm{Spin}(X)}\times_{\widetilde{L\mathrm{Spin}(2n)}}\mathcal{F}\longrightarrow \LX.
\end{align}
Since the $\widetilde{L\mathrm{Spin}(2n)}$ action is only smooth on a dense subspace,
we refer to \cite{Kristel2020SmoothFB} for a rigorous treatment.
\subsubsection{$\tau$-deformed Euler class of $\LX$}
Recall that the super Bismut-Chern character of the spinor bundle, $BCh_s(\mathcal{S}_X)$, 
can be viewed as the $\cir$-equivariant Euler class of $\LX$:
\begin{align}
e_{\cir}(\LX) = BCh_s(\mathcal{S}_X).
\end{align}
We previously demonstrated that it is reasonable to call this the $\cir$-equivariant Euler class since
\begin{align}\label{pre-index formula}
\hat{A}(X) \wedge Ch_s(\mathcal{S}_X) = e(X).
\end{align}
We now define a $\disc$-equivariant Euler class at $\tau \in \mathbb{H}$ by modifying this construction as follows:
\begin{align}
e_{\disc}(\LX)[\tau]:= \tau^{-\mathrm{deg}/2} \cdot e_{\cir}(\LX),
\end{align}
where $\mathrm{deg}:\Omega^*(\LX) \to \Omega^*(\LX)$ denotes the degree operator, 
which acts on differential forms by $\mathrm{deg}(\omega) = m\omega$ for $\omega \in \Omega^m(\LX)$. 
Then $e_{\disc}(\LX)[\tau]$ is $(d + \tau \iota_K)$-closed.
\subsubsection{$\tau$-deformed $\hat{A}$-genus on $\LX$}
When $X$ is string,
we define the $\tau$-deformed super Chern character of the loop spinor bundle $\mathcal{S}_{\LX}$ as
\begin{align}
Ch_s(\mathcal{S}_{\LX})[\tau] = ECh_s(\mathcal{F}) \in \Omega^{2*}_{d+\tau \iota_K}(\LX)^{\cir},
\end{align}
where $ECh_s(\mathcal{F}) = ECh(S^+) - ECh(S^-)$ is the super elliptic Chern character of the Fock representation $\mathcal{F}=S^+ - S^-$. 

The $\disc$-equivariant $\hat{A}$-genus on $\LX$ is given by the formal expansion
\begin{align}\label{hatA on LX}
\hat{A}_{\disc}(\LX)[\tau] = \tau^n \frac{e_{\disc}(\LX)[\tau]}{q^{m_{\Lambda}} \cdot Ch_s(\mathcal{S}_{\LX})[\tau] / \eta(\tau)^{2n}} \in \Omega^{2*}_{d+\tau K}(\LX)^{\cir}.
\end{align}
This can be interpreted as the elliptic analogue of the pre-index formula \cref{pre-index formula}.
Applying the localization formula, we aim to show:
\begin{align}\label{localized to be Witten genus}
\int_{\LX} \hat{A}_{\disc}(\LX)[\tau] &= \int_X \hat{A}_{\disc}(\LX)[\tau]|_X \wedge \hat{A}(X) = \int_X \hat{W}(X),
\end{align}
where $\hat{W}(X)$ is the Witten genus.

When restricted to the constant loop space $X$, 
\begin{align}
\hat{A}_{\disc}(\LX)[\tau]|_X = \frac{Ch_s(\mathcal{S}_X)}{q^{m_{\Lambda}} \cdot Ch(S^+-S^-) / \eta(\tau)^{2n}},
\end{align}
where $Ch(S^+-S^-)$ is the super $q$-graded Chern character on $X$. Expressed in terms of the Chern roots $\{x_i\}$, 
\begin{align}
q^{m_{\Lambda}} \cdot Ch(S^+ - S^-) / \eta(\tau)^{2n} = \prod_{i=1}^n \frac{\theta_{11}(x_i, \tau)}{\eta(\tau)^3}.
\end{align}
Therefore, we obtain
\begin{align}
\hat{A}_{\disc}(\LX)[\tau]|_X \wedge \hat{A}(X) = \prod_{i=1}^n \frac{x_i}{\theta_{11}(x_i, \tau) / \eta(\tau)^3} = \hat{W}(X).
\end{align}
Formally, the $\disc$-equivariant $\hat{A}$-genus on $\LX$ can be viewed as the $\disc$-equivariant index density for the Dirac-Ramond operator (if rigorously defined) on $\LX$:
\begin{align}
\mathrm{Ind}_{\disc}(\slashed{D}_{\LX})[\tau]= \int_{\LX}\hat{A}_{\disc}(\LX)[\tau]=\int_X \hat{W}(X).
\end{align}

On the other hand, torus localization $\LLX\to X$ directly computes the Witten genus (see \cite{Liu1994ModularIA}). 
The $\disc$-equivariant $\hat{A}$-genus on $\LX$ serves as an intermediate result in the localization process $\LLX \to \LX \to X$.
From this perspective, the $\disc$-equivariant $\hat{A}$-genus arises as the equivariant localization from $\LLX$ to $\LX$.
\section{Elliptic Atiyah--Witten Formula}\label{Sec: Elliptic AW and Modular}
In this section, we establish the elliptic Atiyah--Witten formula on $\LLX$ 
and relate the geometric quantization of Chern--Simons gauge theory to the double loop space geometry.

In \cite{AST_1985__131__43_0}, the normalized infinite product
\begin{align}\label{1d infinite product}
\prod_{n\in \Z} (x+n) \sim \sin(\pi x)
\end{align}
plays the dominant role in the loop space Lagrangian formulation of Atiyah-Singer index theory.
The $\cir$-equivariant Euler class of the normal bundle of $X$ in $\LX$ computes the $\hat{A}$-genus
\begin{align}
e_{\cir}(X/\LX)^{-1}\sim \hat{A}(X).
\end{align}
On the other hand, it yields the Atiyah-Witten formula on $\LX$, which identifies the $\zeta$-regularized pfaffian of the covariant derivative $\pf_{\zeta}(\nabla_{\dot{\gamma}})$ with the super-trace of the holonomy of the spinor bundle
$\Tr_s[hol_{\mathcal{S}_X}(\gamma)]$.

Liu observed that the two-dimensional analogue
of \eqref{1d infinite product} is the Eisenstein product formula (see \cite{Liu1994ModularIA})
\begin{align}\label{Eisenstein product}
\prod_{m,n\in \Z}(z+m+n\tau) \sim \theta_{11}(z,\tau)/\eta(\tau).
\end{align}
The torus-equivariant Euler class of the normal bundle of $X$ in $\LLX$ computes the Witten genus
\begin{align}
e_{T^2}(X/\LLX)^{-1}\sim \hat{W}(X).
\end{align}
We expect the normalized infinite product should yield the elliptic Atiyah-Witten formula on $\LLX$.
There are four formulas, as one can replace $z$ with $z+\frac{i+\tau j}{2}, i, j \in \Z^2$ in \cref{Eisenstein product} to derive four Jacobi theta functions.

On the holonomy side, we define the elliptic holonomy on $\LLX$ from the $(1+1)$-transgression perspective
as the $\disc$-deformed equivariant twisted trace of the holonomy of gerbe modules on $\LX$ associated with positive-energy representations:
\be
Ehol: \per^k(LG)\to \Gamma(\LLX,\mathcal{L}(\mathcal{G}_P)^k)^{T^2}.
\ee 
Here, $\mathcal{L}(\mathcal{G}_P)\to \LLX$ denotes the transgression line bundle of the lifting gerbe $(\mathcal{G}_P,\nabla^B)$ on $\LX$.
A key limitation is that we have only established convergence of the elliptic holonomy on gauged poly-stable double loops, rather than on the entire space $\LLX$.

We now focus on $G=\mathrm{Spin}(2n)$ at level one, where
\be
\per^1(L\mathrm{Spin}(2n))=\{S^+-S^-,S^++S^-, S_+-S_-, S_++S_-\}.
\ee
These four virtual level-one positive-energy representations give rise to four corresponding elliptic holonomies on $\LLX$.

On the analytic side, there are pfaffians of four spin structures on the elliptic curve $\Sigma_{\tau}$ which are canonical sections of the corresponding Pfaffian line bundles over $\LLX$.
We show the transgression of lifting gerbe $\mathcal{L}(\mathcal{G}_P,\nabla^B)$ 
and the determinant(Pfaffian) line bundles over $\LLX$ are all isomorphic to Chern-Simons line bundles.

During the exploration, we discover that the elliptic Atiyah-Witten formula is connected to the geometric quantization of Chern-Simons gauge theory. 
This arises from the fact that the push-down of four Pfaffians to the coarse moduli provides a basis-to-basis proof of \emph{quantization commutes with reduction} (for general proof by extension, see \cite{Preparation}):
\begin{align}
    r: \mathcal{H}_1(\mathrm{Spin}(2n))\cong V_1(\mathrm{Spin}(2n)).
\end{align}

Essentially, we show that the push-down of the four Pfaffian sections to the coarse moduli corresponds to the four Jacobi theta functions, 
which are basis of the genus one conformal blocks $V_1(\mathrm{Spin}(2n))$ when $G=\mathrm{Spin}(2n)$ at level one.
The elliptic Atiyah-Witten formula requires only a principal $\mathrm{Spin}(2d)$–bundle with connection $(P,A)\to X$ and the associated vector bundle via the standard representation $\rho:\mathrm{Spin}(2d)\to \mathrm{SO}(2d)$; 
it need not be the spin frame bundle of a Riemannian spin manifold. 
In the case of genus one, the conformal block $V_1(\mathrm{Spin}(2d))$ is special:
 its Verlinde dimension is four, and a natural basis is provided by the four Jacobi theta functions. Four spin structures on the elliptic curve $\Sigma_{\tau}$ correspond to the four virtual level-one positive-energy representations of $L\mathrm{Spin}(2d)$ and to the four Jacobi theta functions.
 One should not expect analogous push-downs of Pfaffians to yield the full conformal blocks for other groups or higher levels.

{\bf This section is organized as follows.}
In Subsection \ref{Subsec: Atiyah-Witten}, we review the classical Atiyah--Witten formula on $\LX$ and its $\Z_2$-torsion generalization for non-spin manifolds. 
In Subsection \ref{subsec: Elliptic Holonomy}, we introduce the elliptic holonomy.
In Subsection \ref{Pfaffians}, we introduce the Pfaffians on $\LLX$ via family index theory.
In Subsection \ref{subsec: EQ of Line Bundles}, we introduce the universal Chern--Simons line bundle and its pullback to mapping spaces. We identify the transgression of the lifting gerbe and the determinant (Pfaffian) line bundles with Chern--Simons line bundles.
In Subsection \ref{subsec:Elliptic Atiyah-Witten on MG}, we compute the pushdown of determinants (Pfaffians) and establish the elliptic Atiyah--Witten formula on the coarse moduli.
In Subsection \ref{subsec:Elliptic Atiyah-Witten}, we give the elliptic Atiyah--Witten formula on $\LLX$.
In Subsection \ref{subsec: from QR=0 of CS and Conformal Blocks}, we relate the quantization of Chern--Simons gauge theory to double loop space geometry.
\subsection{The Classical Atiyah-Witten Formula on \texorpdfstring{$\LX$}{LX}}\label{Subsec: Atiyah-Witten}
Let $(X,g)$ be a compact oriented Riemannian manifold of dimension $2n$. When $X$ is spin, $\mathcal{S}_X$ is the spinor bundle. The Atiyah-Witten formula on $\LX$ in \cite{AST_1985__131__43_0} states that
    \begin{align}
    \pf_{\zeta}(\nabla_{\dot{\gamma}})=\tr_s[hol_{\mathcal{S}_X}(\gamma)], \quad \gamma\in \LX.
    \end{align}
This formula identifies the $\zeta$-regularized Pfaffian of the covariant derivative $\nabla_{\dot{\gamma}}$ with the supertrace of the holonomy of $\mathcal{S}_X$.

When $X$ is not spin, the obstruction is measured by the second Stiefel-Whitney class $w_2(X)$, which can be transgressed to a $\mathbb{Z}_2$-torsion line bundle over $\LX$.

On the analytic side, this torsion line bundle is represented by the Pfaffian line bundle $(\mathrm{PF}_V,\norm{\cdot}_{\mathrm{PF}})\to \LX$, which has a canonical Pfaffian section $\pf\in \Gamma(\LX,\mathrm{PF}_V)$ satisfying
\be 
\norm{\text{pf}(\gamma)}_{\text{PF}}=\text{pf}_{\zeta}(\nabla_{\dot{\gamma}}).
\ee
On the topological side, it is realized as the transgression line bundle of the spin lifting gerbe. Consider the loop frame bundle $L\mathrm{SO}(X)\to \LX$, which is a principal $L\mathrm{SO}(2n)$-bundle. There exists a short exact sequence of Lie groups:
\begin{align}
1\to \Z_2\to \mathrm{Spin}(2n)\to \mathrm{SO}(2n)\to 1.
\end{align}
The monodromy map $m:L\mathrm{SO}(2n)\to \Z_2$ defines the transgression line bundle $\mathcal{L}$ as the associated line bundle $L\mathrm{SO}(X)\times_m \R$ over $\LX$, equipped with the trivial metric $\norm{\cdot}$. The super-trace of holonomy $\tr_s[hol_{\mathcal{S}_X}]\in \Gamma(\LX,\mathcal{L})$ serves as the canonical section, with the property that
\be
\norm{\tr_s[hol_{\mathcal{S}_X}(\gamma)]}^2=\det(I-hol_{TX}(\gamma)).
\ee
The $\Z_2$-torsion Atiyah-Witten formula \cite{Hanisch2017TheFI} states that there is a geometric isometry between the two line bundles:
\begin{align}
\Phi:(\PF_V,\norm{\cdot}_{\PF})\to (\mathcal{L},\norm{\cdot}).
\end{align}
Under the isometry, canonical sections $\pf$ and $\tr_s[hol_{\mathcal{S}_X}]$ are identified.
\subsubsection*{\textbf{Atiyah-Witten formula on $\mathrm{Loc}_G(S^1)$}}
Temporarily, we set $G=\mathrm{SO}(2n)$ and $\widetilde{G}=\mathrm{Spin}(2n)$. We demonstrate that the classical Atiyah–Witten formula can be established on $\mathrm{Loc}_G(S^1)$ before pulling it back to $\LX$.

The $LBG$ can be realized as the quotient stack $\mathrm{Loc}_G(S^1)=[\mathcal{A}_{S^1}/LG]$, which is Morita equivalent to $[G/G]$ 
via the holonomy map $\Psi:\mathcal{A}_{S^1}\to G$ sending a connection to its holonomy around $S^1$,
 (see \cite{Behrend2003EquivariantGO}).

On the analytic side, the Pfaffian line bundle $\mathrm{PF}$ on $\mathrm{Loc}_G(S^1)$ is realized as a $LG$-equivariant geometric line bundle over $\mathcal{A}_{S^1}$, with its canonical section $\pf$ serving as a $LG$-equivariant section.

On the topological side, consider the line bundle $\mathcal{L}_G = \widetilde{G} \times_{\Z_2} \mathbb{R} \to G$, a $\Z_2$-torsion $G$-equivariant line bundle. 
The following diagram illustrates the relationship:
\be
\begin{tikzcd}
\mathcal{L} \arrow[r] \arrow[d] & \mathcal{L}_G \arrow[d] \\
\mathcal{A}_{S^1} \arrow[r, "\Psi"] & G
\end{tikzcd}
\ee
The pullback $\mathcal{L} = \Psi^*\mathcal{L}_G$ is a $LG$-equivariant $\Z_2$-torsion line bundle on $\mathcal{A}_{S^1}$. 
Given the spin representation $\Delta = \Delta^+ \oplus \Delta^-$, its super character $\chi: \widetilde{G} \to \mathbb{R}$ is a class function of $\widetilde{G}$. It can be identified with a $G$-equivariant section $\chi: G \to \widetilde{G} \times_{\Z_2} \mathbb{R}$ by $\chi(g) = [\tilde{g}, \chi_{\Delta}(\tilde{g})]$, where $\tilde{g} \in \widetilde{G}$ is a lift of $g \in G$. 
This definition is independent of the choice of lift. The section $\chi$ serves as the canonical section of $\mathcal{L}_G$, as the representation ring $R(\widetilde{G})$ is a free $R(G)$-module of rank one generated by the virtual representation $\Delta^+ - \Delta^-$. 

There exists a geometric isometry between the Pfaffian line bundle and the transgression line bundle, both regarded as $LG$-equivariant $\Z_2$-torsion line bundles over $\mathcal{A}_{S^1}$. Under this isometry, the canonical section $\pf$ and the pullback $\Psi^*\chi$ are identified. We only need to verify this for constant connections, as any $G$-connection in $\mathcal{A}_{S^1}$ can be gauge-equivalent to a constant connection. The identification for constant connections follows directly from the classical infinite product formula \eqref{1d infinite product}.

These ideas will motivate our approach to the elliptic Atiyah-Witten formula on $\LLX$.
\subsection{Elliptic Holonomy}\label{subsec: Elliptic Holonomy}
We define the elliptic holonomy as the $\disc$-deformed equivariant twisted holonomy of the gerbe module on $\LX$ that arises from positive energy representations. It is a $T^2$-equivariant section of the transgression line bundle of the lifting gerbe on $\LLX$.
\subsubsection{Transgression line bundle of the lifting gerbe}
Recall from Subsection \ref{subsec:geometry of lifting gerbe and cir-equivariance} that the Deligne cocycle $(z_{\alpha\beta\gamma},u_{\alpha\beta},K_{\alpha})$ describes the lifting gerbe $(\mathcal{G}_P,\nabla^B)$. The transgression line bundle with connection $\mathcal{L}(\mathcal{G}_P,\nabla^B)\to \LLX$ is characterized by 
\begin{itemize}
    \item The transition function $h_{\alpha\beta}':L^2U_{\alpha}\cap L^2U_{\beta}\to S^1$,
   $$
     h_{\alpha\beta}':=\exp{( i\int_{S^1}u_{\alpha\beta})}.  
    $$
    \item 
    The connection form $\Theta_{\alpha}\in \Omega^1(L^2U_{\alpha},i\R)$,
    $$
    \Theta_{\alpha}= i\int_{S^1}K_{\alpha}=i\int_{\Sigma}\CCS(A_{\alpha}).
    $$
\end{itemize}
\subsubsection{Definition of the elliptic holonomy}
Given $\mathcal{H}\in \per^k(LG)$, we aim to define 
\be
Ehol_A(\mathcal{H})\in \Gamma(\LLX,\mathcal{L}(\mathcal{G}_P)^k)^{T^2}.
\ee
Let $\partial_z^{\#} = \tau \partial_x - \partial_y$ be the complex vector field dual to $dz$ under the standard Kähler form on the elliptic curve $\Sigma_{\tau}$. 
This induces a complex vector field on $\LLX$. 
We consider the parallel transport of the operator $\mathcal{D}_{\alpha}^0$ defined below, which acts in the $y$-direction and takes values in $\widetilde{L\g_{\CC}}'$:
\begin{align}
\mathcal{D}_{\alpha}^0=\frac{d}{dy}-(\tau \dd+\iota_{\partial_{z}^{\#}}\widetilde{A}_{\alpha}).
\end{align}
Under the positive energy representation $\pi:\widetilde{L\g_{\CC}}'\to \mathrm{End}(\mathcal{H})$, when evaluated at $\gamma\in L^2U_{\alpha}$ and at time $y$, the operator $\pi(\tau \dd +\iota_{\partial_{z}^{\#}}\widetilde{A}_{\alpha}(\gamma)_y)$ acts on $\mathcal{H}$. The transport is given by the iterated integral
\begin{align}
hol(\mathcal{D}_{\alpha}^0)_{\gamma}=\sum_{n=0}^{\infty}\int_{\Delta_n}\prod_{i=1}^{n}\pi( \tau \dd +\iota_{\partial_{z}^{\#}}\widetilde{A}_{\alpha}(\gamma)_{y_i})dy_i.
\end{align}
We first prove that it defines a global section which is called the elliptic holonomy.
\begin{proposition}\label{Global defined elliptic holonomy}
For $\mathcal{H}\in \per^k(LG)$, define 
$
f_{\alpha}(\gamma)=\Tr_{\mathcal{H}}[hol(\mathcal{D}_{\alpha}^0)_{\gamma}], \quad \gamma\in L^2U_{\alpha}.
$
On the intersection $L^2U_{\alpha}\cap L^2U_{\beta}$, we have
$
f_{\beta}(\gamma)=h_{\alpha\beta}^k(\gamma)f_{\alpha}(\gamma).
$
\end{proposition}
\begin{proof}
Recall from \cref{!!!cocycle} that
$
u_{\alpha\beta}=\widetilde{A}_{\beta}-[\widetilde{Ad}_{Lg_{\alpha\beta}^{-1}}\widetilde{A}_{\alpha}+\widetilde{Lg_{\alpha\beta}}^*\widetilde{\mu}],
$
and from Lemma \ref{rel of tau curv-B} that
$
(\tau\dd-\tau \iota_{\partial_x}\widetilde{A}_{\beta})=
\widetilde{Ad}_{Lg_{\alpha\beta}^{-1}}(\tau \dd-\tau \iota_{\partial_x}\widetilde{A}_{\alpha}).
$
We then have
\begin{align}
 f_{\beta}(\gamma)&=\Tr_{\mathcal{H}}hol_{\gamma}\left(\frac{d}{dy}-[\iota_{\partial_y}\widetilde{A}_{\beta}]-[\tau \dd-\tau \iota_{\partial_x}\widetilde{A}_{\beta}]\right) \\
 &=\Tr_{\mathcal{H}}hol_{\gamma}\left(\frac{d}{dy}-[\iota_{\partial_y}(u_{\alpha\beta}+\widetilde{Ad}_{Lg_{\alpha\beta}^{-1}}\widetilde{A}_{\alpha}+\widetilde{Lg_{\alpha\beta}}^*\widetilde{\mu})]-\widetilde{Ad}_{Lg_{\alpha\beta}^{-1}}(\tau \dd-\tau \iota_{\partial_x}\widetilde{A}_{\alpha})\right) \nonumber\\
 &=\exp{(\int_{S^1}iku_{\alpha\beta})}\cdot \Tr_{\mathcal{H}}hol_{\gamma}\left(\widetilde{Lg_{\alpha\beta}}^{-1}\left(\frac{d}{dy}-[\iota_{\partial_y}\widetilde{A}_{\alpha}]-[\tau \dd-\tau \iota_{\partial_x}\widetilde{A}_{\alpha}]\right)\widetilde{Lg_{\alpha\beta}}\right)\nonumber\\
 &=\exp{(\int_{S^1}iku_{\alpha\beta})}\cdot \Tr_{\mathcal{H}}hol_{\gamma}\left(\frac{d}{dy}-[\iota_{\partial_y}\widetilde{A}_{\alpha}]-[\tau \dd-\tau \iota_{\partial_x}\widetilde{A}_{\alpha}]\right) = h_{\alpha\beta}^k(\gamma)f_{\alpha}(\gamma). \nonumber
\end{align}
\end{proof}
The elliptic holonomy  $Ehol_A(\mathcal{H})\in \Omega^0(\LLX,\mathcal{L}(\mathcal{G}_P)^k)^{T^2}$ is defined using the local data
 $\{f_{\alpha},L^2U_{\alpha}\}$.

\subsubsection{Convergence at gauged-poly-stable double loops}
Previously, in \cref{Convergence of Elliptic Loop Chern Character}, we established the convergence of the elliptic Chern character on $\LX$ using Floquet theory, which shows that $G$-connections on the circle can be gauge transformed into constant connections. However, $G$-connections on the torus $\mathcal{A}$ exhibit significantly different behavior: not all connections can be transformed into flat connections, let alone constant ones.

It is related to stability problems when considering the complex structure. Once the complex structure $\tau\in \mathbb{H}$ of the torus is chosen, $\mathcal{A}$ acquires a K\"ahler structure, and we have
$
\mathcal{A} \cong \Omega^{0,1}(\Sigma_{\tau},\g_{\CC})
$
as $\CC$-vector spaces. The corresponding complexification of the gauge group $\Sigma G_{\CC}$ acts on $\mathcal{A}$.

Given a connection $A$, the polarization $\mathcal{A} \cong \Omega^{0,1}(\Sigma_{\tau},\g_{\CC})$ associates it with the Cauchy-Riemann operator $\bar{\partial}_A = \bar{\partial} + A$, which classifies holomorphic $G_{\CC}$-bundles on $\Sigma_{\tau}$. The moment map $\mu$ sends a connection to its curvature. The preimage $\mu^{-1}(0)$ is the space of flat connections.
\begin{definition}
A connection in the $\Sigma G_{\CC}$-orbit of $\mu^{-1}(0)$ in $\mathcal{A}$ is called poly-stable. Denote $\mathcal{A}^{ps}=\Sigma G_{\CC}\cdot \mu^{-1}(0)$.
\end{definition}
From the Riemann-Hilbert correspondence, we have the following
\begin{proposition}\label{equivalent to constant}
Let $\mathfrak{t}_{\CC}\subset \Omega^{0,1}(\Sigma_{\tau},\g_{\CC})$ denote the space of constant connections with values in $\mathfrak{t}_{\CC}$. Every flat connection is gauge equivalent to an element in $\mathfrak{t}_{\CC}$.
\end{proposition}
As a corollary, any poly-stable connection is $\Sigma G_{\CC}$-equivalent to a constant Cartan connection in $\mathfrak{t}_{\CC}$.

For $Y\in \Sigma \g_{\CC}$, consider the operator $\frac{d}{dy}-(\tau \dd+Y)$. On one hand, it is a first-order differential operator in the $y$-direction with values in $L\g_{\CC}$. On the other hand, denoting $\dd$ as $\frac{d}{dx}$ to emphasize its differentiation in the $x$-direction, it can be viewed as a $\bar{\partial}$-operator on $\Sigma_{\tau}$ with values in $\g_{\CC}$. The set $\{\frac{d}{dy}-(\tau \frac{d}{dx}+Y),Y\in \Sigma \g_{\CC}\}$ forms the $(-\tau,1)$-level set of the extended double loop algebra $\widehat{\Sigma \g_{\CC}}=\Sigma \g_{\CC}\oplus \CC\frac{d}{dx}\oplus \CC\frac{d}{dy}$. The identification between this $(-\tau,1)$-level set and the space of $G$-connections $\mathcal{A}$ is $\Sigma G_{\CC}$-equivariant. Consequently, the first-order differential operator $\frac{d}{dy}-(\tau \dd+Y)$ in the $y$-direction, with values in $L\g_{\CC}$, can be identified with the Cauchy-Riemann operator $\bar{\partial}_Y=\bar{\partial}+\frac{1}{2\tau_2}Y d\bar{z}$. Thus, $\mathcal{D}_{\alpha}^0=\frac{d}{dy}-(\tau \dd+\iota_{\partial_{z}^{\#}}\widetilde{A}_{\alpha})$, when ignoring the central part, can be viewed as a $\bar{\partial}$-operator on $\Sigma_{\tau}$ with values in $\g_{\CC}$, parameterized by $\LLX$.

Let $\Psi_A:LP\to \mathcal{A}$ be the evaluation of the connection $A$. It is a $\Sigma G$-equivariant map.
\begin{definition}
A loop $\gamma\in \LLX$ is called gauged-poly-stable if $[\Psi_A(\gamma_P)]$ is a poly-stable $\Sigma G_{\CC}$-orbit in $\mathcal{A}$, 
where $\gamma_P$ is a lift of $\gamma$ in $L^2P$.
\end{definition}
This definition is independent of the choice of lift $\gamma_P$, as they only differ by the action of $\Sigma G$.
\begin{proposition}\label{Convergence at gauged-poly-stable}
Elliptic holonomy $Ehol_{A}(\mathcal{H})$ converges at gauged-poly-stable double loops.
\begin{proof}
For $\gamma\in \LLX$ is gauged-poly-stable, there exits $g\in \Sigma G_{\CC}$ such that 
\begin{align}
g\mathcal{D}_{\alpha}^0({\gamma})g^{-1}=\frac{d}{dy}-(\tau \dd+K_{\alpha})+\text{central part},
\end{align}\
where $K_{\alpha}\in \mathfrak{t}_{\CC}$ is constant. Then the holonomy is the exponential 
\be 
hol(\frac{d}{dy}-(\tau \dd+K_{\alpha}))=\exp(\tau \dd+K_{\alpha}).
\ee
Then we have
\begin{align}
\Tr_{\mathcal{H}}[hol(\mathcal{D}_{\alpha}^0({\gamma}))]&=\Tr_{\mathcal{H}}[g^{-1}hol(\frac{d}{dy}-(\tau \dd+K_{\alpha}))g]\cdot  \exp{(k\square)}\\
&=\Tr_{\mathcal{H}}[\exp(\tau \dd+K_{\alpha})]\cdot  \exp{(k\square)} <\infty \nonumber.
\end{align}
$\square$ is the extra contribution of the central part from the adjoint action.
 $\Tr_{\mathcal{H}}[\exp(\tau \dd+K_{\alpha})]$ converges since it is given by the Kac-Weyl character.
\end{proof}
\end{proposition}

\subsubsection{$q$-transport equation and convergence}
If the connection evaluated on the double loop is Holomorphic(of $z$ after the polarization) or real-analytic,
we can prove the convergence.  
Let $LG^h=Hol(\CC^*,G_{\CC})$ be the holomorphic loop group, and let $L\g^h=Hol(\CC^*,\g_{\CC})$ be the holomorphic loop algebra. For $q\in \CC^*$, the map $R_q:LG^h\to LG^h$ represents the complex rotation defined by
\begin{align}
R_q\cdot g(z):=g(qz), \quad g\in LG^h.
\end{align}
The corresponding extended holomorphic loop group is denoted by $\widehat{LG^h}=LG^h\rtimes \CC^*$. We define the $q$-level set as $LG^h_q=\{(g,q)|g\in LG^h\}$. By adjunction, we can define smooth loops into the holomorphic loop group as $L(LG^h)=C^{\infty}(S^1,LG^h)$ and its Lie algebra as $L(L\g^h)$.

For $Y\in L(L\g^h)$, we solve the transport equation for smooth $\psi:\R \to  \widehat{LG^h}$ that $t\to \psi(t)=(g(t),\lambda(t))$:
\begin{align}\label{q transport eq}
\frac{d\psi}{dy}\psi^{-1}(y)=\tau \dd+Y(y), \ \psi(0)=e.
\end{align}
The holonomy $hol(\tau \dd+Y)=\psi(1)$.
\begin{proposition}
    The solution is given by the $q$-deformed iterated integral
\begin{align}\label{hol tau dd+Y}
hol(\tau \dd+Y)=(\sum_{n=0}^{\infty}\int_{\Delta_n}\prod_{i=1}^{n}(R_{q^{t_i}}(Y_{t_i})dt_i),q).
\end{align}
\begin{proof}
Remember we use the multiplication of the semi-product 
\begin{align}
(\dot{\lambda},\dot{g})(\lambda^{-1},R_{\lambda^{-1}}g^{-1})=(2\pi i\tau,A).
\end{align}
From $\dot{\lambda}\lambda^{-1}=2\pi i\tau$, we have $\lambda(y)=q^y$. 
\begin{align}
R_{q^{-y}}(\dot{g})R_{q^{-y}}(g^{-1})=A_y,\
\text{then}\ 
\frac{dg}{dy}g^{-1}=R_{q^y}(A_y).
\end{align}
The solution is given by the iterated integral
\begin{align}
hol(\tau \dd+Y)=\psi(1)=(\sum_{n=0}^{\infty}\int_{\Delta_n}\prod_{i=1}^{n}(R_{q^{t_i}}(Y_{t_i})dt_i),q).
    \end{align}
\end{proof}
\end{proposition}
Let $\widetilde{LG^h}$ be the central extension of $LG^h$ and $\widetilde{LG^h}'$ be the double extended holomorphic loop group, which is the $\CC^*$ semidirect product of $\widetilde{LG^h}$.
Denote by $\widetilde{LG^h_q}$ the $q$-level set of $\widetilde{LG^h}'$.
Following \cite[Theorem 2.2, Lemma 2.3]{Etingof1994SphericalFO} and \cite{Goodman1984StructureAU},
under positive energy representations, 
all elements of $\widetilde{LG^h_q}$ are trace-class operators when $q\in \disc$.
Then for any $\mathcal{H}\in \per^k(LG)$ and $Y\in L(L\g^h)$, 
\begin{align}
\Tr_{\mathcal{H}}[hol(\tau \dd+Y)]< \infty.
\end{align}
\begin{remark}
    Convergence also holds for real-analytic loops.
Let $L\g_{\CC}^{an}$ denote the real-analytic loop algebra, and let $LG^{an}$ denote the corresponding real-analytic loop group.
The Fourier coefficients of a real-analytic element decay exponentially, so there exists a common $\delta>1$ such that any $Y(t)\in L\g_{\CC}^{an}$ can be analytically continued to an annulus $D_{\delta}=\{z\in \CC\mid \frac{1}{\delta}<\abs{z}<\delta \}$ for all $t\in S^1$.
When $|q|<\frac{1}{\delta}$, the above estimate gives $\Tr_{\mathcal{H}}[hol(\tau \dd+Y)]<\infty$.

However, it is quite subtle to work with real-analytic principal bundles.
Moreover, we cannot expect a common $\delta$ for all real-analytic loops.
Therefore, we will not pursue this direction further in this paper and instead restrict to gauged-polystable double loops in the smooth setting.
\end{remark}

\subsection{Pfaffians on \texorpdfstring{$\LLX$}{LLX}}\label{Pfaffians}
The determinant and Pfaffian line bundles over the mapping space $\Sigma X$ are constructed in \cite{Freed1987OnDL,Freed1986DeterminantsTA,Bunke2009StringSA}.
Let $P\to X$ be a principal $G$-bundle equipped with a connection $A$, and let $V$ be a unitary representation of $G$.
The associated vector bundle is $E:=P\times_G V$, equipped with the induced connection $\nabla^E$.
Let $\Sigma$ be a closed oriented surface, and denote by $\Sigma X$ the mapping space.

We endow $\Sigma$ with a complex structure, and fix a spin structure by choosing a square root $\sqrt{K}$ of the canonical line bundle $K$.
Equip $\Sigma$ with a chosen K\"ahler metric.
The evaluation map $ev:\Sigma X \times \Sigma \to X$ is defined by $ev(\gamma,x):=\gamma(x)$.
The pullback $ev^*(E,\nabla^E)$ is a Hermitian vector bundle with connection over $\Sigma X \times \Sigma$.

With this geometric data, there is a family of Dirac operators parametrized by $\Sigma X$.
By the family index theorem, there arises a determinant line bundle equipped with the Bismut-Freed connection and the Quillen metric as follows:
\begin{equation}
\begin{tikzcd}
   & ev^*(E,\nabla^E) \arrow[d] \\
 \Sigma\arrow[r]  & \Sigma \times \Sigma X \arrow[d] \\
   &    \Sigma X
\end{tikzcd}
\Longrightarrow
\begin{tikzcd}
  (\mathrm{DET}(V),\nabla^{\mathrm{DET}(V)},g^{\mathrm{DET}(V)}) \arrow[d] \\
   \Sigma X
\end{tikzcd}
\end{equation}
When $V$ is a real representation, there's the canonical square root of the determinant line bundle called the Pfaffian line bundle, together with the induced connection and the metric, such that 
\begin{align}
    (\mathrm{PF}(V),\nabla^{\mathrm{PF}(V)},g^{\mathrm{PF}(V)})^2\cong
    (\mathrm{DET}(V),\nabla^{\mathrm{DET}(V)},g^{\mathrm{DET}(V)}).
\end{align}

When genus one, the mapping space $\Sigma X$ is the double loop space $\LLX$.
By Atiyah-Singer index theorem, the index of the Dirac operator is always zero.
There's the canonical section $\det_V\in \Gamma(\LLX,\mathrm{DET}(V))$. When $V$ is a real representation, there's the canonical section $\mathrm{pf}_V\in \Gamma(\LLX,\mathrm{PF}(V))$ such that
\begin{align}
\mathrm{pf}_V^2=\dt_V.
\end{align}
For simplicity, we fix the standard flat, unit-volume K\"ahler metric on the elliptic curve $\Sigma_{\tau}$.
 The four spin structures on $\Sigma_{\tau}$ correspond to the four 2-torsion points.

When $G=\mathrm{Spin}(2n)$ and $V$ is the standard real representation given by
\be
\rho:\mathrm{Spin}(2n)\to \mathrm{SO}(2n),
\ee
the curvature of the Pfaffian line bundle is
\begin{align}
R^{\mathrm{PF}(V)}=\int_{\Sigma}\tfrac{1}{2}p_1(E,\nabla^E)\in \Omega^2(\LLX).
\end{align}

In particular, if $(X^{2n},g)$ is a spin Riemannian manifold and $P=P_{\mathrm{Spin}(X)}$ is the spin frame bundle, then the associated bundle $E$ is $TX$ equipped with the Levi-Civita connection $\nabla^g$. 
Fixing the odd spin structure on $\Sigma_{\tau}$ and the standard Kähler metric, we denote the Pfaffian line bundle by
$
(\mathrm{PF},\nabla^{\mathrm{PF}},g^{\mathrm{PF}})\to \LLX,
$
whose curvature is the transgression of $\tfrac{1}{2}p_1(X,g)$. 
The canonical section is denoted as $\mathrm{pf}[\tau]\in \Gamma(\LLX,\mathrm{PF})$. The other three spin structures yield analogous Pfaffian sections.

\subsection{Equivalence of Line Bundles}\label{subsec: EQ of Line Bundles}
In this subsection, 
we demonstrate that the transgression of the lifting gerbe and the determinant (resp.\ Pfaffian) line bundles are all isomorphic to suitable powers of Chern-Simons line bundles.

Conceptually, the equivalence between the transgression of the lifting gerbe and the Chern-Simons line bundle arises because the Chern-Simons line bundle is the 2‑transgression of the Chern-Simons 2‑gerbe.
The isometry between the determinant (and Pfaffian) line bundles and the Chern-Simons line bundle over the mapping space was established in \cite{Bunke2009StringSA}. To relate the geometry of the mapping space to the moduli of holomorphic $G_{\CC}$-bundles, we will utilize the isometry on $[\mathcal{A}/\Sigma G]$ and subsequently pull it back to $\Sigma X$.

\subsubsection{Universal Chern-Simons line bundle on $[\mathcal{A}/\Sigma G]$}\label{universal CS line bundle}
We begin with an arbitrary genus and will later specialize to genus one.
Let $\Sigma$ be a closed oriented surface and $\Sigma X$ be the mapping space.
Define $\mathcal{A}=\Omega^1(\Sigma,\g)$ as the space of $G$-connections, and let $\Sigma G$ be the gauge group acting on $\mathcal{A}$ through gauge transformations.

 In Atiyah-Bott \cite{Atiyah1983TheYE}, $\mathcal{A}$ is equipped with a minimal integral symplectic form $\omega$ such that $\Sigma G$-action is Hamiltonian. The moment map $\mu$ sends the connection to the curvature.
The universal Chern-Simons line bundle $(\mathcal{L},h,\nabla)$ is a $\Sigma G$-equivariant prequantum line bundle over the Hamiltonian $\Sigma G$-space $(\mathcal{A},\omega,\mu)$. We sometimes omit the metric $h$ of the Chern-Simons line bundle since it's the trivial one.
\begin{lemma}{\label{CSdiff}}
For $A$ is connection of principal $G$-bundle $P\to X$ and $g\in \mathcal{G}(P)$ is the gauge transformation. $A^g=g^{-1}\cdot A=Ad_{g^{-1}}(A)+g^{-1}dg$ is the gauge transformation acting on right. We have
\begin{align}
\CCS_{A^g}-\CCS_A=\frac{1}{12\pi}\Tr(g^{-1}dg)^3-d(\frac{1}{4\pi}\Tr(A\wedge dg\cdot g^{-1})). 
\end{align}
The trace here is given by the minimal form.
\end{lemma}
\begin{proposition}\label{wzw=diff of cs}
Let $B$ be a three-dimensional manifold whose boundary $\partial B=\Sigma$ where the pair $(\Tilde{A},\Tilde{g}) \in \Omega^1(B,g)\times C^{\infty}(B,G)$ is the extension of $(A,g)\in \mathcal{A}\times \Sigma G$ such that
$
(\Tilde{A},\Tilde{g})|_{\partial B=\Sigma}=(A,g)
$. The gauged-WZW cocycle is defined as
\begin{align}
W(g,A)=\int_B (\CCS_{\widetilde{A}}-\CCS_{\widetilde{g}\cdot \widetilde{A}}),
\end{align}
which is independent of the choice of extension $(\widetilde{A},\widetilde{g})$.
\end{proposition}
Follows by Lemma \ref{CSdiff} and Stokes formula,
\begin{align}
W(g,A)=\Gamma(g)-\frac{1}{4\pi}\int_{\Sigma}\langle A \wedge g^{-1}dg\rangle,
\end{align}
where $\Gamma:\Sigma G\to \R/2\pi\Z$ is the gerbal holonomy,
$\Gamma(g)=\int_B \widetilde{g}^*\Theta$, $\widetilde{g}$ is the extension of
$g$ to some $B$ whose boundary is $\Sigma$.
\begin{definition}
The cocycle $W(g,A)$ defines a $\Sigma G$-equivariant line bundle 
over $\mathcal{A}$ as follows.
$\mathcal{L}:=\mathcal{A}\times \CC$, where $\Sigma G$-action is
\begin{align}
g(A,c)=(gA,e^{iW(g,A)}c), \text{ for } g\in \Sigma G, A\in \mathcal{A}.
\end{align} 
We endow $\mathcal{L}$ with the trivial metric $h$.
The unitary connection $\nabla=d-\sqrt{-1}\theta$,
where $\theta\in \Omega^1(\mathcal{A},\R) $ is defined as, for
$A\in \mathcal{A}$ and $a\in T_A\mathcal{A}$,
\begin{align}
\theta_A(a)=\frac{1}{4\pi}\int_{\Sigma}\langle A\wedge a\rangle. 
\end{align}
We call $(\mathcal{L},h,\nabla)$ the universal Chern-Simons line bundle over $\mathcal{A}$.
\end{definition}
\begin{proposition}[{\cite{FREED1995237}}]
The universal Chern-Simons line bundle $(\mathcal{L},h,\nabla)$ is
a $\Sigma G$-equivariant prequantum line bundle over the Hamiltonian $\Sigma G$-space
$(\mathcal{A},\omega,\mu)$.
\end{proposition}
\subsubsection{Pull-back to the mapping space $\Sigma X$} 
On the homotopy level, $[\mathcal{A}/\Sigma G]\sim \Sigma BG$. The universal Chern-Simons line bundle $(\mathcal{L},\nabla,h)\to \mathcal{A}$ is the geometric realization of the 2-transgression of the universal Chern-Simons 2-gerbe on $BG$. For universal constructions on $BG$, see \cite{Gomi2001TheFO,Carey2004BundleGF}.
Let $f_P:X\to BG$ be the classifying map and $\Sigma f_P:\Sigma X\to \Sigma BG$.
Since 2-transgression is functorial, the pullback $\Sigma f_P^*\mathcal{L}\to \Sigma X$ is the 2-transgression of the Chern-Simons 2-gerbe on $X$.

We now give a geometric description of the Chern-Simons line bundle on $\Sigma X$.
Consider the principal $\Sigma G$-bundle $\Sigma P\to \Sigma X$. The connection $A$ on $P$ induces
a $\Sigma G$-equivariant map
\begin{align}
    \Psi_A:\Sigma P\to \mathcal{A}, \quad \Psi_A(\gamma)=\gamma^*A.
\end{align}
The pullback $\Psi_A^*(\mathcal{L},\nabla,h)\to \Sigma P$ is a $\Sigma G$-equivariant line bundle with an invariant connection and trivial metric.
It descends to the free quotient
$\Sigma P/\Sigma G=\Sigma X$.
We call it the Chern-Simons line bundle $\mathcal{L_{\CCS}}$ on $\Sigma X$, together with the pull-back connection $\nabla_{\CCS}^A$ and the trivial metric $h$.
The Chern-Simons line bundle $(\mathcal{L}_{\CCS},\nabla_{\CCS}^A)\to \Sigma X$ is given by the local data $\{\Sigma U_{\alpha},h_{\alpha\beta},\Theta_{\alpha}\}$:
\begin{itemize}
    \item The transition function $h_{\alpha\beta}:\Sigma U_{\alpha}\cap \Sigma U_{\beta}\to \R/2\pi \Z$,
    for $x:\Sigma \to U_{\alpha}\cap U_{\beta}$,
    $$
    h_{\alpha\beta}(x):=\exp{(iW(x\circ g_{\alpha\beta}},x^*A_{\beta}).
    $$
    \item 
    The connection form 
    $$
\Theta_{\alpha}=i\int_{\Sigma}\CCS(A_{\alpha}).
    $$
\end{itemize}
\subsubsection{Transgression of lifting gerbe is Chern-Simons line bundle on $\LLX$}
In the genus one case, the mapping space $\Sigma X$ is the double loop space $\LLX$. 
\begin{proposition}\label{1+1 transgression=2 transgression}
The transgression line bundle of the lifting gerbe $\mathcal{L}(\mathcal{G}_P,\nabla^B)\to \LLX$ coincides with the Chern-Simons line bundle $(\mathcal{L}_{\CCS},\nabla_{\CCS}^A)\to \LLX$.
\end{proposition}
\begin{proof}
To verify that $\mathcal{L}(\mathcal{G}_P,\nabla^B) = (\mathcal{L}_{\CCS},\nabla_{\CCS}^A)$, it suffices to show that they have the same transition functions
Let $x^{\vee}:S^1\to \LX$ be the adjoint of $x:S^1\times S^1\to X$.
\[
h'_{\alpha\beta}(x)
=\exp\!\Bigl(i\!\int_{S^1}(x^{\vee})^*u_{\alpha\beta}\Bigr),
\quad
h_{\alpha\beta}(x)
=\exp\!\bigl(i\,W(x\circ g_{\alpha\beta},x^*A_{\beta})\bigr).
\]
Thus it suffices to show that
\begin{equation}\label{eq:hab-equality}
\int_{S^1}(x^{\vee})^*u_{\alpha\beta}
\equiv
W(x\circ g_{\alpha\beta},x^*A_{\beta})
\quad (\mathrm{mod}\;2\pi\Z).
\end{equation}
Choose an extension $\widetilde{x}:D\times S^1\to X$ of $x$, and let
$\widetilde{x}^{\vee}:D\to \LX$ denote the adjoint map. By Stokes' formula,
\be
\int_{S^1}(x^{\vee})^*u_{\alpha\beta}
=\int_D(\widetilde{x}^{\vee})^*(du_{\alpha\beta}).
\ee
Using \cref{K diff =u}, we have $du_{\alpha\beta}=K_{\alpha}-K_{\beta}$, and recall that
$K_{\alpha}=\int_{S^1}\CCS(A_{\alpha})$. Therefore the right side
\be
\int_D(\widetilde{x}^{\vee})^*(du_{\alpha\beta})
=
\int_{D\times S^1}\widetilde{x}^*\!\bigl(\CCS(A_{\alpha})-\CCS(A_{\beta})\bigr).
\ee
By the definition of the gauged WZW cocycle, it is precisely
$W(x\circ g_{\alpha\beta},x^*A_{\beta})$. Hence
\eqref{eq:hab-equality} holds and consequently $h'_{\alpha\beta}(x)=h_{\alpha\beta}(x)$.
\end{proof}

\subsubsection{Determinant(Pfaffian) line bundle is Chern-Simons line bundle}
We apply the gauge-equivariant family index theorem to identify the determinant line bundle with a suitable power of the Chern--Simons line bundle. 
The construction of the determinant line bundle, together with the Quillen metric and the Bismut--Freed connection, is standard (see \cite{Bismut1986TheAO}). 

In our setting (families of Cauchy--Riemann operators parametrized by the space of connections), the determinant line bundle goes back to the foundational work of Quillen and Atiyah--Singer \cite{Quillen1985DeterminantsOC,Atiyah1984DiracOC}. 
For real representations, one obtains Pfaffian line bundles; their basic properties and their relation to anomaly line bundles are discussed in \cite{Freed1987OnDL,Freed1986DeterminantsTA}.

We first recall the notion of the \emph{level} of a representation.
For a finite-dimensional representation $\rho:\g \to \mathfrak{gl}(V)$,
the Dynkin index $d_V$ is the scalar characterized by
\begin{align}
    \tr\bigl(\rho(x)\rho(y)\bigr)=d_V\langle x,y\rangle,
\end{align}
for all $x,y\in \g$.
In particular, the adjoint representation has Dynkin index
$d_{\mathrm{ad}}=2h^{\vee}$, where $h^{\vee}$ denotes the dual Coxeter number of $\g$.
It is known that a spin structure on $\Sigma$ is determined by a choice of square root $\sqrt{K}$ of the canonical line bundle.

We now summarize the constructions and results below.
 For further details, we refer to \cite{Preparation} for the discussion of $\Sigma G$-equivariance and the isometry with Chern--Simons line bundles.

 In order to define the the determinant line bundle over $\mathcal{A}$, 
we fix the  K\"{a}hler metric of $\Sigma$,
the spin structure $\sqrt{K}$,
and the unitary representation  $V$ of $G$. 
With these geometric information, $\mathcal{A}$ parameterizes the $\Sigma G$-equivariant
family of Dirac operators.
Given finite dimensional unitary representation $ V\in \text{Rep}(G)$,
there's the associated vector bundle of the unitary representation. Since it's smoothly trivial, we still denote it by $V$.
There's the splitting of the bundle connection 
$
\nabla_A=\nabla_A^{1,0}+\nabla_A^{0,1},
$
where $\nabla_A^{0,1}$ serves as the $\bar{\partial}$-operator defining the holomorphic structure. The corresponding holomorphic vector bundle 
is denoted by $V_A$.
The Dirac operator coupled with vector potential is
\begin{align}
 \slashed {\partial}_A^V:\Omega^0(\Sigma,V\otimes \sqrt{K})\to \Omega^{0,1}(\Sigma,V\otimes \sqrt{K}).
\end{align}
and its adjoint
$
 \slashed {\partial}_A^{V,*}:\Omega^{0,1}(\Sigma,V\otimes\sqrt{K}) \to \Omega^{0}(\Sigma,V\otimes \sqrt{K}).
$
The kernel and the co-kernel of the Dirac operator is 
\begin{align}
\mathrm{Ker}(\slashed {\partial}_A^V)=H^0(\Sigma,V_A\otimes \sqrt{K}) \quad
\mathrm{Ker}(\slashed {\partial}_A^{V,*})=H^1(\Sigma,V_A\otimes \sqrt{K}).
\end{align}
The index 
$
\mathrm{Ind}(\slashed {\partial}_A^V)=\dim H^0(\Sigma,V_A\otimes \sqrt{K})-\dim H^1(\Sigma,V_A\otimes \sqrt{K})
$
by Riemann-Roch theorem is
\begin{align}
 \mathrm{Ind}(\slashed {\partial}_A^V)=\chi(\Sigma,V_A\otimes \sqrt{K} ) =-r(V)(1-g).
\end{align}
$r(V)$ is the dimension of representation $V$ and $g$ is the genus of $\Sigma$.
Especially, when genus $g=1$, the index is always zero.

\begin{theorem}[{\cite[Theorem 2.5]{Preparation}}]{\label{detpf}}      
Given the spin structure $\sqrt{K}$, the K\"{a}hler metric of $\Sigma$,
and a finite dimensional unitary representation $V$, there's the determinant line bundle $\mathrm{DET}(V)$ over $\mathcal{A}$ with 
the Quillen metric
 $g^{\mathrm{DET}(V)}$ and the Bismut-Freed connection $\nabla^{\mathrm{DET}(V)}$. 
 The determinant line bundle has following properties:
\begin{enumerate}
    \item $(\mathrm{DET}(V),g^{\mathrm{DET}(V)},\nabla^{\mathrm{DET}(V)})$ is a $\Sigma G$-equivariant  prequantum line bundle of level $d_V$ over the  Hamiltonian $\Sigma G$ space $(\mathcal{A},\omega,\mu)$.
    \item $(\mathrm{DET}(V),g^{\mathrm{DET}(V)},\nabla^{\mathrm{DET}(V)})$ is isomorphic to $(\mathcal{L},h,\nabla)^{d_V}$.
    \item When genus one, the canonical section $\det_V\in H^0_{\Sigma G}(\mathcal{A},\mathrm{DET}(V))$ can be defined, which has the property that
    \begin{align*}
    \norm{\mathrm{det}_V(A)}^2_{g^{\mathrm{DET}(V)}}=\mathrm{det}_{\zeta}(\slashed\partial_A^{V,*}\slashed\partial^V_A),
    \end{align*}
    where $\slashed\partial_A^V$ is the Dirac operator twisted by the vector potential $A$ on $V$ and $\slashed\partial_A^{V,*}$ is its formal adjoint.
\item  When $V$ is a real representation,
there's the canonical square root of the determinant line bundle called the
Pfaffian line bundle $(\mathrm{PF}(V),g^{\mathrm{PF}(V)},\nabla^{\mathrm{PF}(V)})$ with similar properties as above,
whose level is $d_V/2$.
\end{enumerate}
\end{theorem}
The above theorem is the universal version establishing that the determinant (Pfaffian) line bundle is isomorphic to the Chern--Simons line bundle.
The isometry on the mapping space follows directly, since both line bundles are pullbacks from $[\mathcal{A}/\Sigma G]$.
\subsection{Elliptic Atiyah-Witten Formula on Moduli}\label{subsec:Elliptic Atiyah-Witten on MG}
In this subsection, we derive a general formula for the pushdown of determinants on the coarse moduli $M_G[\tau]$.
For $G=\mathrm{Spin}(2n)$ at level one, we establish the elliptic Atiyah-Witten formula on $M_G[\tau]$.

We demonstrate that under an isometry between the Pfaffian line bundle and the Chern-Simons line bundle on $M_G[\tau]$,
the pushdown of the four Pfaffian sections (one for each spin structure) coincides with the modified Kac-Weyl characters of the four virtual level-one representations $\{S_{ij}\}$.
\subsubsection{Coarse Moduli and Genus One Conformal Blocks}\label{genus one CS and CB}
We begin by recalling the K\"ahler geometric quantization of the Atiyah-Bott stack.
For further details, see \cite{Preparation}.

The moduli stack $Bun_{G}(\Sigma_{\tau})$ of holomorphic principal $G_{\CC}$-bundles over the elliptic curve $\Sigma_{\tau}$ admits a presentation as the Atiyah-Bott stack $\mathcal{A}/\Sigma G_{\CC}$, where $\mathcal{A}$ is the space of connections, 
viewed as an infinite-dimensional Kähler manifold, and $\Sigma G_{\CC}$ is the complexified gauge group.
Let $M_G[\tau]$ denote the coarse moduli of $Bun_G(\Sigma_{\tau})$, \ie the moduli space of $S$-equivalence classes of semistable holomorphic $G_{\CC}$-bundles on $\Sigma_{\tau}$.
Complex-analytically, $M_G[\tau]$ can be described via the infinite-dimensional geometric invariant theory (GIT) quotient.
The theory of the Yang-Mills flow on surfaces establishes the Kempf-Ness theorem, which identifies the GIT quotient with the Kähler quotient $\mathcal{A}//\Sigma G$.

Fortunately, the genus one coarse moduli $M_G[\tau]$ admits a much simpler explicit description since it has a global quotient model (see, e.g., \cite{Friedman1997PrincipalGB,etingof_central_1994}).
Geometrically, the global model can be
constructed from constant Cartan connections $\mathfrak{t}_{\CC}$ (see \cite{Preparation}).
The restricted gauge group on $\mathfrak{t}_{\CC}$ is the affine Weyl group $W_{\mathrm{aff}}=(\Lambda \oplus \tau \Lambda) \rtimes W$, where $\Lambda$ is the co-root lattice of $G$ and $W$ is the Weyl group.
The global model is then given by the quotient of constant Cartan connections by the affine Weyl group 
\begin{align}\label{global quotient model}
M_G[\tau]=\mathfrak{t}_{\CC}/W_{\mathrm{aff}}=(\Sigma_{\tau}\otimes_{\Z}\Lambda)/W.
\end{align}
Consequently, the genus one conformal block is given by Weyl group invariant combinations of theta functions, as in \cite{Looijenga1976RootSA}.
\begin{definition}\label{Conformal block}
The genus one conformal block at $\tau \in \mathbb{H}$ is 
\begin{align}
V_k(G)|_{\tau}=H^0(M_G[\tau],\widetilde{\mathcal{L}}^k).
\end{align}  
Here $\widetilde{\mathcal{L}}$ is the descent of the Chern-Simons line bundle $\mathcal{L}$ to the coarse moduli $M_G[\tau]$.
\end{definition}
Since 
$
M_G[\tau]=(\Sigma_{\tau}\otimes_{\Z}\Lambda)/W,
$
the genus-one conformal block can be identified with
\begin{align}
V_k(G)|_{\tau}=H^0((\Sigma_{\tau}\otimes_{\Z}\Lambda)/W,\widetilde{\mathcal{L}}^k)=H^0(\Sigma_{\tau}\otimes_{\Z}\Lambda,\mathcal{L}_0^k)^W,
\end{align}
where the basic line bundle $\mathcal{L}_0$ over the abelian variety $\Sigma_{\tau}\otimes_{\Z}\Lambda$ is isomorphic to the restriction of the Chern-Simons line bundle $\mathcal{L}$ 
to constant Cartan connections $\mathfrak{t}_{\CC}\subset \mathcal{A}$.
\subsubsection{Push-down of Determinants}\label{pf as theta}
We present a general formula for the push-down of determinants associated with unitary representations. 
        Throughout, we  always assume that the elliptic curve $\Sigma_{\tau}$ is equipped with the standard flat K\"ahler metric of unit volume.

Denote the Jacobian of $\Sigma_{\tau}$ by $\widehat{\Sigma_{\tau}}$.
The push-down of the determinant line bundle $\mathrm{DET}(V)$ is given by the Theta divisor $\Theta(V)$, which is described as follows:
\begin{definition}[{\cite{Drzet1989GroupeDP}}]\label{Theta line bundle}
    Given a unitary representation $V\in \mathrm{Rep}(G)$, 
for any $\delta \in \widehat{\Sigma_{\tau}}$, define
\begin{align}
    D_{\delta}=\{E\in M_G[\tau] \mid H^0(\Sigma_{\tau},\delta \otimes E(V))\neq 0\},
\end{align}
where $E(V)$ is the associated bundle of $E$ with respect to the representation $V$.
$D_{\delta}$ is a Cartier divisor and does not depend on the choice of $\delta$. 
\begin{align}
    \Theta(V)=\mathcal{O}_{\Sigma}(D_{\delta})
\end{align}
is called the Theta divisor of the representation $V$.
\end{definition}
To show that $\Theta(V)$ is the push-down of the determinant line bundle $\mathrm{DET}(V)$ to the coarse moduli $M_G[\tau]$, 
we consider 
the universal holomorphic vector bundle $\mathcal{E}\to \Sigma_{\tau}\times M_G[\tau]$, which restricts to $E\in M_G[\tau]$ as the associated holomorphic vector bundle $E(V)$.
By the holomorphic family index theorem, we have the push-down determinant line bundle $\mathrm{DET}(V)\to M_G[\tau]$, whose canonical section $\dt_V$ has the zero locus $D_{\delta}$.
Since $M_G[\tau]$ is singular, we need to work on the Weyl group equivariant line bundle on the abelian variety $\Sigma_{\tau}\otimes_{\Z}\Lambda$ and then descend to the coarse moduli $M_G[\tau]$.

\subsubsection*{\textbf{On Jacobian $\widehat{\Sigma_{\tau}}$}}
To begin, we first consider the (Theta) determinant line bundle over the Jacobian $\widehat{\Sigma_{\tau}}$ (see \cite{Freed1987OnDL,AlvarezGaum1986ThetaFM}).

There exists the Poincar\'e line bundle $\mathcal{P}$ on $\Sigma_{\tau}\times \widehat{\Sigma_{\tau}}$, whose restriction to $\Sigma_{\tau}\times\{L\}$ is the line bundle $L\in\widehat{\Sigma_{\tau}}$ itself. Given $\delta\in \widehat{\Sigma_{\tau}}$, we can twist the Poincar\'e line bundle by $\delta$ as
$
\mathcal{P}_{\delta}:=\mathcal{P}\otimes\delta,
$
whose restriction to $\Sigma_{\tau}\times\{L\}$ is $L\otimes\delta$. 
By the holomorphic family index theorem, 
\begin{equation}\label{eq:det-theta}
\begin{tikzcd}
                      & \mathcal{P}_{\delta} \arrow[d] \\
\Sigma_{\tau}\arrow[r] & \Sigma_{\tau}\times \widehat{\Sigma_{\tau}} \arrow[d] \\
                      & \widehat{\Sigma_{\tau}}
\end{tikzcd}\longrightarrow 
\begin{tikzcd}
\mathcal{L}_{\delta} \arrow[d]\\
\widehat{\Sigma_{\tau}}
\end{tikzcd}
\end{equation}
Every $L\in \widehat{\Sigma_{\tau}}$ gives the $\bar{\partial}$-operator $\bar{\partial}_{\mathcal{L}\otimes \delta}$. 
 $\mathcal{L}_{\delta}$ is the corresponding determinant line bundle.
 The canonical section 
 $\dt_{\delta}\in H^0(\widehat{\Sigma_{\tau}},\mathcal{L}_{\delta})$ whose zero locus is the Theta divisor
$D_{\delta}=\{x\in \widehat{\Sigma_{\tau}}|H^0(\Sigma_{\tau},\delta \otimes x) \neq 0 \}=-\delta$.
The Abel-Jacobi map that identifies $\Sigma_{\tau}$ with its Jacobian $\widehat{\Sigma_{\tau}}$ can be explicitly described as follows.
Given  $a+\tau b$ as the coordinate of $\Sigma_{\tau}$, where $a,b\in \R/\Z$.
 The point $a+\tau b$ maps to $\delta=A(a,b)\in \widehat{\Sigma_{\tau}}$ corresponding to the character 
$m\tau+n \to e^{2\pi\sqrt{-1}(ma-nb)}$.

\subsubsection*{\textbf{On Abelian Variety $(\widehat{\Sigma_{\tau}}\otimes_{\Z}\bigwedge)$}}

We start with $\Delta(\lambda)$, the irreducible representation of $G$ given by the highest weight $\lambda$ of $\g$.

Let $\mathcal{E}_{\lambda}\to \Sigma_{\tau}\times M_G[\tau]$ be the universal holomorphic vector bundle associated with $\Delta(\lambda)$.
Let $\Theta(\lambda)$ be the Theta divisor of $\Delta(\lambda)$ as in Definition \ref{Theta line bundle}.
We describe the geometry on $\widehat{\Sigma_{\tau}}\otimes \bigwedge $ before the quotient by the Weyl group $W$.

The weight $\mu\in \mathrm{Hom}_{\Z}(\bigwedge,\Z)$ induces the map 
\begin{align}
 \mu:(\widehat{\Sigma_{\tau}}\otimes_{\Z}\bigwedge)\to \widehat{\Sigma_{\tau}}.
\end{align}
Fix $\delta\in \widehat{\Sigma_{\tau}}$, we use the weight $\mu$ to pull back the twisted Poincare line bundle $\mathcal{P}_{\delta}$
\begin{equation}
\begin{tikzcd}
    &  \mu^*\mathcal{P}_{\delta} \arrow[r]\arrow[d]& \mathcal{P}_{\delta}\arrow[d]\\ 
 \Sigma_{\tau}\arrow[r]   & \Sigma_{\tau}\times (\widehat{\Sigma_{\tau}}\otimes_{\Z}\bigwedge) \arrow[r,"Id\times \mu"]\arrow[d] & \Sigma_{\tau}\times \widehat{\Sigma_{\tau}} \arrow[d]\\
    & (\widehat{\Sigma_{\tau}}\otimes_{\Z}\bigwedge) \arrow[r,"\mu"]& \widehat{\Sigma_{\tau}} 
\end{tikzcd}
\rightarrow
\begin{tikzcd}
\mu^*\mathcal{L}_{\delta} \arrow[d] \arrow[r]& \mathcal{L}_{\delta} \arrow[d]\\ 
(\widehat{\Sigma_{\tau}}\otimes_{\Z}\bigwedge) \arrow[r,"\mu"]&  \widehat{\Sigma_{\tau}}
\end{tikzcd}  
\end{equation}
By the weight space decomposition of the highest weight module
\begin{align}
    \Delta(\lambda)=\bigoplus_{\mu\preceq\lambda}\Delta(\lambda)_{\mu}
\end{align}
where $\Delta(\lambda)_{\mu}$ is the weight space of weight $\mu$ with dimension $m_{\mu}$.
We can describe the universal holomorphic vector bundle $\mathcal{E}_{\lambda}\to \Sigma_{\tau}\times M_G[\tau]$ before taking the quotient as the direct sum of pullback Poincaré line bundles
\be V(\lambda)=\bigoplus_{\mu\preceq \lambda}m_{\mu}\mu^*\mathcal{P}\to \Sigma_{\tau}\times (\widehat{\Sigma_{\tau}}\otimes_{\Z} \bigwedge).
\ee 
Let $V_{\delta}(\lambda)=V(\lambda)\otimes \delta$ be the twist of $V(\lambda)$ by $\delta$, 
which is equivalent to $V_{\delta}(\lambda)=\bigoplus_{\mu\preceq \lambda}m_{\mu}\mu^*\mathcal{P}_{\delta}$.
\be
\begin{tikzcd}
                       &  V_{\delta}(\lambda) \arrow[d] \\
\Sigma_{\tau} \arrow[r] & \Sigma_{\tau}\times(\widehat{\Sigma_{\tau}}\otimes_{\Z}\bigwedge) \arrow[d]\\
                       & (\widehat{\Sigma_{\tau}}\otimes_{\Z}\bigwedge)
\end{tikzcd}
\rightarrow
\begin{tikzcd}
                 \mathrm{DET}(\lambda) \arrow[d] \\
                 (\widehat{\Sigma_{\tau}}\otimes_{\Z}\bigwedge)
\end{tikzcd}
\rightarrow
\begin{tikzcd}
                    \Theta(\lambda) \arrow[d] \\
                    M_G[\tau]
\end{tikzcd}
\ee
The determinant line bundle is given by
\be
\mathrm{DET}(\lambda)=\bigotimes_{\mu\preceq\lambda}(\mu^*\mathcal{L}_{\delta})^{\otimes m_{\mu}},
\ee
which is independent of $\delta$. Also, $\mathrm{DET}(\lambda)$ is $W$-equivariant.
The canonical section $\dt_{\delta}^{\Delta(\lambda)}\in H^0(\widehat{\Sigma_{\tau}}\otimes_{\Z}\bigwedge,\mathrm{DET}(\lambda))$ is defined by
\be \label{splitting}
\dt_{\delta}^{\Delta(\lambda)}=\bigotimes_{\mu \preceq \lambda}(\mu^*\dt_{\delta})^{\otimes m_{\mu}},
\ee
which is $W$-invariant and descends to a section in the conformal block $V_{d_{\lambda}}(G)$, where $d_{\lambda}$ denotes the Dynkin index of the dominant weight $\lambda$.

To display a formula of push-down determinants, we pass to the universal cover which is the space of constant Cartan connections $\mathfrak{t}_{\CC}=\CC\otimes_{\Z}\bigwedge \to \widehat{\Sigma_{\tau}}\otimes_{\Z}\bigwedge$.
 Given $z\in \mathfrak{t}_{\CC}$ and weight $\mu$, set $\mu(z)$ as the image of $z$ under the map $\mu:\mathfrak{t}_{\CC}\to \CC$.
 The push-down determinant is defined via restriction to constant Cartan connections.

\begin{proposition}\label{det as theta}
 Let $\mathrm{res}(\dt^{\Delta(\lambda)}_{a,b})$ be the restriction of the determinant twisted by $\delta=A(a,b)\in \widehat{\Sigma_{\tau}}$ to  $\mathfrak{t}_{\CC}$.
 It is a holomorphic section of the corresponding determinant line bundle.
For $z\in \mathfrak{t}_{\CC}$, up to a phase factor,
\be
    \mathrm{res}(\dt^{\Delta(\lambda)}_{[a,b]})(z)\sim 
      \prod_{\mu \preceq \lambda}\Big\{\frac{\vartheta\!\begin{bmatrix}a+\frac{1}{2}\\ b+\frac{1}{2}\end{bmatrix}\!(\mu(z),\tau)}{\eta(\tau)}\Big\}^{m_{\mu}}.
\ee
For the expressions of theta functions with characteristics $(a,b)$, see Subsection \ref{theta functions}.
\begin{proof}
By \cref{splitting},
\be 
\mathrm{res}(\dt^{\Delta(\lambda)}_{[a,b]})=\prod_{\mu \preceq \lambda}\mu^*
(\dt_{\delta})^{m_{\mu}}.
\ee
We only need to show that, under the covering map $\pi:\CC \to \widehat{\Sigma_{\tau}}$ of $\Sigma_{\tau} \to \widehat{\Sigma_{\tau}}$,
\be
\pi^*\dt_{\delta}(z)\sim \frac{\vartheta\!\begin{bmatrix}a+\frac{1}{2}\\ b+\frac{1}{2}\end{bmatrix}\!(z,\tau)}{\eta(\tau)}.
\ee
For $\delta$ is trivial one,
$
\pi^*\dt_0(z)\sim \frac{\theta_{11}(z,\tau)}{\eta(\tau)}
$
is given in \cite[Theorem 4.1]{Ray1973AnalyticTF}.
For arbitrary $\delta=A(a,b)$, 
the formula is similarly given by comparing the zeros (see \cite{AlvarezGaum1986ThetaFM}).
\end{proof}
\end{proposition}
\begin{corollary}
 Let $(a,b)$ be $2$-torsion points, in this case, we write $\mathrm{res}(\dt^{\Delta(\lambda)}_{i,j})$ 
as the determinant with respect to the spin structure determined by the $2$-torsion point $\frac{i+\tau j}{2},i,j\in \{0,1\}$.
\be
        \mathrm{res}(\dt^{\Delta(\lambda)}_{ij})(z)\sim \prod_{\mu \preceq \lambda}\frac{\theta_{1-i,1-j}(\mu(z),\tau)}{\eta(\tau)}.
\ee
\end{corollary}
Consequently,
when $V$ is a real representation and $\delta$ is a $2$-torsion point,
there exists a corresponding Pfaffian line bundle. Taking the square root of both sides yields the formulas for the push-down Pfaffian sections.
\begin{remark}
In \cite[Section 5]{Axelrod1991GeometricQO}, the authors computed the Pfaffian for the adjoint representation of a general Lie group $G$.

A key feature in this case is that the adjoint representation has a constant-dimensional kernel arising from the Cartan subalgebra, 
leading to the vanishing of the Pfaffian section. 
After removing this kernel, one encounters a $\mathbb{Z}_2$-anomaly, once corrected, the section becomes a Weyl-group \textbf{anti-invariant} theta function of level $h^{\vee}$. The Pfaffian of the adjoint representation plays a significant role in metaplectic quantization (Coxeter-number shifting) and in heat-operator expressions for the Hitchin connection.
\end{remark}

\subsubsection{Elliptic Atiyah-Witten formula on $M_G[\tau]$ for \texorpdfstring{$G=\mathrm{Spin}(2n)$}{G=Spin(2n)}}
We now consider $G=\mathrm{Spin}(2d)$ and its standard representation $\rho:\mathrm{Spin}(2d)\to \mathrm{SO}(2d)$, which is a real representation with Dynkin index $d_{\rho}=2$.  

For the spin structure given by $(i,j)\in \Z_2\times \Z_2$,
the Pfaffian line bundle $(\mathrm{PF}_{ij},\nabla^{\mathrm{PF}}_{ij},g^{\mathrm{PF}}_{i,j})$ is isomorphic to the level-one Chern-Simons line bundle $(\mathcal{L},\nabla,h)$ via the isometry $\Phi_{ij}$.

Let $\chi_{ij}$ denote the modified Kac-Weyl character of the level-one virtual representation $S_{ij}\in \per^1(L\mathrm{Spin}(2n))$.
\begin{theorem}\label{AW formula on MG by computation}
    Under the induced isometry $\Phi_{ij}'$ on $M_G[\tau]$, there is an identification between the push-down Pfaffian section and the modified Kac-Weyl character of the level-one virtual representation $S_{ij}\in \per^1(L\mathrm{Spin}(2n))$:
\be 
\Phi_{1-i,1-j}'(\mathrm{res}(\pf_{1-i,1-j}))=\chi_{i,j}.
\ee
\end{theorem}
\begin{proof}
When restricted to constant connections $\mathfrak{t}_{\CC}=\CC^n$, the push-down Pfaffian section corresponding to the spin structure determined by the $2$-torsion point $\frac{i+\tau j}{2}$ is given by
\begin{align}
\mathrm{res}(\pf_{i,j})(z_1,\cdots z_n,\tau)\sim \prod_{1\leq k\leq n}\frac{\theta_{1-i,1-j}(z_k,\tau)}{\eta(\tau)}.
\end{align} 
Recall from Proposition \ref{characters are theta functions} that the modified Kac-Weyl characters of the four level-one virtual representations are 
\be 
\chi_{S_{ij}}=\prod_{1\leq k\leq n}\frac{\theta_{i,j}(z_k,\tau)}{\eta(\tau)}.
\ee 
Therefore, up to a phase factor, we have
$
\chi_{S_{ij}}\sim \mathrm{res}(\pf_{1-i,1-j}).
$
The isometry $\Phi_{ij}$ can be chosen by fixing this phase factor appropriately.
\end{proof}

\subsection{Elliptic Atiyah-Witten Formula on \texorpdfstring{$\LLX$}{LLX}}\label{subsec:Elliptic Atiyah-Witten}
We now provide a proof of \cref{elliptic AW formula}.
We consider the special case in which $(P,A)$ is the spin frame bundle equipped with its spin connection on a Riemannian spin manifold $(X^{2n},g)$.

As shown previously, the loop spinor gerbe module $\mathcal{S}_{\LX}\to \LX$ is the associated gerbe module bundle of the Fock representation
$
\mathcal{F}=S^+-S^-
$.
Its elliptic holonomy $Ehol_{S^+-S^-}[\tau]$ serves as the natural choice for the holonomy of the loop spinor bundle on the loop space.

For convenience, we denote by $(\mathrm{PF},\nabla^{\mathrm{PF}},g^{\mathrm{PF}})$ the Pfaffian line bundle corresponding to the odd spin structure.
We state and prove the elliptic Atiyah–Witten formula only for the odd spin structure; the proofs for the other three even spin structures follow by entirely analogous arguments.
\subsubsection{Global transformation property}
We first show that the transformation of the elliptic holonomy under global gauge transformations aligns with the gauged-WZW cocycle.
\begin{proposition}
For $\mathcal{H}\in \per^k(LG)$, let $g\in \mathcal{G}(P)$ be the global gauge transformation of $P\to X$. The elliptic holonomy transforms as follows:
\be 
Ehol_{gA}(\mathcal{H})=e^{iW(\hat{g},\hat{A})}Ehol_A(\mathcal{H}),
\ee
where $\hat{g}$ and $\hat{A}$ are the evaluations on double loops. 
In other words,
when stated locally, let $g_{\alpha}:U_{\alpha}\to G$ be the local gauge transformation, and $A_{\alpha}$ be the local connection form. 
Fix $\gamma \in L^2U_{\alpha}$, we have
\be 
Ehol_{gA}(\mathcal{H})(\gamma)|_{L^2U_{\alpha}}=e^{ikW(\gamma^*g_{\alpha},\gamma^*A_{\alpha})}Ehol_A(\mathcal{H})(\gamma)|_{L^2U_{\alpha}}.
\ee
\end{proposition}
\begin{proof}
The elliptic holonomy locally is given by 
\be
Ehol_A(\mathcal{H})(\gamma)|_{L^2U_{\alpha}}=f_{\alpha}(\gamma)=\Tr_{\mathcal{H}}hol_{\gamma}\left(\frac{d}{dy}-[\iota_{\partial_y}\widetilde{A_{\alpha}}]-[\tau \dd-\tau \iota_{\partial_x}\widetilde{A_{\alpha}}]\right).
\ee
Recall that $\widetilde{A_{\alpha}}=\widehat{A_{\alpha}}+c_{\alpha}\cc \in \Omega^1(LU_{\alpha},\widetilde{L\g})$, where the central part is defined as $c_{\alpha}=\int_{S^1}\langle \iota_K\widehat{A_{\alpha}}, \widehat{A_{\alpha}}\rangle dt$.
By convention, we use the same symbols to denote their pullbacks to the double loop space.
Consider
 $\widetilde{g_{\alpha}A_{\alpha}}=\widehat{g_{\alpha}{A}_{\alpha}}+c_{g,\alpha}\cc\in \Omega^1(LU_{\alpha},\widetilde{L\g})$, 
 where $c_{g_{\alpha},\alpha}=\int_{S^1}\langle g_{\alpha} A_{\alpha},g_{\alpha}A_{\alpha}\rangle $.
 
Let $u_{\alpha,g_{\alpha}}=\widetilde{g_{\alpha}A_{\alpha}}-[\widetilde{Ad}_{Lg_{\alpha}}(\widetilde{A_{\alpha}})+\widetilde{Lg_{\alpha}^{-1}}^*\widetilde{\mu}]$,
where $\widetilde{Lg_{\alpha}}:LU_{\alpha}\to \widetilde{LG}$ is a lift of $Lg_{\alpha}$.
Although $u_{\alpha,g_{\alpha}}$ depends on the choice of lift $\widetilde{Lg_{\alpha}}$, 
the differential $du_{\alpha,g_{\alpha}}$ is independent of this choice. This follows from the same computation as in \cref{subsec:geometry of lifting gerbe and cir-equivariance}:
\be 
du_{\alpha,g_{\alpha}}=\int_{S^1}\CCS(A_{\alpha})-\int_{S^1}\CCS(gA_{\alpha})\in \Omega^2(LU_{\alpha}).
\ee
We now proceed by mimicking the computation in Proposition \ref{Global defined elliptic holonomy}:
\begin{align}
&Ehol_{gA}(\mathcal{H})(\gamma)|_{L^2U_{\alpha}}=f_{\alpha,g_{\alpha}}(\gamma)
=\Tr_{\mathcal{H}}hol_{\gamma}(\frac{d}{dy}-[\iota_{\partial_y}\widetilde{g_{\alpha}A_{\alpha}}]-[\tau \dd-\tau \iota_{\partial_x}\widetilde{g_{\alpha}A_{\alpha}}])   \\
&=
\Tr_{\mathcal{H}}hol_{\gamma}(\frac{d}{dy}-[\iota_{\partial_y}(u_{\alpha,g_{\alpha}}+\widetilde{Ad}_{Lg_{\alpha}}\widetilde{A}_{\alpha}+\widetilde{Lg_{\alpha}^{-1}}^*\widetilde{\mu})]-\widetilde{Ad}_{Lg_{\alpha}}(\tau \dd-\tau \iota_{\partial_x}\widetilde{A}_{\alpha})) \nonumber \\
&=\exp{(k\int_{S^1}\gamma^*u_{\alpha,g_{\alpha}})}\cdot \Tr_{\mathcal{H}}[hol_{\gamma}(\widetilde{Lg_{\alpha}}(\frac{d}{dy}-[\iota_{\partial_y}\widetilde{A_{\alpha}}]-[\tau \dd-\tau \iota_{\partial_x}\widetilde{A_{\alpha}}]) )\widetilde{Lg_{\alpha}^{-1}}]  \nonumber \\
&=\exp(k\int_{S^1}\gamma^*u_{\alpha,g_{\alpha}})\cdot f_{\alpha}(\gamma)=e^{ikW(\gamma^*g_{\alpha},\gamma^*A_{\alpha})}f_{\alpha}(\gamma)
=e^{ikW(\gamma^*g_{\alpha},\gamma^*A_{\alpha})}Ehol_A(\mathcal{H})(\gamma)|_{L^2U_{\alpha}}.\nonumber
\end{align}
\end{proof}
\subsubsection{Proof of Elliptic Atiyah-Witten formula(\cref{elliptic AW formula})}
We present the proof of odd spin case, other three even spin structures follow the same way.
\begin{theorem}
    There's a isometry between the Pfaffian line bundle and the transgression line bundle of the lifting gerbe on $\LLX$:
    \begin{align}
        \Phi: (\mathrm{PF},\nabla^{\mathrm{PF}},g^{\mathrm{PF}}) \xrightarrow{\cong} (\mathcal{L}(\mathcal{G}_P,\nabla^B),h).
    \end{align}
Under the isometry, on gauged poly-stable double loops,
\be q^{m}\cdot Ehol_{S^+-S^-}[\tau]=\Phi(\pf[\tau]). \ee  
$m$ is the modular anomaly of of $S^+-S^-$.
\begin{proof}
By the lemma above, the elliptic holonomy transforms under the $\Sigma G$-action in the same way as sections of the Chern-Simons line bundle $\mathcal{L}\to \mathcal{A}$. 
Let $(L^2P)^{ps}$ denote the gauged poly-stable double loops. The evaluation map
$\Psi_A^*:(L^2P)^{ps}\to \mathcal{A}^{ps}$ is $\Sigma G$-equivariant.
The elliptic holonomy defined on $(L^2P)^{ps}$ is the pullback of a $\Sigma G_{\CC}$-equivariant section of $\mathcal{L}\to \mathcal{A}^{ps}$.
Since poly-stable connections are $\Sigma G_{\CC}$-equivalent to constant connections in $\mathfrak{t}_{\CC}$, it suffices to establish the formula 
for constant connections.

For $\gamma \in L^2U_{\alpha}$ such that $i_{\partial_z^{\#}}(\gamma^*A_{\alpha})=K_{\alpha}\in \mathfrak{t}_{\CC}$ is a constant connection,
the elliptic holonomy is then given by
\begin{align}
&q^m\cdot Ehol_A(S^+-S^-)(\gamma)|_{L^2U_{\alpha}}
=\ q^m\cdot \Tr_{\mathcal{H}}\left[\exp\left(\tau \dd+K_{\alpha}+i_{\partial_z^{\#}}\gamma^*c_{\alpha}\cc\right)\right]\\
&\ = e^{\frac{\pi i\langle K_{\alpha},K_{\alpha}-\overline{K_{\alpha}}\rangle}{2\mathrm{Im}{(\tau)}}}q^m\cdot \Tr_{S^+-S^-}[\exp{(\tau \dd+K_{\alpha})}]. \nonumber
\end{align}
Identify $K_{\alpha}=(z_1,\cdots, z_n)\in \CC^n$; it is computed by
\be
F(z_1,\cdots, z_n)=e^{\frac{\pi i\sum_{i=1}^n (z_i,z_i-\bar{z_i})}{2\mathrm{Im}{(\tau)}}}\prod_{i=1}^n\frac{\theta_{11}(z_i,\tau)}{\eta(\tau)}.
\ee
By Theorem \ref{AW formula on MG by computation}, it can be identified with the odd Pfaffian under the chosen isometry. The prefactor here arises from the isometry between the Chern-Simons line bundle restricted to constant connections and the basic line bundle $\mathcal{L}_0$ defined on the Abelian variety $(\Sigma_{\tau}\otimes_{\Z}\bigwedge)$.
\end{proof}
\end{theorem}

\subsection{Chern-Simons Gauge Theory and Double Loop Space Geometry}\label{subsec: from QR=0 of CS and Conformal Blocks}
By the principle of \emph{quantization commutes with reduction} in Chern-Simons gauge theory as discussed in \cite{Axelrod1991GeometricQO}, we derive a $2$-transgression description of the elliptic holonomy. This description can be extended to higher genus Riemann surfaces and potentially to surfaces with marked points.

When the complex structure $\tau \in \mathbb{H}$ given, $\mathcal{A}$ is K\"ahler. The space \emph{quantization before reduction} is
$\mathcal{H}_k(G)|_{\tau}=H^0_{\Sigma G}(\mathcal{A},\mathcal{L}^k)$.
The space of \emph{quantization} after reduction is the genus one conformal block
$V_k(G)|_{\tau}=H^0(M_{G}[\tau],\widetilde{\mathcal{L}}^k)$.
In \cite[(1.56)]{Axelrod1991GeometricQO}, authors stated that the push-down 
\begin{align}\label{QR=0}
r:\mathcal{H}_k(G)|_{\tau}\to V_k(G)|_{\tau}    
\end{align}
is an isomorphism of vector spaces.
A rigorous treatment is provided in \cite{Preparation}. 

\begin{corollary}
The torus $T^2$ naturally acts on connections $\mathcal{A}=\Omega^1(\Sigma,\g)$ as the pull-back.
Let $\mathcal{H}_k(G)^{T^2}|_{\tau}$ be the $T^2$-invariant subspace of $\mathcal{H}_k(G)|_{\tau}$.
We have
    \begin{align}
\mathcal{H}_k(G)|_{\tau}=\mathcal{H}_k(G)^{T^2}|_{\tau}.
\end{align}
\begin{proof}
Constant connections are fixed points of $T^2$-action. 
Since the push-down $r$ is injective, any section in $\mathcal{H}_k(G)|_{\tau}$ must be $T^2$-invariant. 
\end{proof}
\end{corollary}
Recall that 
the Chern-Simons line bundle $(\mathcal{L}_{\CCS},\nabla_{\CCS}^A)\to \LLX$ is the pull-back $\Psi_A^*\mathcal{L}\to \mathcal{L}^2 P$ descended down to $\LLX$. Combined with the corollary above, it induces the map
\begin{align}
\Psi_A^*:\mathcal{H}_k(G)|_{\tau} \to \Omega^0(\LLX,\mathcal{L}_{\CCS}^k)^{T^2}.
\end{align}
We define the elliptic holonomy from $2$-transgression as the composition of the extension $r^{-1}$ and the pull-back by $\Psi_A$,
\begin{equation}
 \begin{tikzcd}
V_k(G)|_{\tau}\arrow[r,"r^{-1}"] & \mathcal{H}_k(G)|_{\tau}  \arrow[r,"\Psi_A^*"]&\Omega^0(\LLX,\mathcal{L}^k_{\CCS})^{T^2}. \\
\end{tikzcd} 
\end{equation}
The above construction generalizes naturally to higher genus. 
Let $\Sigma$ be a closed surface of genus $g>1$ equipped with a complex structure $J$, which provides the polarization for geometric quantization.
We denote by $\mathcal{A}=\Omega^{0,1}(\Sigma_J,\g_{\CC})$ the space of $G$-connections, and let $\mathcal{L}$ denote the Chern-Simons line bundle.

The K\"ahler quotient $M_{G}[J]=\mathcal{A}//\Sigma G$ can be identified with the coarse moduli obtained analytically via GIT-quotient.
In this setting, the space of quantization before reduction is 
\begin{align}
\mathcal{H}_{k}(G)|_J:=H^0_{\Sigma G}(\mathcal{A},\mathcal{L}^k),
\end{align}
and the conformal block is the space of \emph{quantization after reduction}
\begin{align}
V_{k}(G)|_J= H^0(M_{G}[J],\widetilde{\mathcal{L}}^k).
\end{align}
The \emph{quantization commutes with reduction} is still valid for higher genus.
Likewise, the composition of the extension and the pull-back
\begin{align}
V_{k}(G)|_{J} \to \mathcal{H}_{k}(G)|_J \to \Gamma(\Sigma X,\mathcal{L}_{\CCS}^k)
\end{align}
gives the generalization of the elliptic holonomy for higher genus Riemann surfaces.

\subsubsection{Relative Pfaffian}\label{relative pfaffian}
As $\tau$ varies over $\mathbb{H}$, the conformal blocks form the Verlinde bundle over the moduli space of elliptic curves $\mathcal{M}_{ell}=[\mathbb{H}//SL_2(\Z)]$:
\begin{align}
    V_k(G)\to \mathcal{M}_{ell}.
\end{align}
The Verlinde bundle is equipped with a projectively flat connection known as the Hitchin connection (see explicit formulas in \cite{Axelrod1991GeometricQO}).
Projective flatness ensures that the monodromy representation of $\mathbb{P}SL_2(\Z)$ acts on the conformal blocks.

Now we focus on $G=\mathrm{Spin}(2n)$ and level $k=1$.
We have already shown that the four Pfaffians form a basis of the conformal block $V_1(\mathrm{Spin}(2n))|_{\tau}$.
We can understand the modular invariance of $V_1(\mathrm{Spin}(2n))$ by studying the modular properties of the relative Pfaffians.

In \cite{Freed1987OnDL},
the upper half-plane $\mathbb{H}$ parameterize the complex structures of the torus, thereby defining a family of $\bar{\partial}$-operators. Let $\mathfrak{P} \to \mathbb{H}$ denote Pfaffian line bundle associated with the odd spin structure.

We investigate the modular properties of the Pfaffian line bundle corresponding to the odd spin structure.
The $SL_2(\Z)$ action on $\mathcal{A} \times \mathbb{H}$ is defined as follows: for $\gamma \in SL_2(\Z)$,
\begin{align}
\gamma(A, \tau) = (\gamma^{-1,*}A, \gamma(\tau)).
\end{align}
The Pfaffian line bundle $\mathrm{PF} \to \mathcal{A}$ can be extended to a \textbf{relative} Pfaffian line bundle as follows:
\be
\mathrm{PF}^{\mathrm{rel}}:=\mathrm{PF} \otimes \mathfrak{P}^{-2n} \to \mathcal{A} \times \mathbb{H}.
\ee
Both the Pfaffian line bundle $\mathrm{PF} \to \mathcal{A} \times \mathbb{H}$ and $\mathfrak{P} \to \mathbb{H}$ admit only $Mp_2(\Z)$-equivariant structures, where $Mp_2(\Z)$ denotes the metaplectic double cover of $SL_2(\Z)$. 
This is because the $SL_2(\Z)$ action does not lift to the spinors, only $Mp_2(\Z)$ does.
The relative structure cancels the $\Z_2$-anomaly, yielding the following result:
\begin{proposition}
The relative Pfaffian line bundle  $\mathrm{PF}^{\mathrm{rel}}\to \mathcal{A}\times \mathbb{H}$ is $\Sigma G\times T^2\times SL_2(\Z)$-equivariant. 
\end{proposition}
When pulled back to $\LLX\times \mathbb{H}$, it yields 
the $SL_2(\Z)$-equivariant relative Pfaffian line bundle on $\LLX\times \mathbb{H}$ as described in \cite[(5.3)]{Freed1987OnDL}.
\begin{remark}
Let $\Gamma<SL_2(\Z)$ be the subgroup preserving the given spin structure, and let $\widetilde{\Gamma}$ denote its double cover.
The corresponding Pfaffian line bundle is then $\widetilde{\Gamma}$-equivariant.
\end{remark}
On the moduli side,
when restricted to constant Cartan connections $\mathfrak{t}_{\CC}\times \mathbb{H}$, since now $T^2$-action is trivial,
the restricted relative Pfaffian line bundle admits $W_{\text{aff}}\times SL_2(\Z)$ action.

We study the modular property of the relative Pfaffian
under the trivialization.
Let $T \cdot (z, \tau) = (z, \tau+1)$ and $S \cdot (z, \tau) = \left(\frac{z}{\tau}, -\frac{1}{\tau}\right)$ be the generators of $SL_2(\Z)$.
The transformation laws for the normalized theta function are:
\begin{align}
    T \cdot \frac{\theta_{11}(z, \tau)}{\eta(\tau)}= e^{\frac{\pi i}{6}} \cdot \frac{\theta_{11}(z, \tau)}{\eta(\tau)},\quad
    S \cdot \frac{\theta_{11}(z, \tau)}{\eta(\tau)}= -i\, e^{i\pi z^2/\tau} \cdot \frac{\theta_{11}(z, \tau)}{\eta(\tau)}.
\end{align}
The relative Pfaffian is given by the Pfaffian divided by the Dedekind $\eta$-function $\eta(\tau)^{2n}$, after the trivialization, it is
\begin{align}
    \mathrm{res}(\pf^{\mathrm{rel}})(z_1, \ldots, z_n, \tau)\sim \prod_{k=1}^n \frac{\theta_{11}(z_k, \tau)}{\eta(\tau)^3}.
\end{align}
From the transformation law of $\eta(\tau)$, we have:
\begin{align}
    T \cdot \frac{\theta_{11}(z, \tau)}{\eta(\tau)^3}= \frac{\theta_{11}(z, \tau)}{\eta(\tau)^3},\quad
    S \cdot \frac{\theta_{11}(z, \tau)}{\eta(\tau)^3}= \frac{e^{i\pi z^2/\tau}}{\tau} \cdot \frac{\theta_{11}(z, \tau)}{\eta(\tau)^3}.
\end{align}
 $\frac{\theta_{11}(z, \tau)}{\eta(\tau)}$ is modified by dividing by $\eta(\tau)^2$ and then the factor $e^{\frac{\pi i}{6}}$ under $T$-transformation is killed.
This explicitly shows how the relative structure eliminates the global anomaly given by the character
\be 
\mu:Mp_2(\Z)\to \Z/24\Z.
\ee

Under the $S$-transformation, it acquires a projective factor $e^{i\pi z^2/\tau}/\tau$ that cannot be eliminated. To render $\frac{\theta_{11}(z,\tau)}{\eta(\tau)^3}$ modular, we must incorporate the exponential factor derived from the second Eisenstein series.

Let $E_k(\tau)=\sum_{m,n\in \Z^2-(0,0)}\frac{1}{(m\tau+n)^k}$ denote the Eisenstein series of weight $k$.
Here we use the normalized Eisenstein series \( G_k(\tau) = \frac{(k-1)!}{2(2\pi i)^k} \cdot E_k(\tau) \).
For even \( k \), \( G_k(\tau) \) is a modular form of weight \( k \), whereas \( G_2(\tau) \) is only quasi-modular, with the following transformation law:
\begin{align}
 G_2(\tau+1)=G_2(\tau),\quad  G_2\left(-\frac{1}{\tau}\right) = \tau^2 G_2(\tau)+\frac{\tau}{4\pi i}.
\end{align}
From the transformation law of $G_2(\tau)$, directly we have 
\begin{align}
e^{-4\pi^2G_2(\tau)z^2}\cdot \frac{z}{\theta_{11}(z,\tau)/\eta(\tau)^3}
\end{align}
is $SL_2(\Z)$-invariant.

By \cite{Zagier1988NoteOT}, $\frac{z}{\theta_{11}(z,\tau)/\eta(\tau)^3}$ represents the characteristic series of the Witten genus. 
From the formula above, it's direct to see that the Witten genus is modular if $\frac{1}{2}p_1=0$.

The projective factor $e^{4\pi^2G_2(\tau)z^2}$ is intricately linked to the Hitchin connection on the Verlinde bundle. We aim to find the corresponding differential operator associated with the characteristic series $\frac{z}{\theta_{11}(z,\tau)/\eta(\tau)^3}$.

It is known that
Jacobi theta functions $\theta_{ij}(z,\tau)$ satisfy the classical heat equation, where $\mathcal{D}=\frac{\partial}{\partial \tau}-\frac{1}{4\pi i}\frac{\partial^2}{\partial z^2}$, such that
\begin{align}
\mathcal{D}\theta_{ij}(z,\tau)=0.
\end{align}

We define the \textbf{modified heat operator} 
\begin{align}
   \widehat{\mathcal{D}}=\frac{\partial}{\partial \tau}-3G_2(\tau)-\frac{1}{4\pi i}\frac{\partial^2}{\partial z^2}
\end{align}
$\frac{\theta_{ij}(z,\tau)}{\eta(\tau)^3}$ are solutions of the modified heat operator $\widehat{\mathcal{D}}$ since
\begin{align}
\widehat{\mathcal{D}}\frac{\theta_{ij}(z,\tau)}{\eta(\tau)^3}&=\frac{1}{\eta(\tau)^3}\left(\frac{\partial}{\partial \tau}-3G_2(\tau)-\frac{1}{4\pi i}\frac{\partial^2}{\partial z^2}\right)\theta_{ij}(z,\tau)+\theta_{ij}(z,\tau)\left(\frac{\partial}{\partial \tau}\right)\frac{1}{\eta(\tau)^3}\\
&=\frac{\theta_{ij}(z,\tau)}{\eta(\tau)^3}\left(-3G_2(\tau)-3\frac{\eta'(\tau)}{\eta(\tau)}\right)=0\nonumber.
\end{align}
The last equality follows from the fact that the logarithmic derivative of $\eta(\tau)$ 
\begin{align}
(\log \eta(\tau))'=-G_2(\tau).
\end{align}

We then consider the conjugation
\begin{align}
\widetilde{\mathcal{D}}=e^{4\pi^2G_2(\tau)z^2}\circ \widehat{\mathcal{D}}\circ e^{-4\pi^2G_2(\tau)z^2}=\mathcal{D}+C_1 z\partial_z+C_2 z^2+C_3,
\end{align}
where $C_1,C_2,C_3$ are elements in the ring $\CC[G_2,G_4]$.
We have 
\begin{align}
\widetilde{\mathcal{D}}\left(e^{4\pi^2G_2(\tau)z^2}\cdot \frac{\theta_{11}(z,\tau)}{\eta(\tau)^3}\right)=0.
\end{align}
We expect to investigate the modular properties of $\widetilde{\mathcal{D}}$ since they are closely related to the deprojectivization of the Hitchin connection.
\begin{remark}
To achieve modularity, one can instead replace $G_2(\tau)$ with its non-holomorphic completion $\hat{G}_2(\tau)=G_2(\tau)-\frac{1}{8\pi \Im(\tau)}$.
Then
\begin{align}
\hat{G}_2\left(-\frac{1}{\tau}\right)=\tau^2G_2(\tau)+\frac{\tau}{4\pi i}-\frac{1}{8\pi \mathrm{Im}(-1/\tau)}
=\tau^2\hat{G}_2(\tau).
\end{align}
The completion $\hat{G}_2(\tau)$ arises geometrically from the Bismut-Freed connection on the Pfaffian line bundle $\mathfrak{P}^{1/2}\to \mathbb{H}$. 
Since the kernel of the Dirac operator $D_{\tau}$ is always one-dimensional, it fits together to form a line bundle $L\to \mathbb{H}$,
 where $w_{\tau}=\sqrt{dz}$ is a holomorphic section of $L$.  The square of the norm of $w_{\tau}$ is $\mathrm{Im}(\tau)$.
It is expected that the Hitchin connection induced on the determinant bundle $\det(V_k(G))\to \mathcal{M}_{\mathrm{ell}}$ is compatible with the Bismut-Freed connection. 
\end{remark}

\section{Elliptic Bismut-Chern Characters on \texorpdfstring{$L^2X$}{LLX}}\label{sec: Elliptic Bismut Chern Character}
Based on Section~\ref{sec:twisted bismut}, we formulate the elliptic Bismut-Chern character $EBCh$ on $\LLX$ as the $(S^1\times \disc)$-equivariant closed extension of the elliptic holonomy.
It fits into the following commutative diagram:
\begin{equation}
\begin{tikzcd}
    & h_{S^1\times \disc}^{2*}(\LLX,(\mathcal{L}_{\CCS},\nabla_{\CCS}^A,\overline{H})^k)|_{\tau}  \arrow[ld, "i_{10}^*"'] \arrow[dr,"i_{01}^*"] & \\
h_{S^1}^{2*}(\LX)[[q]] \arrow[dr,"i^*",swap]& \per^k(LG)\arrow[dashed,l,"BCh"]\arrow[dashed,d,"Ch"]\arrow[dashed,r,"ECh"]\arrow[dashed,u,"EBCh"]&  h_{\disc}^{2*}(\LX,kH)|_{\tau} \arrow[ld, "i^*"'swap] \\
    & h^{2*}(X)[[q]]   &
\end{tikzcd}
\end{equation}
\noindent\textbf{Organization of this section.}
In Subsection \ref{subsec: completed Periodic Exotic Twisted equivariant Cohomology of LLX}, we construct the (completed-periodic) exotic twisted $(S^1\times \disc)$-equivariant cohomology of $\LLX$.
In Subsection \ref{subsec: Elliptic Bismut Chern Character on LLX}, we introduce the elliptic Bismut-Chern character as the $(S^1\times \disc)$-equivariant twisted Bismut-Chern character.
In Subsection \ref{EBCH further discussions}, we discuss further directions and open questions.
\subsection{Exotic Twisted \texorpdfstring{$(S^1\times \disc)$}{Torus}-equivariant Cohomology of \texorpdfstring{$\LLX$}{LLX}}\label{subsec: completed Periodic Exotic Twisted equivariant Cohomology of LLX}
Following Subsection \ref{subsec: completed Periodic Exotic Twisted equivariant Cohomology of LM}, 
we set up the completed-periodic exotic twisted $(S^1\times \disc)$-equivariant cohomology of $L(LX)$.

As before, we distinguish between the two circles.
Let the torus $T^2 = S^1 \times \cir$, where
the first circle $S^1$ corresponds to the $y$-direction, and the second circle $\cir$ corresponds to the $x$-direction.
$\partial_x$ is the vector field generating rotations in the $x$-direction, which is the evaluation of the vector field $K$;
$\partial_y$ is the vector field on $L(\LX)$ generating rotations in the $y$-direction.
Let $\tau = \tau_1 + i\tau_2 \in \mathbb{H}$,
and $z = x + \tau y$ denote the complex coordinate on the elliptic curve $\Sigma_{\tau} = \CC / (\Z \oplus \tau \Z)$.
The standard flat Kähler form with unit volume is $\frac{-1}{2\tau_2} dz \wedge d\bar{z}$.
The complex vector field dual to $dz$ is
$\partial_z^{\#} = \tau \partial_x - \partial_y$.
Replacing $\partial_x$ with $\tau \partial_x$ effectively deforms the $x$-direction circle $\cir$ to the punctured disk $\disc$.
\begin{definition}
We introduce non-twisted version first. The $(S^1\times \disc)$-equivariant double loop space cohomology
$h_{S^1\times \disc}^*(\LLX)|_{\tau}$ is given by the
$\Z_2$-graded complex
$(\Omega^*(\LLX)^{T^2},d-i_{\partial_z^{\#}})$
where $\Omega^*(\LLX)^{T^2}$ is the $T^2$-invariant part of $\Omega^*(\LLX)$, and on it $\mathcal{D}_{\tau}=d-i_{\partial_z^{\#}}$ is square zero.
The cohomology is
called the completed periodic $(S^1\times \disc)$-equivariant cohomology of $\LLX$.
\end{definition}

For the exotic twisted version,
we consider the $\Z_2$-graded complex $\Omega^*(\LLX,\mathcal{L}_{\CCS})$ and
set $\widehat{\Phi}\in \Omega^4(\LLX)$ as the evaluation on $\LLX$.
\begin{align}
\overline{H}=\int_{S^1}(\iota_{\partial_x}\widehat{\Phi})dy\in \Omega^3(\LLX)
\end{align}
is the average of the curving $H\in \Omega^3(\LX)$ along the $y$-direction. Consider
\begin{align}
\mathcal{D}_{\tau}=\nabla_{\CCS}^A-\iota_{\partial_{z}^{\#}}-\overline{H},
\end{align}
which is an odd operator acting on $\Omega^*(\LLX,\mathcal{L}_{\CCS})$.
Locally, recall that the local $B$-field is given by the transgression of the Chern-Simons $3$-form $K_{\alpha}=\int_{S^1}\CCS(A_{\alpha})\in \Omega^2(LU_{\alpha})$.
$\overline{K_{\alpha}}$ is the average of $K_{\alpha}$ along the $y$-direction.
\begin{align}
\mathcal{D}_{\tau}|_{\mathcal{L}^2U_{\alpha}}&=d-\iota_{\partial_{z}^{\#}}\overline{K_{\alpha}}-\iota_{\partial_{z}^{\#}}-d\overline{K_{\alpha}}=e^{\overline{K_{\alpha}}}(d-\iota_{\partial_{z}^{\#}})e^{-\overline{K_{\alpha}}}.
\end{align}
Then $\mathcal{D}_{\tau}^2=-L_{\partial_{z}^{\#}}$, which is zero restricted on thec $T^2$-invariant part
$\Omega^*(\LLX,\mathcal{L}_{\CCS})^{T^2}$.
\begin{definition}
$(\Omega^*(\LLX,\mathcal{L}_{\CCS})^{T^2},\mathcal{D}_{\tau})$ is a $\Z_2$-graded complex,
we call the cohomology of this complex the exotic twisted $(S^1\times \disc)$-equivariant cohomology of $\LLX$ with respect to the Chern-Simons line bundle $(\mathcal{L}_{\CCS},\nabla_{\CCS}^A)$.
At $\tau \in \mathbb{H}$, it is
denoted by
$
h^*_{S^1\times \disc}(\LLX,(\mathcal{L}_{\CCS},\nabla_{\CCS}^A,\overline{H}))|_{\tau}.
$
\end{definition}
\subsubsection{String-$G$ Structure and Trivialization}\label{subsec: string-G structure and trivialization}
Recall from Subsection \ref{subsec: Trivialized Elliptic Loop Chern Character by String G-structure}
that the geometric string-$G$ structure provides a geometric trivialization of the Chern-Simons $2$-gerbe, which gives:
\begin{itemize}
    \item  $C\in \Omega^3(X)$ satisfying
$
   dC=\Phi.
$
    \item  $B_{\alpha}\in \Omega^2(U_{\alpha})$ satisfying
$
  \CCS(A_{\alpha})-C=dB_{\alpha}.   
$
\end{itemize}
We now use the string-$G$ structure to construct a flat connection $\nabla^C$ and a covariant flat section $s^C$ of unit norm, which provides a trivialization of the Chern--Simons line bundle.

\begin{proposition}
    For each $\alpha$, define $f_{\alpha}=i\int_{\Sigma}B_{\alpha}\in \Omega^0(\Sigma U_{\alpha})$.
    The collection $\{df_{\alpha}\}$ defines a flat connection $\nabla^C$ on $\mathcal{L}_{\CCS}$.
\end{proposition}
\begin{proof}
Since $\Theta_{\alpha}-\Theta_{\beta}=\sqrt{-1}(\int_{\Sigma}\CCS(A_{\alpha})-\int_{\Sigma}\CCS(A_{\beta}))=h_{\alpha\beta}dh_{\alpha\beta}^{-1}$ and $H_{\alpha}=H_{\beta}$, we have
\begin{align}
df_{\alpha}-df_{\beta}=\sqrt{-1}(\int_{\Sigma}dB_{\alpha}-\int_{\Sigma}dB_{\beta})=\Theta_{\alpha}-\Theta_{\beta}=h_{\alpha\beta}dh_{\alpha\beta}^{-1}.
\end{align}
\end{proof}

\begin{proposition}
    Define $s_{\alpha}(x)=e^{f_{\alpha}(x)}=e^{i\int_{\Sigma}x^*B_{\alpha}}$ for $x\in \mathcal{L}^2 U_{\alpha}$.
    The collection $\{s_{\alpha}\}$ glues together to define a globally flat covariant section $s^C$ of unit norm on the trivialized Chern--Simons line bundle.
\end{proposition}
\begin{proof}
Locally, $s_{\alpha}$ is covariant flat since
$
(d-df_{\alpha})s_{\alpha}=0.
$
To verify that $\{s_{\alpha}\}$ defines a global section, observe that
\begin{align}
s_{\alpha}/s_{\beta}(x)=e^{f_{\alpha}-f_{\beta}(x)}=e^{\sqrt{-1}\int_{\Sigma}x^*(B_{\alpha}-B_{\beta})}.
\end{align}
Choose an arbitrary three-dimensional manifold $B$ with boundary $\partial B=\Sigma$, and extend $x$ to a map $\widetilde{x}:B\to U_{\alpha}\cap U_{\beta}$ such that $\widetilde{x}|_{\partial B}=x$.
By Stokes formula,
\begin{align}
   \int_{\Sigma}x^*(B_{\alpha}-B_{\beta})=\int_{B}\widetilde{x}^*(dB_{\alpha}-dB_{\beta})=\int_B \widetilde{x}^*(\CCS(A_{\alpha})-\CCS(A_{\beta})). 
\end{align} 
Consequently,
$
s_{\alpha}/s_{\beta}(x)=e^{iW(x\circ g_{\alpha\beta},x^*A_{\beta})}=h_{\alpha\beta}(x).
$
\end{proof}

When a geometric string $G$-structure exists, the Chern-Simons line bundle can be trivialized, leading to
\begin{align}\label{trivialized double loop cohomology}
    h^*_{S^1\times \disc}(\LLX,(\mathcal{L}_{\CCS},\nabla_{\CCS}^A,\overline{H}))|_{\tau}\cong h_{S^1\times \disc}^*(\LLX)|_{\tau}.
\end{align}

Similarly, we can define the cohomology with respect to multiple products of Chern-Simons line bundles associated with principal $G_k$-bundles with connections $(P_k,A_k)\to X$.
When anomaly cancellation occurs, \ie there exists $C\in \Omega^3(X)$ such that
$dC=\sum_k \Phi_{K}$, where $\Phi_K=\langle R_k,R_k\rangle$,
the cohomology becomes isomorphic to the non-twisted version $h_{S^1\times \cir}^*(\LLX)|_{\tau}$.

\subsubsection{Restrictions}
We set $M=\LX$, so that $M^{\cir}=X$.
Let $i_{00}:X\to \LLX$ denote the inclusion of double loops constant in both directions.
Let $i_{01}=i_{\LX}:\LX \to \LLX$ denote the inclusion of double loops constant in the $y$-direction. 
Let $i_{10}=Li:\LX \to \LLX$ denote the inclusion of double loops constant in the $x$-direction.
The following diagram commutes:
\begin{equation}
\begin{tikzcd}[rotate=45]
    & \LLX  & \\
    \LX \arrow[ur,"i_{10}"] & & \LX\arrow[ul,"i_{01}",swap]\\
    & X \arrow[ur,"i"]\arrow[ul,"i",swap]\arrow[uu,"i_{00}"] &
\end{tikzcd}
\end{equation}
Under restriction via $i_{10}$, the cohomology reduces to the $S^1$-equivariant cohomology of $\LX$, since the lifting gerbe is trivial on the constant loop space $X$.
Under restriction via $i_{01}$, the cohomology reduces to the twisted $\disc$-equivariant cohomology of $\LX$, governed by
the operator $\mathcal{D}_{\tau}=d-\tau \iota_{\partial_x}-kH$.
The following diagram commutes:
\begin{equation}\label{eq:double-loop-cohomology-diagram}
\begin{tikzcd}
    & h_{S^1\times \disc}^{*}(\LLX,(\mathcal{L}_{\CCS},\nabla_{\CCS}^A,\overline{H})^k)|_{\tau}  \arrow[ld, "i_{10}^*"'] \arrow[dr,"i_{01}^*"] \arrow[dd,"i_{00}^*"] & \\
h_{S^1}^{*}(\LX) \arrow[dr,"i^*",swap]& &  h_{\disc}^{*}(\LX,kH)|_{\tau} \arrow[ld, "i^*"'swap] \\
    & h^{*}(X)   &
\end{tikzcd}
\end{equation}
\subsection{Elliptic Bismut-Chern Characters}\label{subsec: Elliptic Bismut Chern Character on LLX}
Building on the formulation of the equivariant twisted Bismut-Chern character on $\LX$,
we set $M=\LX$, and then $LM=\LLX$. Consider
the lifting gerbe $(\mathcal{G}_P,\nabla^B)\to \LX$ which is $\cir$-equivariant and whose moment map vanishes.

The elliptic Bismut-Chern Character is defined as the $(S^1\times \disc)$-equivariant twisted Bismut-Chern character associated with the gerbe module given by $\mathcal{H}\in \per^k(LG)$.
It serves as the double loop space refinement of the elliptic Chern character: 
\begin{equation}
     \begin{tikzcd}
    \per^k(LG) \arrow[rr,"EBCh"] \arrow[dr,"ECh"] &  &  h_{S^1\times \disc}^{2*}(\LLX,(\mathcal{L_{\CCS}},\nabla_{\CCS}^A,\overline{H})^k|_{\tau}) \arrow[dl,"i_{10}^*"] \\
        &  h_{\disc}^{2*}(\LX,kH)|_{\tau} &
    \end{tikzcd} 
\end{equation}
From  \cref{eq-twisted bismut D operator},
over $L^2U_{\alpha}$, we consider $\mathcal{D}_{\alpha}$, a first-order differential operator of $y$ with values in $\widetilde{L\g_{\CC}}'\otimes\Omega^*(L^2U_{\alpha})$,
\begin{align}
 \mathcal{D}_{\alpha}= (\frac{d}{dy}+\iota_{\partial_y}\widetilde{A}_{\alpha})-(\widetilde{R}_{\alpha}[\tau]+K_{\alpha}).
\end{align}
$\widetilde{A}_{\alpha},\widetilde{R}_{\alpha}[\tau],K_{\alpha}$ are all defined on $LU_{\alpha}$. Abuse of notation, we denote their evaluation on $L^2U_{\alpha}$ by the same symbols.
Denote $\mathcal{D}_{\alpha}=\frac{d}{dy}-(\tau \dd+Y_{\alpha})$,
where 
\begin{align}
Y_{\alpha}=\iota_{\partial_{z}^{\#}}\widetilde{A}_{\alpha}+\widetilde{R}_{\alpha}+K_{\alpha} \in \Omega^*(L^2U_{\alpha},\widetilde{L\g_{\CC}}').
\end{align}
Let $Y_{\alpha}(t)$ be the pull-back of $Y_{\alpha}$ by the evaluation on $y=t$.
Under the representation $\mathcal{H}\in \per^k(LG)$, the elliptic Bismut-Chern Character is locally given by the trace of holonomy
which can be expressed as the iterated integral,
\begin{align}
EBCh_A(\mathcal{H})|_{L^2U_{\alpha}}=\Tr_{\mathcal{H}}[hol(\mathcal{D}_{\alpha})]=\Tr_{\mathcal{H}}[\sum_{m=0}^{\infty}\int_{\Delta_m}Y_{\alpha}(t_1)\cdots Y_{\alpha}(t_m)dt_1\cdots dt_m].
\end{align}
First, the degree-zero component corresponds to the elliptic holonomy. 
While we cannot address the convergence for general smooth double loops, we can ensure convergence at gauged-poly-stable double loops.

By the same argument in \cref{Global defined elliptic holonomy},
it defines a global form on $\LLX$ with values in the Chern-Simons line bundle.
Follows by \cref{D_H closed},
 $EBCh_A(\mathcal{H})$ is $\mathcal{D}_{\tau}$-closed, it is a cocyle of $h_{S^1\times \disc}^{2*}(\LLX,(\mathcal{L_{\CCS}},\nabla_{\CCS}^A),\overline{H})|_{\tau}$.

We show its rectriction to one direction is the elliptic Chern character.
 \begin{proposition}
$i_{10}=Li:\LX \to L(\LX)$ is the inclusion of double loops constant on $y$, 
\begin{align}
i_{10}^*(EBCh_A(\mathcal{H}))=ECh_A(\mathcal{H}).
\end{align}
\end{proposition}
\begin{proof}
Since constant on $y$, $\iota_{\partial_y}\widetilde{A}_{\alpha}=0$.
The holonomy $hol(\mathcal{D}_{\alpha})$ reduces to the exponential
$\exp(\widetilde{R}_{\alpha}[\tau]+K_{\alpha})$. Then
\begin{align}
    \Tr_{\mathcal{H}}[hol(\mathcal{D}_{\alpha})]=\Tr_{\mathcal{H}}[\exp(\widetilde{R}_{\alpha}[\tau]+K_{\alpha})]=Ech_A(\mathcal{H})|_{LU_{\alpha}}.
\end{align}
\end{proof}
We show that its restriction to another direction is the $q$-graded Bismut-Chern character.
\begin{proposition}
Let $i_{01}=i_{\LX}:\LX \to L(\LX)$ be the inclusion of loops constant in the $x$-direction. Then,
\begin{align}
i_{01}^*(EBCh_A(\mathcal{H}))=BCh_A(\mathcal{H})[q].
\end{align}
\begin{proof}
Since the loops are constant in the $x$-direction, both the central term $c_{\alpha}$ and the $B$-field $K_{\alpha}$ vanish. Consequently, $\mathcal{D}_{\alpha}$ simplifies to
$
    \mathcal{D}_{\alpha}=\frac{d}{dy}-\iota_{\partial_y}\hat{A}_{\alpha}-(\tau \dd+\hat{R}_{\alpha}).
$
Evaluating at fixed time $y=t$ gives a $\g_{\CC}$-valued expression. From the iterated integral, the trace decomposes according to the energy decomposition of $\mathcal{H}$:
\begin{align}
    &\Tr_{\mathcal{H}}[hol(\mathcal{D}_{\alpha})]=\sum_{n=0}^{\infty} \Tr_{\mathcal{H}_n}[hol(\frac{d}{dy}-\iota_{\partial_y}\hat{A}_{\alpha}-(\tau \dd+\hat{R}_{\alpha}))] \\
    &\ =\sum_{n=0}^{\infty}q^n \Tr_{\mathcal{H}_n}[hol(\frac{d}{dy}-(\iota_{\partial_y}\hat{A}_{\alpha}+\hat{R}_{\alpha}))]=\sum_{n=0}^{\infty}q^n BCh_A((P\times_G \mathcal{H}_n)) \nonumber.
\end{align}
This results in the $q$-graded Bismut-Chern character.
\end{proof}
\end{proposition}
\subsection{Further Discussions}\label{EBCH further discussions}
We propose the analytical elliptic Bismut-Chern character on $\LLX$ as super Pfaffians. 
From our construction of the elliptic Bismut-Chern character, it is evident that the two loops play distinct roles in the $(1+1)$-transgression. One loop is associated with loop groups and loop bundles, while the other loop is used for transgression. Consequently, the modular property cannot be directly observed, as the two loops are treated asymmetrically.
In contrast, the $2$-transgression approach appears more natural for understanding modularity. 
If we consider the Pfaffian line bundle over the double loop space $\LLX$ discussed earlier, there should exist the analytical elliptic Bismut-Chern character 
$\pf(\frac{\partial}{\partial_{\bar{z}}}-(i_{\partial_z^{\#}}\hat{A}+\hat{R}))$, which is a form on $\LLX$
    with values in the Pfaffian line bundle.
Note that $\nabla_{\bar{z}}^A=\frac{\partial}{\partial_{\bar{z}}}-i_{\partial_z^{\#}}\hat{A}$ is given by the family of Dirac operators coupled with the connection $\gamma^*A$ at $\gamma \in \LLX$.
The degree-zero component corresponds precisely to the canonical section Pfaffian.

On the other hand, when restricted to the constant loop space $X$,
 it simplifies to the regularized Pfaffian $\pf(\frac{\partial}{\partial_{\bar{z}}}-\hat{R})$, which can be rigorously defined as demonstrated in \cite[(5.2)]{BerwickEvans2019SupersymmetricLM}. 
The relative regularized pfaffian $\pf(\frac{\partial}{\partial_{\bar{z}}}-\hat{R})/\pf(\frac{\partial}{\partial_{\bar{z}}})$ computes the Witten genus when $X$ is further string.

However, there are significant challenges in providing a precise definition of $\pf(\nabla_{\bar{z}}^A-\hat{R})$ over $\LLX$.
We won't delve into the technical details here but give a brief outline and highlight the main difficulties.

First, we define the higher twisted Laplacian $\hat{\Delta}_A=(\nabla_{\bar{z}}^A-\hat{R})^*(\nabla_{\bar{z}}^A-\hat{R})=\Delta_A+\mathcal{F}$, 
where $\Delta_A$ is the ordinary Laplacian and $\mathcal{F}$ is the higher degree term given by the curvature.
Formally we can write down $e^{-t\hat{\Delta}_A}$ by the Volterra series
\begin{align}
    e^{-t\hat{\Delta}_A}=e^{-t\Delta_A}+\sum_{k=1}^{\infty}(-t)^k \int_{\Delta_k}e^{-s_0t\Delta_A}\mathcal{F}e^{-s_1t\Delta_A}\cdots \mathcal{F}e^{-s_kt\Delta_A}ds_1\cdots ds_k.
\end{align}
Here the summation extends to infinity because $\LLX$ is now infinite-dimensional.
If the heat kernel of $\hat{\Delta}_A$ can be constructed rigorously,
we can define the trace $\tr[e^{-t\hat{\Delta}_A}]$ by integrating the heat kernel on the diagonal. Now 
$\tr[e^{-t\hat{\Delta}_A}]$ is a form defined on $\LLX$ which has exponential decay as $t\to \infty$ and the asymptotic expansion as $t\to 0$ where the coefficients are forms.
Then we can define the zeta-regularized determinant by Mellin transformation 
\begin{align}
\zeta_{\hat{\Delta}_A}(s)=\frac{1}{\Gamma(s)}\int_0^{\infty}t^{s-1}\tr[e^{-t\hat{\Delta}_A}]dt,
\end{align}
and then the form version zeta-regularized determinant is given by
\begin{align}
\dt_{\zeta}(\hat{\Delta}_A)=\exp{(-\zeta_{\hat{\Delta}_A}'(0))}.
\end{align}
We are uncertain whether this can be well-defined as a form on $\LLX$. 
The only certainty is that the degree-zero component corresponds to the $\zeta$-regularized determinant of the Laplacian of Dirac operators, which defines the Quillen metric for determinant and Pfaffian line bundles.
Finally, we anticipate that the analytical elliptic Bismut-Chern character $\pf(\nabla_{\bar{z}}^A-\hat{R})$ satisfies
\begin{align}
    \|\pf(\nabla_{\bar{z}}^A-\hat{R})\|_{\PF}^2=\dt_{\zeta}(\hat{\Delta}_A).
\end{align}
The analytical elliptic Bismut-Chern character $\pf(\nabla_{\bar{z}}^A-\hat{R})$ refers to the super Pfaffian.
 If a string structure exists on $X$ and the Pfaffian line bundle is trivialized, then
$\pf(\nabla_{\bar{z}}^A-\hat{R})$ can be viewed as a function on the super double loop space $\mathcal{S}(\LLX)$.
It would be more practical to consider a finite-dimensional space $B$ parameterizing double loops, i.e., there is a map $\rho:B\times \Sigma \to X$. Furthermore, we require $B$ to have a $T^2$-action and $\rho$ to be $T^2$-invariant under the $T^2$-action on $B\times \Sigma$.

Even if the analytic difficulties above can be resolved, another issue remains: how to formulate the $(S^1\times \disc)$-equivariant cohomology twisted by the relative Pfaffian line bundle. 
We need to find an analytical construction of $3$-form \(\xi\) on \(\LLX\) such that \(i_{\partial_z^{\#}}\xi\) equals the curvature of the Pfaffian line bundle. We then define the twisted equivariant differential
\begin{align}\label{Q}
    Q_{\tau}=\nabla^{\mathrm{rel}}-i_{\partial_z^{\#}}-\xi,
\end{align}
acting on \(\Omega^*(\LLX\times \mathbb{H},\mathrm{PF}^{\mathrm{rel}})^{T^2}\), where $\mathrm{PF}^{\mathrm{rel}}$ is the relative Pfaffian line bundle and $\nabla^{\mathrm{rel}}$ is the relative Bismut-Freed connection.
This situation differs from the (1+1)-transgression case, where the 3-form \(\xi=\overline{H}\) is obtained by averaging the curving \(H\) along the \(y\)-direction. It would be interesting to develop such an equivariant twisted cohomology for \(\LLX\) with respect to the $2$-transgression line bundle.

In \cite{BerwickEvans2019SupersymmetricLM}, the double loop space model of complex analytic elliptic cohomology is presented as a $SL_2(\Z)$-equivariant sheaf of commutative differential graded algebras over $\mathbb{H}$.
For a fixed $\tau \in \mathbb{H}$, when a string structure exists, the cohomology given in \cref{trivialized double loop cohomology} is compatible with the above $Q$-cohomology at the fiber $\tau$.

A natural question arises: can we extend the results of \cite{BerwickEvans2019SupersymmetricLM} to the scenario 
where $X$ is not string and consider twisting by the relative Pfaffian line bundle over $\LLX\times \mathbb{H}$ as in \cite{Freed1986DeterminantsTA}? 
This question illuminates the connection between the Green-Schwarz string anomaly and the conformal anomaly arising from the Hitchin connection being only projectively flat.
The relative Bismut-Freed connection is related with the logarithmic derivative of the Dedekind eta function, which is the second normalized Eisenstein series $G_2(\tau)$.
This structure plays a dual role: it provides essential geometric data for constructing the double loop space cohomology twisted by the relative Pfaffian line bundle, while simultaneously governing the deprojectivization of the Hitchin connection.
We anticipate that a unified framework integrating Verlinde bundles, Hitchin connections, and the double loop space model of complex-analytic elliptic cohomology will emerge from further investigation of these relationships.

\noindent\rule{\textwidth}{0.4pt}
\bibliographystyle{amsalpha}
\bibliography{references-2}

@article{Slodowy1985ACA,
  title={{A Character Approach to Looijenga's Invariant Theory for Generalized Root Systems}},
  author={P. Slodowy},
  journal={Compositio Mathematica},
  year={1985},
  volume={55},
  pages={3-32},
  url={https://api.semanticscholar.org/CorpusID:54534032}
}

@article{Atiyah1983TheYE,
  title={{The Yang-Mills Equations over Riemann Surfaces}},
  author={M. F. Atiyah and R. Bott},
  journal={Philosophical Transactions of the Royal Society of London. Series A, Mathematical and Physical Sciences},
  year={1983},
  volume={308},
  pages={523 - 615},
  url={https://api.semanticscholar.org/CorpusID:13601126}
}

@article{Witten1987,
author = {E. Witten},
title = {{Elliptic Genera and Quantum Field Theory}},
volume = {109},
journal = {Communications in Mathematical Physics},
number = {4},
publisher = {Springer},
pages = {525 -- 536},
year = {1987},
}

@InProceedings{Zagier1988NoteOT,
author="Zagier, Don",
editor="Landweber, Peter S.",
title="Note on the Landweber-Stong elliptic genus",
booktitle="Elliptic Curves and Modular Forms in Algebraic Topology",
year="1988",
publisher="Springer Berlin Heidelberg",
address="Berlin, Heidelberg",
pages="216--224",
isbn="978-3-540-39300-9"
}

@article{Distler2007HeteroticCW,
  title={{Heterotic Compactifications with Principal Bundles for General Groups and General Levels}},
  author={J. Distler and E. Sharpe},
  journal={Advances in Theoretical and Mathematical Physics},
  year={2007},
  volume={14},
  pages={335-397},
  url={https://api.semanticscholar.org/CorpusID:17485101}
}

@article{Looijenga1976RootSA,
  title={{Root Systems and Elliptic Curves}},
  author={E. Looijenga},
  journal={Inventiones Mathematicae},
  year={1976},
  volume={38},
  pages={17-32},
  url={https://api.semanticscholar.org/CorpusID:121219851}
}

@book{Kac1990InfiniteDL,
  title={{Infinite Dimensional Lie Algebras}},
  author={V. G. Kac},
  year={1990},
  publisher={Cambridge University Press},
  url={https://api.semanticscholar.org/CorpusID:118284658}
}

@article{Etingof1994SphericalFO,
  title={{Spherical Functions on Affine Lie Groups}},
  author={P. Etingof and I. B. Frenkel and A. A. Kirillov},
  journal={Duke Mathematical Journal},
  year={1994},
  volume={80},
  pages={59-90},
  url={https://api.semanticscholar.org/CorpusID:16800000}
}

@Inbook{Lurie2009ASO,
author="Lurie, J.",
title={{A Survey of Elliptic Cohomology}},
bookTitle="Algebraic Topology: The Abel Symposium 2007",
year="2009",
publisher="Springer Berlin Heidelberg",
address="Berlin, Heidelberg",
pages="219--277",
isbn="978-3-642-01200-6",
doi="10.1007/978-3-642-01200-6_9",
url="https://doi.org/10.1007/978-3-642-01200-6_9"
}

@article{Teleman1998TheQC,
  title={{The Quantization Conjecture Revisited}},
  author={C. Teleman},
  journal={Annals of Mathematics},
  year={1998},
  volume={152},
  pages={1-43},
  url={https://api.semanticscholar.org/CorpusID:17691297}
}

@article{Behrend2003EquivariantGO,
  title={{Equivariant Gerbes over Compact Simple Lie Groups}},
  author={K. A. Behrend and P. Xu and B. Zhang},
  journal={Comptes Rendus Mathematique},
  year={2003},
  volume={336},
  pages={251-256},
  url={https://api.semanticscholar.org/CorpusID:17795851}
}

@article{Goodman1984StructureAU,
  title={{Structure and Unitary Cocycle Representations of Loop Groups and the Group of Diffeomorphisms of the Circle}},
  author={R. Goodman and N. Wallach},
  journal={{Journal f{\"u}r die reine und angewandte Mathematik (Crelles Journal)}},
  year={1984},
  pages={133 - 69},
  url={https://api.semanticscholar.org/CorpusID:120493268}
}

@article{Mickelsson1987KacMoodyGT,
  title={{Kac-Moody Groups, Topology of the Dirac Determinant Bundle, and Fermionization}},
  author={J. Mickelsson},
  journal={Communications in Mathematical Physics},
  year={1987},
  volume={110},
  pages={173-183},
  url={https://api.semanticscholar.org/CorpusID:122536611}
}

@article{etingof_central_1994,
    title = {{Central Extensions of Current Groups in Two Dimensions}},
    volume = {165},
    issn = {0010-3616, 1432-0916},
    url = {http://arxiv.org/abs/hep-th/9303047},
    doi = {10.1007/BF02099419},
    number = {3},
    urldate = {2023-12-21},
    journal = {Communications in Mathematical Physics},
    author = {P. Etingof and I. B. Frenkel},
    year = {1994},
    pages = {429--444},
}

@article{Bismut1986TheAO,
  title={{The Analysis of Elliptic Families}},
  author={J.-M. Bismut and D. S. Freed},
  journal={Communications in Mathematical Physics},
  year={1986},
  volume={107},
  pages={103-163},
  url={https://api.semanticscholar.org/CorpusID:84838651}
}

@article{Drzet1989GroupeDP,
  title={{Groupe de Picard des vari{\'e}t{\'e}s de modules de fibr{\'e}s semi-stables sur les courbes alg{\'e}briques}},
  author={J.-M. Dr{\'e}zet and M. Narasimhan},
  journal={Inventiones mathematicae},
  year={1989},
  volume={97},
  pages={53-94},
  url={https://api.semanticscholar.org/CorpusID:119759080}
}

@article{Atiyah1984DiracOC,
  title={{Dirac Operators Coupled to Vector Potentials}},
  author={M. F. Atiyah and I. M. Singer},
  journal={Proceedings of the National Academy of Sciences of the United States of America},
  year={1984},
  volume={81 8},
  pages={2597-600},
  url={https://api.semanticscholar.org/CorpusID:11756103}
}

@article{FrenkelOrbital,
  title={{Orbital Theory for Affine Lie Algebras}},
  author={I. B. Frenkel},
  journal={Inventiones mathematicae},
  year={1984},
  volume={77},
  pages={301-352},
  url={https://api.semanticscholar.org/CorpusID:121084447}
}

@article{Mathai2002ChernCI,
  title={{Chern Character in Twisted K-Theory: Equivariant and Holomorphic Cases}},
  author={V. Mathai and D. Stevenson},
  journal={Communications in Mathematical Physics},
  year={2002},
  volume={236},
  pages={161-186},
  url={https://api.semanticscholar.org/CorpusID:2419207}
}

@inbook{Grojnowski2007EllipticCD, place={Cambridge}, series={London Mathematical Society Lecture Note Series}, title={{Delocalised Equivariant Elliptic Cohomology (with an Introduction by Matthew Ando and Haynes Miller)}}, booktitle={Elliptic Cohomology: Geometry, Applications, and Higher Chromatic Analogues}, publisher={Cambridge University Press}, author={I. Grojnowski},
year={2007}, pages={114–121}, collection={London Mathematical Society Lecture Note Series}}

@book{pressley_loop_1988,
    title = {{Loop Groups}},
    isbn = {978-0-19-853561-4},
    publisher = {Clarendon Press},
    author = {A. Pressley and G. Segal},
    year = {1988},

}

@article{Murray2001HiggsFB,
  title={{Higgs Fields, Bundle Gerbes and String Structures}},
  author={M. K. Murray and D. Stevenson},
  journal={Communications in Mathematical Physics},
  year={2001},
  volume={243},
  pages={541-555},
  url={https://api.semanticscholar.org/CorpusID:18411127}
}

@article{Berwick-Evans:2021jlr,
    author = "D. Berwick-Evans and A. Tripathy",
    title = "{{A De Rham Model for Complex Analytic Equivariant Elliptic Cohomology}}",
    doi = "10.1016/j.aim.2021.107575",
    journal = "Adv. Math.",
    volume = "380",
    pages = "107575",
    year = "2021"
}

@article{Kristel2020SmoothFB,
  title={{Smooth Fock bundles, and Spinor Bundles on Loop Space}},
  author={P. Kristel and K. Waldorf},
volume = {128},
journal = {Journal of Differential Geometry},
number = {1},
publisher = {Lehigh University},
pages = {193 -- 255},
year = {2024},
doi = {10.4310/jdg/1721075262},
URL = {https://doi.org/10.4310/jdg/1721075262}
}

@article{bismut_index_1985,
    title = {{Index Theorem and Equivariant Cohomology on the Loop Space}},
    volume = {98},
    issn = {0010-3616, 1432-0916},
    url = {http://link.springer.com/10.1007/BF01220509},
    doi = {10.1007/BF01220509},
    number = {2},
    urldate = {2024-02-14},
    journal = {Communications in Mathematical Physics},
    author = {J.-M. Bismut},
    month = jun,
    year = {1985},
    pages = {213--237},
}

@article{Coquereaux1989StringSO,
  title={{String Structures on Loop Bundles}},
  author={R. Coquereaux and K. Pilch},
  journal={Communications in Mathematical Physics},
  year={1989},
  volume={120},
  pages={353-378},
  url={https://api.semanticscholar.org/CorpusID:96450837}
}

@article{Waldorf2009StringCA,
  title={{String Connections and Chern-Simons Theory}},
  author={K. Waldorf},
  journal={Transactions of the American Mathematical Society},
  year={2009},
  volume={365},
  pages={4393-4432},
  url={https://api.semanticscholar.org/CorpusID:115176784}
}

@article{Stolz2005TheSB,
  title={{The Spinor Bundle on Loop Space}},
  author={S. Stolz and P. Teichner},
  year={2005},
  url={https://api.semanticscholar.org/CorpusID:247949845}
}

@article{Freed2005LoopGA,
  title={{Loop Groups and Twisted K-Theory II}},
  author={D. S. Freed and M. J. Hopkins and C. Teleman},
  journal={Journal of the American Mathematical Society},
  year={2005},
  volume={26},
  pages={595-644},
  url={https://api.semanticscholar.org/CorpusID:106406047}
}

@article{Han2014ExoticTE,
  title={{Exotic Twisted Equivariant Cohomology of Loop Spaces, Twisted Bismut–Chern Character and T-Duality}},
  author={F. Han and V. Mathai},
  journal={Communications in Mathematical Physics},
  year={2014},
  volume={337},
  pages={127-150},
  url={https://api.semanticscholar.org/CorpusID:119598186}
}

@article{Gomi2001ConnectionsAC,
  title={{Connections and Curvings on Lifting Bundle Gerbes}},
  author={K. Gomi},
  journal={Journal of the London Mathematical Society},
  year={2001},
  pages={510–526},
  volume={67},
  url={https://api.semanticscholar.org/CorpusID:119677126}
}

@article{Hanisch2017TheFI,
  title={{The Fermionic Integral on Loop Space and the Pfaffian Line Bundle}},
  author={F. Hanisch and M. Ludewig},
  journal={Journal of Mathematical Physics},
  volume={63},
  number={12},
  pages={123502},
  year={2022},
  publisher={AIP Publishing LLC}
}

@incollection{AST_1985__131__43_0,
     author = {M. F. Atiyah},
     title = {{Circular Symmetry and Stationary-Phase Approximation}},
     booktitle = {Colloque en l'honneur de Laurent Schwartz - Volume 1},
     series = {Ast\'erisque},
     pages = {43--59},
     publisher = {Soci\'et\'e math\'ematique de France},
     year = {1985},
     url = {https://www.numdam.org/item/AST_1985__131__43_0/}
}

@misc{waldorf_spin_2012,
    title = {{Spin Structures on Loop Spaces that Characterize String Manifolds}},
    url = {http://arxiv.org/abs/1209.1731},
    doi = {10.48550/arXiv.1209.1731},
    urldate = {2025-03-15},
    publisher = {arXiv},
    author = {K. Waldorf},
    month = sep,
    year = {2012},
    note = {arXiv:1209.1731 [math]
version: 1},
}

@inproceedings{Bismut2011DuistermaatHeckmanFA,
author="J.-M. Bismut ",
title={{Duistermaat--Heckman Formulas and Index Theory}},
bookTitle="Geometric Aspects of Analysis and Mechanics: In Honor of the 65th Birthday of Hans Duistermaat",
year="2011",
publisher="Birkh{\"a}user Boston",
pages="1--55",
doi="10.1007/978-0-8176-8244-6_1",
url="https://doi.org/10.1007/978-0-8176-8244-6_1"
}

@article{Waldorf2010ALS,
  title={{A Loop Space Formulation for Geometric Lifting Problems}},
  author={K. Waldorf},
  journal={Journal of the Australian Mathematical Society},
  year={2010},
  volume={90},
  pages={129 - 144},
  url={https://api.semanticscholar.org/CorpusID:119152820}
}

@article{Brylinski1990RepresentationsOL,
  title={{Representations of Loop Groups, Dirac Operators on Loop Space, and Modular Forms}},
  author={J.-L. Brylinski},
  journal={Topology},
  year={1990},
  volume={29},
  pages={461-480},
  url={https://api.semanticscholar.org/CorpusID:120679388}
}

@article{Carey2004BundleGF,
  title={{Bundle Gerbes for Chern-Simons and Wess-Zumino-Witten Theories}},
  author={A. Carey and S. Johnson and M. K. Murray and D. Stevenson and B.-L. Wang},
  journal={Communications in Mathematical Physics},
  year={2004},
  volume={259},
  pages={577-613},
  url={https://api.semanticscholar.org/CorpusID:11481993}
}

@article{Wassermann1998OperatorAA,
  title={{Operator Algebras and Conformal Field Theory}},
  author={A. Wassermann},
  journal={Inventiones Mathematicae},
  year={1998},
  volume={133},
  pages={966-979},
  url={https://api.semanticscholar.org/CorpusID:121850146}
}

@article{Gomi2001TheFO,
  title={{The Formulation of the Chern-Simons Action for General Compact Lie Groups Using Deligne Cohomology}},
  author={K. Gomi},
  journal={Journal of Mathematical Sciences-the University of Tokyo},
  year={2001},
  volume={8},
  pages={223-242},
  url={https://api.semanticscholar.org/CorpusID:14127813}
}

@article{BerwickEvans2019SupersymmetricLM,
  title={{Supersymmetric Localization, Modularity and the Witten Genus}},
  author={D. Berwick-Evans},
volume = {126},
journal = {Journal of Differential Geometry},
number = {2},
publisher = {Lehigh University},
pages = {401 -- 430},
year = {2024},
doi = {10.4310/jdg/1712344216},
URL = {https://doi.org/10.4310/jdg/1712344216}
}

@article{Bunke2009StringSA,
  title={{String Structures and Trivialisations of a Pfaffian Line Bundle}},
  author={U. Bunke},
  journal={Communications in Mathematical Physics},
  year={2009},
  volume={307},
  pages={675-712},
  url={https://api.semanticscholar.org/CorpusID:17651213}
}

@article{Meinrenken2002TheBG,
  title={{The Basic Gerbe Over a Compact Simple Lie Group}},
  author={E. Meinrenken},
journal="Enseign. Math. (2)",
year="2003",
volume="49",
pages="307-333",
URL="https://cir.nii.ac.jp/crid/1571135650774402560"
}

@Inbook{Brylinski1993,
author="J.-L. Brylinski",
Title="Loop Spaces, Characteristic Classes and Geometric Quantization",
year="1993",
publisher={Publisher
Birkhäuser Boston, MA},
doi="10.1007/978-0-8176-4731-5_6",
url="https://doi.org/10.1007/978-0-8176-4731-5_6"
}

@article{Axelrod1991GeometricQO,
  title={{Geometric Quantization of Chern-Simons Gauge Theory}},
  author={S. Axelrod and S. Della Pietra and E. Witten},
  journal={Journal of Differential Geometry},
  year={1991},
  volume={33},
  pages={787-902},
  url={https://api.semanticscholar.org/CorpusID:119974874}
}

@article{FREED1995237,
title = {{Classical Chern-Simons Theory, 1}},
journal = {Advances in Mathematics},
volume = {113},
number = {2},
pages = {237-303},
year = {1995},
issn = {0001-8708},
doi = {https://doi.org/10.1006/aima.1995.1039},
url = {https://www.sciencedirect.com/science/article/pii/S0001870885710390},
author = {D. S. Freed}
}

@article{Liu1996MODULARFA,
  title={{Modular Forms and Topology}},
  author={K. Liu},
  journal={Contemporary Mathematics},
  volume={193},
  pages={237--262},
  year={1996},
  publisher={AMERICAN MATHEMATICAL SOCIETY}
}

@article{Liu1995OnMI,
  title={{On Modular Invariance and Rigidity Theorems}},
  author={K. Liu},
  journal={Journal of Differential Geometry},
  year={1995},
  volume={41},
  pages={343-396},
  url={https://api.semanticscholar.org/CorpusID:17190230}
}

@article{Liu1994ModularIA,
  title={{Modular Invariance and Characteristic Numbers}},
  author={K. Liu},
  journal={Communications in Mathematical Physics},
  year={1994},
  volume={174},
  pages={29-42},
  url={https://api.semanticscholar.org/CorpusID:18630}
}

@misc{Preparation,
      title={{Quantization Commutes with Reduction of Chern-Simons Gauge Theory}}, 
      author={Geyang Dai and Ruiming Liang and Yang Zhang},
      eprint={2601.09666},
      archivePrefix={arXiv},
      primaryClass={math-ph},
note = {2601.09666},
      url={https://arxiv.org/abs/2601.09666}, 
}

@article {Baranovsky1996ConjugacyCI,
    AUTHOR = {Baranovsky, V. and Ginzburg, V.},
     TITLE = {{Conjugacy Classes in Loop groups and {$G$}-bundles on Elliptic
              Curves}},
   JOURNAL = {Internat. Math. Res. Notices},
  FJOURNAL = {International Mathematics Research Notices},
      YEAR = {1996},
    NUMBER = {15},
     PAGES = {733--751},
      ISSN = {1073-7928,1687-0247},
   MRCLASS = {20G35 (14F05)},
  MRNUMBER = {1413870},
MRREVIEWER = {Vladimir\ L.\ Popov},
       DOI = {10.1155/S1073792896000463},
       URL = {https://doi.org/10.1155/S1073792896000463},
}

@article{AlvarezGaum1986ThetaFM,
  title={{Theta Functions, Modular Invariance, and Strings}},
  author={L. Alvarez-Gaum{\'e} and G. W. Moore and C. Vafa},
  journal={Communications in Mathematical Physics},
  year={1986},
  volume={106},
  pages={1-40},
  url={https://api.semanticscholar.org/CorpusID:121490383}
}

@inproceedings{Freed1987OnDL,
author = {D. S. Freed},
title = {{On Determinant Line Bundles}},
booktitle = {Mathematical Aspects of String Theory},
chapter = {},
pages = {189-238},
doi = {10.1142/9789812798411_0011},
URL = {https://www.worldscientific.com/doi/abs/10.1142/9789812798411_0011},
eprint = {https://www.worldscientific.com/doi/pdf/10.1142/9789812798411_0011},
year = "1987"
}

@article{Ray1973AnalyticTF,
  title={{Analytic Torsion for Complex Manifolds}},
  author={D. Ray and I. M. Singer},
  journal={Annals of Mathematics},
  year={1973},
  volume={98},
  pages={154-177},
  url={https://api.semanticscholar.org/CorpusID:119485746}
}

@incollection{Miller1989EllipticCW,
  author = {H. Miller},
  title = {{The Elliptic Character and the Witten Genus}},
  booktitle = {Algebraic topology (Evanston, IL, 1988)},
  editor = {M. Mahowald and S. Priddy},
  series = {Contemporary Mathematics},
  volume = {96},
  pages = {281--289},
  publisher = {American Mathematical Society},
  address = {Providence, RI},
  year = {1989},
  doi = {10.1090/conm/096},
  note = {pdf},
}

@article{Ando2000TheWG,
  title={{The Witten genus and Equivariant Elliptic Cohomology}},
  author={M. Ando and M. d. R. Basterra},
  journal={Mathematische Zeitschrift},
  year={2000},
  volume={240},
  pages={787-822},
  url={https://api.semanticscholar.org/CorpusID:6173109}
}

@article{Witten1988TheIO,
  title={{The Index of the Dirac Operator in Loop Space}},
  author={E. Witten},
  journal={Lecture Notes in Mathematics},
  year={1988},
  volume={1326},
  pages={161-181},
  url={https://api.semanticscholar.org/CorpusID:116202649}
}

@article{Quillen1985DeterminantsOC,
  title={{Determinants of Cauchy-Riemann Operators over a Riemann Surface}},
  author={D. Quillen},
  journal={Functional Analysis and Its Applications},
  year={1985},
  volume={19},
  pages={31-34},
  url={https://api.semanticscholar.org/CorpusID:122340883}
}

@article{Friedman1997PrincipalGB,
  title={{Principal G Bundles over Elliptic Curves}},
  author={R. Friedman and J. W. Morgan and E. Witten},
  journal={Mathematical Research Letters},
  year={1997},
  volume={5},
  pages={97-118},
  url={https://api.semanticscholar.org/CorpusID:14347135}
}

@article{Freed1986DeterminantsTA,
  title={{Determinants, Torsion, and Strings}},
  author={D. S. Freed},
  journal={Communications in Mathematical Physics},
  year={1986},
  volume={107},
  pages={483-513},
  url={https://api.semanticscholar.org/CorpusID:121998986}
}

@inproceedings{Hopkins2002ATM,
  author = {M. Hopkins},
  title = {{Algebraic Topology and Modular Forms}},
  booktitle = {Proceedings of the International Congress of Mathematicians},
  volume = {1},
  year = {2002}
}

@incollection{StolzTeichner2004Elliptic,
  author    = {S. Stolz and P. Teichner},
  title     = {{What is an Elliptic Object?}},
  booktitle = {Topology, Geometry and Quantum Field Theory},
  editor    = {Ulrich Tillmann},
  series    = {London Mathematical Society Lecture Note Series},
  volume    = {308},
  pages     = {247--343},
  publisher = {Cambridge University Press},
  address   = {Cambridge},
  year      = {2004}
}

\end{document}